\documentclass[11pt, reqno]{amsart}

\usepackage[margin=0.9in]{geometry}

\usepackage{amsmath,amsthm,soul,ulem,mathtools
}
\usepackage{dsfont}

\usepackage{upgreek}

\usepackage{mathrsfs}
\usepackage{color}
\usepackage{xcolor}
\usepackage{graphicx}
\usepackage{amssymb}

\usepackage{url}
\usepackage{multicol}
\usepackage{tikz}
\usepackage{amsmath}
\usepackage{fancyhdr}
\usetikzlibrary{matrix}
\usetikzlibrary{arrows}
\usepackage{subfig}
\usepackage{hyperref}
\usepackage{listings}

\usepackage{enumitem}
\usepackage{varwidth}

\allowdisplaybreaks

\definecolor{bluegreen}{rgb}{0.0, 0.3, 0.9}

\def\tofill{\vskip20pt $\cdots$ To fill in $\cdots$ \vskip20pt}

\newcommand{\comment}[1]{\vskip.3cm\fbox{%
\parbox{0.93\linewidth}{\footnotesize #1}}
\vskip.3cm}

\title[Stability of vortices in GL evolution]{Linearized dynamic stability for vortices \\
of Ginzburg-Landau evolutions}

\author[J.M. Palacios]{Jos\'e M. Palacios}
\address{University of Toronto}
\email{jose.palacios@utoronto.ca}

\author[F. Pusateri]{Fabio Pusateri}
\address{University of Toronto}
\email{fabiop@math.toronto.edu}

\def\eps{\epsilon}

\def\jxi{\langle \xi \rangle}
\def\jr{\langle r \rangle}
\def\jk{\langle k \rangle}

\def\wt{\widetilde}
\def\wtF{\wt{\mathcal{F}}}

\newcommand{\be}{\begin{equation}}
\newcommand{\ee}{\end{equation}}
\newcommand{\bp}{\begin{proof}}
\newcommand{\ep}{\end{proof}}
\newcommand{\bel}{\begin{equation}\label}
\newcommand{\eeq}{\end{equation}}
\newcommand{\bea}{\begin{eqnarray}}
\newcommand{\eea}{\end{eqnarray}}
\newcommand{\bee}{\begin{eqnarray*}}
\newcommand{\eee}{\end{eqnarray*}}
\newcommand{\ben}{\begin{enumerate}}
\newcommand{\een}{\end{enumerate}}

\newcommand{\R}{\mathbb{R}}

\newcommand{\N}{\mathbb{N}}
\newcommand{\Z}{\mathbb{Z}}

\newcommand{\C}{\mathbb{C}}

\newcommand{\supp}{\operatorname{supp}}

\newtheorem{thm}{Theorem}[section]

\newtheorem{lem}[thm]{Lemma}
\newtheorem{prop}[thm]{Proposition}

\newtheorem{rem}{Remark}[section]

\definecolor{codegreen}{rgb}{0,0.6,0}
\definecolor{codegray}{rgb}{0.5,0.5,0.5}
\definecolor{codepurple}{rgb}{0.58,0,0.82}
\definecolor{backcolour}{rgb}{0.95,0.95,0.92}

\lstdefinestyle{mystyle}{
	backgroundcolor=\color{backcolour},   
	commentstyle=\color{codegreen},
	keywordstyle=\color{magenta},
	numberstyle=\tiny\color{codegray},
	stringstyle=\color{codepurple},
	basicstyle=\footnotesize,
	breakatwhitespace=false,         
	breaklines=true,                 
	captionpos=b,                    
	keepspaces=true,                 
	numbers=left,                    
	numbersep=5pt,                  
	showspaces=false,                
	showstringspaces=false,
	showtabs=false,                  
	tabsize=2
}
\numberwithin{equation}{section}



\theoremstyle{definition}

\numberwithin{ej}{section}

\begin{document}





\renewcommand{\sectionmark}[1]{\markright{\thesection.\ #1}}
\renewcommand{\headrulewidth}{0.5pt}
\renewcommand{\footrulewidth}{0.5pt}



\begin{abstract}
We consider the problem of dynamical stability for the $n$-vortex of the Ginzburg-Landau model.
Vortices are one of the main examples of topological solitons, and their dynamic stability is
the basic assumption of the asymptotic ``particle plus field'' description of interacting vortices 
\cite{WX,bookManSut,Neu}.
In this paper we focus on co-rotational perturbations of vortices
and establish 
decay estimates for their linearized evolution
in the relativistic case. 

One of the main ingredients is a construction of the distorted Fourier basis
associated to the linearized operator at the vortex.
The general approach follows that of Krieger-Schlag-Tataru 
\cite{KST} and Krieger-Miao-Schlag 
\cite{KMS} 
and relies on the spectral analysis of 
Schr\"odinger operators with strongly singular potentials \cite{GZ}.
Since 
one of the operators appearing in the linearization
has zero energy solutions that oscillate at infinity, 
additional work is needed for our construction 
and to control the spectral measure. 
The decay estimates that we obtain are of both wave and Klein-Gordon type,
and are consistent with the general theory for $2$d Schr\"odinger operators,
including those that have an $s$-wave resonance, as in the present case, but faster decaying potentials \cite{Ya,Toprak}.


Finally, we give a new proof of the absence of unstable spectrum 
and provide an estimate on the location of embedded eigenvalues by using a suitable
Lieb-Thirring inequality due to Ekholm and Frank \cite{EkhFra}. In particular, we show that eigenvalues
must lie in the interval $(1.332,2)$, where $2$ represents the effective mass 
of one of the two scalar operators appearing in the linearization. 

\end{abstract}

\maketitle

\setcounter{tocdepth}{1}

\begin{quote}

\tableofcontents

\end{quote}

\section{Introduction}

\subsection{The GL Energy functional and vortices}
The Ginzburg-Landau energy functional in two space dimensions has the form
\begin{align}\label{GLEnergy}
\mathcal{E}(\phi) 
  = \int_{\R^2} | \nabla \phi |^2 + \frac{1}{2} \big(1-|\phi|^2)^2 \, dxdy,
\end{align}
for a complex field $\phi : (x,y) \in \R^2 \rightarrow \C$. 
In this paper we are interested in the dynamic stability
of vortex solutions (mostly) for the relativistic flow associated to \eqref{GLEnergy}.
In general, one may look at three canonical types of flows associated to \eqref{GLEnergy} (see for example \cite{bookManSut});
given the variation
$-\tfrac{\delta}{\delta \overline{\phi}} \mathcal{E}(\phi) = \Delta \phi + (1-|\phi|^2)\phi$,
one may consider the associated gradient (or parabolic, or heat) flow 
\begin{align}
\label{Heat}
\partial_t\phi & 
  = \Delta\phi + (1-|\phi|^2)\phi,
\end{align}
the relativistic (or second order) flow
\begin{align}
\label{KG}
\partial_t^2\phi-\Delta\phi = (1-|\phi|^2)\phi,
\end{align}
and the Schr\"odinger (or Gross-Pitaevskii) flow
\begin{align}
\label{Schr}
i\partial_t\phi - \Delta \phi = (1-|\phi|^2)\phi.
\end{align}
The relativistic case \eqref{KG} is the restriction to the case of a trivial connection
of the appropriate dynamics in the context of gauge field theory, 
the so-called Maxwell-Higgs (or Abelian-Higgs) equations \cite{GusVort,GusVort,bookManSut,bookJT}.

The equations \eqref{Heat}-\eqref{Schr} admit stationary solutions, called vortices
of degree (or winding number) $n$, of the form
\begin{align}\label{phin}
\phi_n(x,y) & = U_n(r)e^{in\theta}, \qquad  n  = \pm1,\pm2,...
\qquad U_n(0)=0, \quad U_n(+\infty)=1,
\end{align}
where $(r,\theta)$ are polar coordinates in the plane 
and $U_n$ is a real-valued smooth increasing function  
satisfying the ODE
\begin{align}\label{eq_vortex}
U_{rr} + \dfrac{1}{r} U_r - \dfrac{n^2}{r^2}U + (1-U^2)U = 0, \quad U(0)=0, \quad U(+\infty)=1, \quad U'(r)>0.
\end{align}
$U_n$ satisfies the following asymptotics at infinity 
\begin{align}\label{asympt_U_infty}
U_n \approx 1-\dfrac{n^2}{2r^2} +O\Big(\dfrac{1}{r^4}\Big)\quad \hbox{ as } \quad r \to \infty,
\end{align}
and around the origin
\begin{align}\label{asympt_U_0}
U_n(r)\approx a r^{|n|}\Big(1 - \dfrac{r^2}{4|n|+4} + O(r^4) \Big) \quad \hbox{ as } \quad   r \rightarrow 0.
\end{align}
With a slight abuse of notation we will sometimes refer to $U_n$ as a vortex as well.
It is important to note that the vortices 
convergence to $1$ rather slowly at infinity,
as well as the fact that they do not have finite energy;
even though $\partial_r U_n\in L^2(\R^2)$, 
the angular derivative $\tfrac{1}{r} U_n \notin L^2(\R^2)$.
This has long been a source of difficulty in the analysis of vortices
addressed by various `renormalization' strategies; 
we refer the reader to the discussions in \cite{GPS} and references therein.

\subsection{Dynamic stability, some literature and our results}
The nonlinear asymptotic stability of vortices for \eqref{Heat}-\eqref{Schr} 
and related models is, in general, a wide open problem.
Even in the case of a single vortex solution $\phi_n$ (as opposed to the 
case of multiple vortices configurations, say), 
there are many fundamental issues that need to be addressed; these include, for example,
understanding linearized stability, obtaining a precise enough description of the spectrum of the linearization
and of the distorted Fourier transform associated to it,
possible obstructions to local decay (e.g. zero-energy resonances), 
the slow convergence to $1$ of the vortices, 
and the `genuine' quadratic nonlinearity appearing in the equations for the perturbation of the vortex.
Because of all of these difficulties, very few rigorous stability results have been established for vortices.
This is in contrast to the relatively better understood case of one dimensional topological solitons (kinks);
see below for some literature on these latter.
In this respect, our paper also initiates the detailed analysis necessary for the study 
of the long-time stability of vortices in GL and related models.

In the case of the heat flow \eqref{Heat}, stability results 
were proven by Gustafson \cite{GusVort0} and Weinstein-Xin \cite{WX}.
These authors considered co-rotational perturbations of the $n$-vortex, that is, 
solutions of the form $\phi_n + \varepsilon(r) e^{in\theta}$ with suitably small $\varepsilon$.
Also building on the work of Ovchinnikov-Sigal \cite{OvSig1} on spectral stability (and instability) of the vortices,
\cite{GusVort0} proved nonlinear stability of co-rotational perturbations with small $H^1$ norm;
\cite{WX} also obtained a similar result and, in addition, 
proved nonlinear asymptotic stability with decay in time of certain $L^p$ norms of $\varepsilon$.

%
%
%
Only spectral and orbital stability results are known for the relativistic flow;
in particular, we mention the works of Gustafson-Sigal \cite{GusSigVort} 
on 
magnetic vortices, 
and Gustafson \cite{GusVort} on 
orbital stability. 
We also mention the more recent work of Gravejat-Pacherie-Smets \cite{GPS} 
who obtained orbital stability for the 
$1$-vortex under the Schr\"odinger evolution introducing an elegant new functional-analytic framework
for this problem.

\smallskip
One of our main contributions is a 
quantitative analysis 
of the linearized dynamic stability of vortices 
under co-rotational perturbations in the relativistic case \eqref{KG}
(as well as in the parabolic case),
with a proof of pointwise decay for two classes of localized data; see Theorems \ref{MT1} and \ref{MT2} 
and Proposition \ref{Heatprop}.
We also provide a new proof of spectral stability based on a supersymmetric factorization 
(see Proposition \ref{Propspec0})
and show a gap property for embedded eigenvalues (see Theorem \ref{theoLT})
with potential implications for the nonlinear problem. 



\smallskip
To describe briefly our first set of results let us fix $n=1$ for simplicity 
and consider solutions of \eqref{KG} in the form
\begin{align}\label{introcorot}
\phi(t,x,y) = e^{i\theta} (U_1(r) +  \varepsilon(t,r)) 
\end{align}
from which one obtains
\begin{align}\label{full_eq_pertrad}
\partial_t^2\varepsilon - \Delta\varepsilon 
  + \dfrac{1}{r^2}\varepsilon - \varepsilon 
  + 2U_1^2\varepsilon + U_1^2\overline{\varepsilon} = - 2 U_1\vert\varepsilon\vert^2 
   - U_1\varepsilon^2 - \vert\varepsilon\vert^2\varepsilon.
\end{align}
Restricting the analysis to co-rotational (radial) perturbations, 
and letting 
$\varepsilon=\alpha+i\beta$, $(\alpha,\beta):=\big(\mathrm{Re}\,\varepsilon,\mathrm{Im}\,\varepsilon)$,
leads to
\begin{equation}
\begin{aligned}
\label{alpha_beta_system}
\partial_t^2\left( \begin{matrix}
\alpha \\ \beta 
\end{matrix}\right)&=\left(\begin{matrix}
\Delta+1-\tfrac{1}{r^2}-3U_1^2(r) & 0 \\ 0 & \Delta+1-\tfrac{1}{r^2}-U_1^2(r)
\end{matrix}\right)\left( \begin{matrix}
\alpha \\ \beta 
\end{matrix}\right)
\\ & =: \left( \begin{matrix}
-\mathcal{L}_{1} & 0 \\ 0 & - \mathcal{L}_{2}
\end{matrix}\right)\left( \begin{matrix}
\alpha \\ \beta 
\end{matrix}\right) =: \mathcal{L}(\alpha,\beta).
\end{aligned}
\end{equation}
Theorem \ref{MT1} shows that linear solutions of \eqref{alpha_beta_system} 
with data $(P^1_c \alpha(0), \beta(0))$, where $P^1_c$ 
denotes the projection onto the continuous spectrum of $\mathcal{L}_1$,
decay pointwise in time, provided that some initial $L^1$-based norm of the data is finite. 
We also show that similar decay estimates hold provided that 
certain natural norms of the data measured in (distorted) Fourier space are finite.
The decay rates that we obtain are sharp, consistently with 
the fact that $\beta = \mathrm{Im} \, \varepsilon$ satisfies a wave type equation, while $\alpha = \mathrm{Re} \, \varepsilon$
satisfies an equation of Klein-Gordon (massive wave) type;
indeed, we obtain a $t^{-1}$ rate of decay for $\alpha$ 
and a rate of $t^{-1/2}$, with additional decay away from the light-cone $r=t$, for $\beta$. 
We also show that the same estimates hold for co-rotational perturbations of the vortex of degree $n$ under the same hypothesis.

%

\smallskip
A key ingredient in the proofs of our decay estimates
is the construction of the distorted Fourier transform, with a robust description of the basis, 
associated to the linearized operator at the vortex.
This construction is of independent interest, and should be expected to be 
a key ingredient towards a better understanding of the nonlinear long-time/asymptotic stability
in this and related models.
In fact, in recent years, techniques based on the use of the distorted Fourier transform (dFT)
have proven to be very efficient in the analysis of stability problems for dispersive PDEs,
where the study of nonlinear oscillations plays a major role.
For works in this direction we refer the reader to
\cite{LLSS,GP20,CP21,ChenNLSV,CP22,KaPu22,KGVSim,ColGer23,LiLu24},
around the stability of solitons and kinks in $1+1$ space dimensions;
see also \cite{DelortNLSV,DMKink} for an alternative approach using wave operators,
and \cite{LSch,LSch23} for the use of supersymmetric factorizations.
For general references on the dFT for regular potentials see \cite{Y10,DT79};
for singular potentials such as those appearing in our problems, see Gezstesy-Zinchenko \cite{GZ} and \cite{KST,KMS}.




\smallskip
Prior to the works on $1$d problems cited above, 
the development and use of the dFT 
associated to Schr\"odinger operators 
played a crucial role in several important stability results;
without being exhaustive, we mention the works 
of Krieger-Schlag \cite{KriSchNLS},
Krieger-Nakanishi-Schlag \cite{KNS12}, Schlag-Soffer-Staubach \cite{ScSoStI,ScSoStII}
and the most relevant works for the present paper
by Krieger-Schlag-Tataru \cite{KST} and Krieger-Miao-Schlag \cite{KMS}
in the context of stable blow-up for wave maps in $2+1$ dimensions.
The literature on general time-decay estimates for the evolutions
associated with Schr\"odinger operators is vast,
so we will not attempt to give an exhaustive list of references
and 
will only mention works that are closer to the present one.
For more complete discussions and extensive lists of references, we refer the readers
to 
\cite{GolSch,
ErdSch,JSS,
YaWO}, the recent work on the Coulomb potential \cite{BlToVeZo}, 
and the surveys \cite{SchlagSurvey,Schlagdecay} 



The general ideas that we employ for the construction of the dFT follow
\cite{KST,KMS} and rely on the spectral analysis of Schr\"odinger operators with strongly singular potentials
by Gezstesy-Zinchenko \cite{GZ}.
One of the difficulties in our case is that 
the operator $\mathcal{L}_1$ appearing in \eqref{alpha_beta_system} 
has zero energy solutions that oscillate at $r=\infty$.
This requires additional work in the construction of the distorted Fourier basis
and in the control of the associated spectral measure.

Once the distorted Fourier transform is constructed, and enough information is available 
on the generalized eigenfunctions and on the spectral measure, 
we are able to run stationary phase-type arguments
to prove $L^1-L^\infty$ decay estimates, as well as decay estimates where the 
controlling norms are $L^\infty$ and $H^k$ norms of the distorted transform of the data.
These latter sets of estimates are typically more useful in the context of nonlinear problems,
as shown in many of the works in $1$d cited above, 
as well as in the $3$d results \cite{PS,LP22}.

\subsection{Linearized relativistic evolution}
The Hamiltonian associated with \eqref{KG} is 
\begin{align*}
\mathcal{E}(\phi,\phi_t) = \int_{\R^2} \Big( \frac{1}{2}|\partial_t\phi|^2 
  + \frac{1}{2}|\nabla \phi|^2 + \frac{1}{4}\big( 1 - |\phi|^2 )^2 \Big) \, dxdy.
\end{align*}
Notice that equation \eqref{KG} can be thought of as a $2$d version of the one-dimensional $\phi^4$ model
$\phi_{tt} - \phi_{xx} = (1-\phi^2)\phi$. 
Consider a perturbation $\varepsilon=\varepsilon(t,r,\theta)$ of the vortex $U_n$, 
so that a solution of \eqref{KG} has the form
$\phi(t,r,\theta)=\big(U_n(r)+\varepsilon(t,r,\theta)\big)e^{in\theta}$.
Then
\begin{align}\label{full_eq_pert}
\partial_t^2\varepsilon-\Delta\varepsilon-\dfrac{2ni}{r^2}\partial_\theta\varepsilon+\dfrac{n^2}{r^2}\varepsilon - \varepsilon= -2U_n^2\varepsilon-U_n^2\overline{\varepsilon}- 2 U_n\vert\varepsilon\vert^2  - U_n\varepsilon^2 - \vert\varepsilon\vert^2\varepsilon.
\end{align}
Since linear operator above is 
not complex-linear, 
it is convenient to write 
\eqref{full_eq_pert} as a system for $(\alpha,\beta)$, $\varepsilon=\alpha+i\beta$. 
Specifying to the case $n=1$ and letting $U:= U_1$, we obtain 
\begin{equation}\label{system_pert_nonradial} 
\left\{
\begin{aligned}
\partial_t^2\alpha-\Delta\alpha + \dfrac{2}{r^2}\partial_\theta\beta +\dfrac{1}{r^2}\alpha-\alpha+3U^2\alpha 
  & =-U(\alpha^2+\beta^2)-2U\alpha^2-(\alpha^2+\beta^2)\alpha,
\\ 
\partial_t^2\beta-\Delta\beta - \dfrac{2}{r^2}\partial_\theta\alpha +\dfrac{1}{r^2}\beta-\beta+U^2\beta 
  & =-2 U \alpha\beta-(\alpha^2+\beta^2)\beta.
\end{aligned}
\right.
\end{equation}
At the linear level, $(\alpha,\beta)$ satisfy the system
\begin{align}\label{full_lin_Re_Im}
\begin{split}
& \partial_t^2 \left( \begin{matrix} \alpha \\ \beta \end{matrix}\right) 
  = \mathfrak{L}\left( \begin{matrix} \alpha \\ \beta \end{matrix}\right),
\quad \, \mathfrak{L} 
:=\left(\begin{matrix}
\Delta+1-\tfrac{1}{r^2}-3U^2(r) & 0 \\ 0 & \Delta+1-\tfrac{1}{r^2}-U^2(r)
\end{matrix}\right)
+ \dfrac{2}{r^2}
\left(\begin{matrix}
0 & -\partial_\theta \\ \partial_\theta  & 0 
\end{matrix}\right).
\end{split}
\end{align}
Results from Weinstein-Xin \cite{WX} guarantee that the operator $-\mathfrak{L}$ has no negative eigenvalues,
while the absolutely continuous spectrum is $[0,\infty)$ by more standard perturbation theory.
We will give a new proof of the absence of negative eigenvalues in Subsection \ref{SecSpec}.
Moreover, one has that\footnote{%
One can also write a system in terms of 
$(\varepsilon,\overline{\varepsilon})$ instead of $(\alpha,\beta)$ as above, 
in which case \eqref{full_eq_pert} 
becomes 
\begin{align*}
\partial_t^2\left(\begin{matrix}
\varepsilon \\ \overline{\varepsilon}
\end{matrix}\right) = \mathcal{M}
\left(\begin{matrix}
\varepsilon \\ \overline{\varepsilon}
\end{matrix}\right)
+
\left(\begin{matrix}
\tfrac{2i}{r^2}\partial_\theta \!\! & \!\! 0
\\ 0 \!\! & \!\! -\tfrac{2i}{r^2}\partial_\theta 
\end{matrix}\right)
\left(\begin{matrix}
\varepsilon \\ \overline{\varepsilon}
\end{matrix}\right),
\qquad 
\mathcal{M} := \left(\begin{matrix}
\Delta +1-\tfrac{1}{r^2} -2U^2& -U^2
\\ -U^2 & \Delta+1-\tfrac{1}{r^2}-2U^2 
\end{matrix}\right).
\end{align*}
The matrix $\mathcal{M}$ is the one appearing in \cite{WX}.
In this case the ($s$-wave) resonance writes as $(\varepsilon,\overline{\varepsilon})=i(U,-U)$
}
\begin{align}\label{introres}
\mathfrak{L}(0,U)^T = 0,
\qquad \mathfrak{L} \, \mathrm{Re} \big( U_re^{i\theta}, \tfrac{i}{r}Ue^{i\theta}\big)^T = 
  \mathfrak{L} \, \mathrm{Re} (-iU_re^{i\theta},\tfrac{1}{r}Ue^{i\theta})^T = 0.
\end{align}
The zero-energy (or threshold) resonance $(0,U)$ arises from the gauge invariance of the equations,
%
%
while the other two resonances 
stem from the invariance under space-translations. 
$(0,U)$ is a so-called $s$-wave resonance since it belongs to $L^\infty \smallsetminus L^p$, for all $p\in[1,\infty)$,
while the other two modes are in $L^p$ for $p>2$, and are referred to as $p$-waves;
see \cite{Toprak,ErGoGr,Ya} and references therein on how these resonances can impact linear decay.

Restricting the analysis to co-rotational perturbations leads to \eqref{alpha_beta_system};
the linear operator in \eqref{alpha_beta_system} still has the zero energy resonance $(0,U)^T$.
Other than this resonance, $\mathcal{L}_2$
only has absolutely continuous spectrum $[0,\infty)$. 
The operator $\mathcal{L}_1$ has absolutely continuous spectrum $[2,\infty)$
and, to our knowledge, its whole spectrum has not been completely identified precisely yet.
Our analysis shows that the bottom of the spectrum is a resonance.
Moreover, due to the slow decay of the potential \eqref{introH1}, $\mathcal{L}_1$ has infinitely many 
eigenvalues in $(0,2)$; see for example \cite[Ch. 4]{FrankLTBook}.
%
Our Theorem \ref{theoLT} shows that there are no eigenvalues below $\lambda_0 \approx 1.332$. 
This has an important implication for possible applications to the nonlinear problem:
the discrete modes of the operator $\mathcal{L}_1$ 
cannot interact among themselves\footnote{How discrete modes interact with the continuous spectrum, and if a suitable 
nonlinear Fermi-Golden rule holds true, is an open question.}
 at the quadratic level
since $\lambda_1\pm\lambda_2\pm\lambda_3\neq0$ whenever $\lambda_i \in (1,2]$.
For the proof of Theorem \ref{theoLT} we use a Lieb-Thirring inequality for the second moment 
of the discrete spectrum due to Ekholm and Frank \cite{EkhFra} with the optimal constant 
guaranteed by the work of Laptev-Weidl \cite{LapWei}, and some (numerical) computations.

\subsection{Main results and ideas of the proofs}
In this subsection we provide some ideas for the constructions of the
distorted Fourier transforms associated to the linearized operators $\mathcal{L}_1$ and $\mathcal{L}_2$,
and the proofs of sharp pointwise decay estimates.
%
Our first 
theorem shows the pointwise decay of solutions for $L^1$ data.

\begin{thm}[$L^1 - L^\infty$ decay]\label{MT1} 
Let $(\alpha,\beta)$ be a solution of \eqref{alpha_beta_system},
assume that $(\alpha(0),\partial_t\alpha(0)) = (P_c^1 \alpha(0), P_c^1 \partial_t\alpha(0))$,
where $P_c^1$ is the projection onto the continuous
spectrum of $\mathcal{L}_1$, and that
\begin{align}\label{MT1as}
\begin{split}
& \sum_{\ell} (1+2^{2\ell}) {\big\| P_\ell^{\mathcal{L}_1} \alpha(0) \big\|}_{L^1} 
  + (1+2^{\ell}) {\big\| P_\ell^{\mathcal{L}_1} \alpha_t(0) \big\|}_{L^1} \leq 1,
\\
& \sum_{\ell} 2^\ell (1+2^{\ell/2}) {\big\| P_\ell^{\mathcal{L}_2} \beta(0) \big\|}_{L^1} 
  + (1+2^{\ell/2}) {\big\| P_\ell^{\mathcal{L}_2} \beta_t(0) \big\|}_{L^1} \leq 1.
\end{split}
\end{align}
Here $ P_\ell^{\mathcal{L}_j}$ denote the natural Littlewood-Paley projections associated
to the operators $\mathcal{L}_j$, $j=1,2$ (see \eqref{LPF1L} and \eqref{LPF1} for the definition).
Then
\begin{align}
\label{MT1conc1}
& |\alpha (t,r)| \lesssim \frac{1}{t}, 
\qquad  |\beta (t,r)| \lesssim \frac{1}{\sqrt{t}} \frac{1}{\sqrt{|t-r|+1}}.
\end{align}
\end{thm}
We refer to Propositions \ref{propKG1dec}, respectively, Proposition \ref{propKG2dec},
for the proof of the first, respectively, the second, inequality in \eqref{MT1conc1}.

\begin{rem}[The case of the heat flow: decay and asymptotic stability]
A statement analogous to that of Theorem \ref{MT1} can be obtained for the gradient flow
by a similar (in fact, simpler) proof;  
see Proposition \ref{Heatprop}.
This gives an alternative proof of decay to that of Weinstein-Xin \cite{WX}
who used different techniques, inspired by works of Nash, Aaronson and others,
to obtain Gaussian upper bounds.
Nonlinear asymptotic stability as in \cite[Theorem 1.1, Part 1]{WX} then follows from a bootstrap argument
as in the cited work, so we will not reproduce it.
\end{rem}


\smallskip
To state our main weighted-type estimates, we need to introduce the Fourier transforms associated to the
operators $\mathcal{L}_j$, $j=1,2$. These are constructed in Sections \ref{secL2} and \ref{secL1}.
The upshot of the construction is the following: define the operators 
$\mathcal{H}_j := r^{1/2} \mathcal{L}_j r^{-1/2}$ 
and consider the two families of generalized eigenfunctions 
$\Phi_j = \Phi_j(r,\xi)$, $r\geq 0$, 
solving the two eigenvalue problems
\begin{align}\label{Ljevalue}
\mathcal{H}_1 \Phi_1 = (\xi + 2)\Phi_1 , \qquad \mathcal{H}_2 \Phi_2  = \xi \Phi_2, \quad  \xi \geq 0;
\end{align}
then, we have the Fourier representations\footnote{We use the unconventional $\sqrt{\xi}$ in the argument of $\wtF_j (g)$
here for consistency of notation, since we will mostly 
replace $\xi=k^2$ in the actual statements and proofs.}
\begin{align}\label{FTj}
\begin{split}
& P_c^1 f(r) = \frac{1}{\sqrt{r} }\int_{0}^\infty \Phi_1(r,\xi) \wtF_1 (\sqrt{r}f ) (\sqrt{\xi}) \, \rho_1(d\xi), \qquad  
\\
& g(r) = \frac{1}{\sqrt{r} }\int_{0}^\infty \Phi_2(r,\xi) \wtF_2 (\sqrt{r}g) (\sqrt{\xi}) \, \rho_2(d\xi), \qquad   
\\
& \wtF_j (h)(\sqrt{\xi}) := \int_{0}^\infty \Phi_j(r,\xi) h(r) \, dr, \qquad j=1,2.
\end{split}
\end{align}
The spectral measures $\rho_j(d\xi)$ 
are absolutely continuous with respect to the Lebesgue measures on $[0,\infty)$,
$\rho_j(d\xi) = \rho_j'(\xi)d\xi$, and we will show that $\rho_j'(\xi) \approx \jxi$.
Moreover, we will obtain convergent expansions for the generalized eigenfunctions 
in the regions $r^2\xi \gtrsim 1$ and $r^2\xi \lesssim 1$.
We can now state our main weighted-type estimates:


\begin{thm}[Decay estimates using Fourier norms]\label{MT2} 
Let $(\alpha,\beta)$ be a solution of \eqref{alpha_beta_system} 
with $(\alpha(0),\partial_t\alpha(0)) = (P_c^1 \alpha(0),P_c^1 \partial_t\alpha(0))$.
With $ \wtF_j$ defined by \eqref{Ljevalue}-\eqref{FTj}, let
\begin{align*}
F_1(k) &:= \frac{1}{a_1(k)} \wtF_1 (\sqrt{r} \alpha(0)) (k),
\qquad 
G_1(k) := \frac{1}{a_1(k)} \wtF_1 (\sqrt{r} \alpha_t(0)) (k),
\end{align*}
and
\begin{align*}
F_2(k) &:= \frac{1}{a_2(k)} \wtF_2 (\sqrt{r} \beta(0)) (k),
\qquad
G_2(k) := \frac{1}{a_2(k)} \wtF_2 (\sqrt{r} \beta_t(0)) (k),
\end{align*}
where $a_j$, $j=1,2$, is such that $\rho_j'(k^2) = (1/4\pi) |a_j(k^2)|^{-2}$ (in particular $|a_j(k^2)| \approx \jk^{-1}$).
Then, we have 
\begin{align}\label{MT2conc1}
\begin{split}
|\alpha(t,r)| & \lesssim \frac{1}{t} {\| \jk^2 F_1 \|}_{L^\infty_k}
 + \frac{1}{t^{5/4}} {\| \jk^{11/4}\partial_k F_1 \|}_{L^2_k} + \frac{1}{t^{7/4}} {\|  \jk^{17/4}\partial_k^2 F_1 \|}_{L^2_k}
 \\ 
 & + \frac{1}{t} {\| \jk G_1 \|}_{L^\infty_k}
 + \frac{1}{t^{5/4}} {\| \jk^{7/4}\partial_k G_1 \|}_{L^2_k} + \frac{1}{t^{7/4}} {\|  \jk^{13/4}\partial_k^2 G_1 \|}_{L^2_k},
\end{split}
\end{align}
and
\begin{align}\label{MT2conc2}
\begin{split}
|\beta(t,r)| & \lesssim \dfrac{1}{\sqrt{t}} \Vert k \jk^{1/2+}F_2 \Vert_{L^\infty_k}
  + \dfrac{1}{t}\Vert k\jk^{1/2}\partial_kF_2 \Vert_{L^2_k}
 + \dfrac{1}{\sqrt{t}}\Vert \jk^{1/2+}G_2 \Vert_{L^\infty_k}
 + \dfrac{1}{t}\Vert \jk^{1/2}\partial_kG_2 \Vert_{L^2_k} .
\end{split}
\end{align}
\end{thm}

We refer to Propositions \ref{WDecKG1} and \ref{WDecWa1} for the proofs of \eqref{MT2conc1}
and \eqref{MT2conc2} respectively. Lemma \ref{WDecKG4} gives bounds for the types of Fourier
norms appearing on the right-hand sides of \eqref{MT2conc1} and \eqref{MT2conc2} in terms of
physical space norms. 
We will discuss the analogous results for the $n$-vortex in Remark \ref{section1_nvortex}.

\smallskip
\begin{rem}[About the decay estimates]
The decay estimates in Theorems \ref{MT1} and \ref{MT2}
are consistent with the general theory for Schr\"odinger operators 
that have an $s$-wave resonance, as in the present case, but faster decaying potentials \cite{Ya,Toprak}.
Note also that one should not expect improved local decay because of the threshold resonance.
We also mention that estimates like those in Theorem \eqref{MT1} have been obtained
by Black-Toprak-Vergara-Zou \cite{BlToVeZo} for the Coulomb potential in $3$d.

Finally, slightly more precise estimates than \eqref{MT2conc1}-\eqref{MT2conc2},
with an explicit leading order term (rather than an upperbound) 
can be deduced from our arguments, but we do not state these here. 
Also, there is some flexibility in the weighted estimates in Theorem \ref{MT2} as one could potentially
play with the powers $\jk$ in \eqref{MT2conc1}, 
as well as sacrifice some powers of $k$ to gain factors of $|r-t|^{-1}$ in \eqref{MT2conc2}.
\end{rem}






We now give a brief description of the arguments for the construction of the distorted Fourier
transform and the proofs of Theorems \ref{MT1} and \ref{MT2}.
We focus on the Klein-Gordon type operator $\mathcal{L}_1$ since the wave-like operator $\mathcal{L}_2$ 
can be handled similarly, and it is in fact slightly simpler to deal with.
%
The first main goal is to solve the eigenvalue problem 
$\mathcal{H}_1 f = (\xi + 2)f$, for $\xi \geq 0$.
This is usually done by first finding a fundamental set of real-valued
zero-energy ($\xi=0$, here) solutions, 
and then perturbing the $L^2((0,1))$ solution (also called Weyl-solution at zero) to obtain a series representation 
for a real-valued solution $\Phi_1=\Phi_1(r,\xi)$ 
in the region $r^2\xi \lesssim 1$.
An asymptotic expansion for a (complex-valued) solution $\Psi_1=\Psi_1(r,\xi)$ 
in the complementary region $r^2\xi \gtrsim 1$ follows from a general argument,
essentially independent of the shape of the potential,
perturbing off of the approximate solution $\xi^{-1/4} e^{ir\sqrt{\xi}}$ 
(this is the so-called Weyl solution at infinity, with a conveniently chosen normalization).

Using the theory of \cite{GZ}, the spectral density $\rho_1'(\xi)$ can be calculated from the formula 
\begin{align}\label{introSM}
\rho_1'(\xi)=(1/4\pi)|a_1(\xi)|^{-2}, \qquad a_1 = (i/2) W(\Phi_1,\overline{\Psi_1}), 
\end{align}
where $W(f,g) = fg' - f'g$ is the Wronskian; see \ref{SsecST} and \ref{sec_specm_2} for more details.
The elegant idea of Krieger-Schlag-Tataru \cite{KST} and Kriger-Miao-Schlag \cite{KMS} 
is to rely on the second identity in \eqref{introSM} to obtain upper and lower bounds for $a_1$, 
hence for the spectral measure, by evaluating $W(\Phi_1,\overline{\Psi_1})$ in the region $r^2\xi \approx 1$
where asymptotics are available for both $\Phi_1$ and $\Psi_1$.

One of the main issues in our case 
comes from the behavior
\begin{align}\label{introH1}
\mathcal{H}_1 - 2 = r^{1/2} \mathcal{L}_1 r^{-1/2} - 2 \approx -\partial_r ^2 -(9/4)r^{-2}, \qquad r\rightarrow \infty,
\end{align}
%
%
having used \eqref{asympt_U_infty}, 
%
which implies that a fundamental system of (real-valued) zero-energy solutions asymptotically behaves as 
\begin{align}\label{introapps}
y_1(r) & =\sqrt{r}\cos(\sqrt{2}\log r), \quad y_2(r)=\sqrt{r}\sin(\sqrt{2}\log r).
\end{align}
We then need to pay particular attention to the expansion of $\Phi_1$ in the region $r^2\xi \lesssim 1$,
and track a subtle cancellation in the coefficients of this expansion;
see Lemma \ref{Phi1Op1_1} 
for a detailed statement.
Furthermore, we need an extension of the arguments in \cite{KST,KMS} to lower bound $a_1$
and circumvent the vanishing of $\Phi_1$.

Once the distorted Fourier transforms have been constructed, and enough information
is available on the generalized eigenfunctions and their derivatives (in both space and frequency), 
we can write a convenient representation for the solutions of \eqref{alpha_beta_system}.
If we denote by $\Phi_j(r,k^2)$ the generalized eigenfunctions solving $\mathcal{H}_1 \Phi_1 = (2+k^2) \Phi_1$
and $\mathcal{H}_2 \Phi_2 = k^2 \Phi_2$, for $k\geq 0$,
we have that $\alpha$ is a linear combination of expression of the form
\begin{align}\label{introalpha1}
\begin{split}
& 
\frac{1}{\sqrt{r}} 
  \int_0^\infty \Phi_1(r,k^2) e^{\pm it\sqrt{k^2+2}} \wt{F_\pm}(k) \, 2k \rho_1^\prime(k^2) \, dk,
\\
& \mbox{with} \quad \wt{F_{\pm}}(k) := 
  \frac{1}{2}\wt{\mathcal{F}_1}\Big( \sqrt{r}\alpha(0) \pm \frac{1}{i\sqrt{\mathcal{H}_1}} \sqrt{r}
  \alpha_t(0)\Big)(k),
\end{split}
\end{align}
or, equivalently, of the form
\begin{align}\label{introalpha2}
\begin{split}
& \int_0^\infty K_\pm(t,r,s) \alpha(0,s) s \, ds  + \cdots
\end{split}
\end{align}
where the kernel is
\begin{align}\label{introalpha2Ker}
\begin{split}
K_\pm(t,r,s) :=  \frac{1}{\sqrt{rs}} 
  \int_0^\infty \Phi_1(r,k^2) \Phi_1(s,k^2) e^{\pm it\sqrt{k^2+2}} \, k \rho_1^\prime(k^2) \, dk,
\end{split}
\end{align}
and ``$\cdots$'' are similar contributions from $\alpha_t(0)$.
Similar representations hold for $\beta$ as in \eqref{alpha_beta_system},
with the oscillatory factor $e^{\pm i t \sqrt{k^2+2}}$ replaced by $e^{\pm i t k}$.

Based on \eqref{introalpha2}-\eqref{introalpha2Ker}, and its analogue for $\beta$, in Section \ref{secdecay}
we prove Theorem \ref{MT1} by suitably bounding frequency localized versions of the kernel using 
stationary phase arguments.
Using instead \eqref{introalpha1}, and its analogue for $\beta$, in Section \ref{WDecKG_Sec}
we prove Theorem \ref{MT2}.

\smallskip
We conclude this introduction with a remark on the case of the general $n$-vortex.

\begin{rem}[The case $|n| \geq 2$]\label{section1_nvortex}
Our proof applies, up to simple modifications, to co-rotational perturbations of $U_n$,
the vortex of degree $n$ for any $|n|\geq 2$, yielding almost the same formulas
and identical decay estimates. 

Indeed, let us define the operators
\begin{align*}
\mathcal{H}_1^n &:=-\partial_r^2+\left(n^2-\tfrac{1}{4}\right)r^{-2}-(1-3U_n^2(r)),
  \qquad \mathcal{H}_2^n = -\partial_r^2+\left(n^2-\tfrac{1}{4}\right)r^{-2}-(1-U_n^2(r)),
\end{align*}
which are the analogues for $|n|\geq 2$ of the conjugated linearized operators $\mathcal{H}_1$
and $\mathcal{H}_2$ (see \eqref{notdefH_1}) for $n=1$.
Since $U_n(r)\approx 1-\tfrac{n^2}{2r^2}$ as $r\to\infty$, near infinity we can approximate 
\begin{align*}
\mathcal{H}_1^n & \approx -\partial_r^2+2 - (2n^2 + \tfrac{1}{4})r^{-2}, 
  \qquad \mathcal{H}_2^n \approx -\partial_r^2- \tfrac{1}{4}r^{-2}.
\end{align*}
Fundamental systems of solutions of 
\[
-y''(r)-(2n^2+ \tfrac{1}{4})r^{-2} y(r)=0 \qquad \hbox{ and } \qquad -y''(r)-\tfrac{1}{4}r^{-2} y(r)=0
\]
for $n\geq 2$ are given, respectively, by
$(\sqrt{r}\sin( n\sqrt{2} \ln(r)),\sqrt{r}\cos( n\sqrt{2} \ln(r)))$
and $(\sqrt{r},\sqrt{r}\ln(r))$.
Comparing with \eqref{introapps}, we see that the first set of fundamental solutions above is almost the same,
and notice that the difference between the factors $\sqrt{2}$ and $n\sqrt{2}$ does not play any role 
in our arguments and estimates in Sections \ref{secL1}-\ref{WDecKG_Sec}; 
the second set of fundamental solutions above, associated with the approximation of $\mathcal{H}_2^n$, 
is exactly the same as that of $\mathcal{H}_2$. 

Near $r=0$ 
the dominant term 
is $r^{-2}$. The approximate equation in this case is 
$-y''(r) + (n^2-\tfrac{1}{4}) r^{-2} y(r)=0,$
for both operators, and a fundamental system is given by 
$
(r^{1/2-n}, 
r^{1/2+n}).$
Thus, the Weyl solution near zero satisfies $\Phi_{1}^n(r)\approx r^{1/2+n}$. 
The Weyl solution at infinity is essentially the same as in the case of $U_1(r)$, 
and one can show that $a_{1}^n(k^2)\approx \langle k\rangle^{-n}$, 
so that the  density of the spectral measure associated to $\mathcal{H}_{1}^n$ 
behaves like 
$\langle k\rangle^{2n}$. Similar estimate hold for $\mathcal{H}_{2}^n$.

Finally, it is not difficult to verify that all our calculations in Section \ref{WDecKG_Sec}
carry over with minor modifications to the case $|n|\geq 2$. 
In particular, all additional powers of $\jk^n$ are automatically compensated (by the behavior of $\Phi_{1}^n$
near $r=0$, for example)
giving integrals and kernels of the same form as those appearing in that section. 
Thus, Theorems \ref{MT1} and \ref{MT2}, also hold true for the vortex of degree $n$. 
We will give some more details in Remarks \ref{section3_nvortex}, \ref{section4_nvortex} and \ref{section5_nvortex}.
\end{rem}

%

\subsection{Notation}\label{secnot}
We adopt the following notation, most of which are standard: First, we use 
$\langle x \rangle$ as a short-hand for $\sqrt{1+|x|^2}$. 
%
%
%
%
%
We use $a\lesssim b$ when $a \leq Cb$ for some absolute constant $C>0$ independent on $a$ and $b$;
$a \approx b$ means that $a\lesssim b$ and $b\lesssim a$.
When $a$ and $b$ are expressions depending on our main variables or parameters, the inequalities
are assumed to hold uniformly over these.
We denote by $a+$ (respectively $a-$) a number $b>a$ (respectively $b<a$) 
that can be chosen arbitrarily close to $a$. 
Also, we use standard notation for Lebesgue and Sobolev norms, such as $L^p$, $W^{s,p}$, with $H^s = W^{s,2}$.

\medskip
\noindent
{\it Cutoffs}.
We fix a smooth even cutoff function  $\varphi: \R \to [0,1]$ 
supported in $[-8/5,8/5]$ and equal to $1$ on $[-5/4,5/4]$.
For $k \in 
\Z$, we define $\varphi_k(x) := 
\varphi(2^{-k}x) - \varphi(2^{-k+1}x)$, 
so that the family $(\varphi_k)_{k \in\Z}$ forms a partition of unity,
\begin{equation*}
 \sum_{k\in\mathbb{Z}}\varphi_k(\xi)=1, \quad \xi \neq 0.
\end{equation*}
We let
\begin{align}\label{cut0}
\varphi_{I}(x) := \sum_{k \in I \cap \mathbb{Z}} \varphi_k, \quad \text{for any} \quad I \subset \R, \quad
\varphi_{\leq a}(x) := \varphi_{(-\infty,a]}(x), \quad \varphi_{> a}(x) = \varphi_{(a,\infty)}(x),
\end{align}
with similar definitions for $\varphi_{< a}$ and $\varphi_{\geq a}$.
We also define the cutoff functions
\begin{equation}\label{cut1}
\varphi_k^{(k_0)}(\xi) = 
\left\{
\begin{array}{ll}
\varphi_k(\xi) \quad & \mbox{if} \quad k> \lfloor k_0 \rfloor,
\\        
\varphi_{\leq \lfloor k_0 \rfloor }(\xi) \quad & \mbox{if} \quad k= \lfloor k_0 \rfloor,
\end{array}\right. 
\end{equation}
Here $\lfloor x \rfloor$ denotes the largest integer smaller than $x$.
Note that the index $k_0$ in \eqref{cut1} does not need to be an integer.

\medskip
\noindent
{\it Some conventions}.
For ease of reference, we recall here our notation for the linearized operators,
and the conjugated operators $\mathcal{H}_j := r^{1/2} \mathcal{L}_j r^{-1/2}$, $j=1,2$:
\begin{align}\label{notdefH_1}
\begin{split}
\mathcal{L}_1 & := -\partial_r^2-\dfrac{1}{r}\partial_r+\dfrac{1}{r^2}-(1-3U^2(r)),  
  \qquad \mathcal{H}_1 := -\partial_r^2+\dfrac{3}{4r^2}-(1-3U^2(r)),
\\ 
\mathcal{L}_2 & := -\partial_r^2-\dfrac{1}{r}\partial_r+\dfrac{1}{r^2}-(1-U^2(r)),  
  \qquad \mathcal{H}_2 := -\partial_r^2+\dfrac{3}{4r^2}-(1-U^2(r)).
\end{split}
\end{align}

\noindent
Since only the operator $\mathcal{L}_1$ needs to be projected on the continuous spectrum,
we may drop the apex in the notation $P_c^1$ and denote this projection just by $P_c$.
We are also using the notation $\wtF_1$ as in \eqref{FTj} to denote 
the Fourier transform restricted to the absolutely 
continuous spectral subspace.

\noindent
We will often use a dot symbol `$\cdot$' to distinguish the bounds on various quantities 
involved in product estimates.

\subsection*{Acknowledgements} J.M.P. and F. P. are supported in part by a start-up grant from the
University of Toronto and NSERC grants RGPIN-2018-06487. 



\medskip
\section{Linear spectral analysis for $\mathcal{L}_2$}\label{secL2}
Recall the definition of $\mathcal{L}_2$
and introduce the operator
\begin{align}\label{defH_2}
\mathcal{H}_2 := -\partial_r^2+\tfrac{3}{4r^2}-(1-U^2(r)), \qquad  \mathcal{H}_2 = r^{1/2} \mathcal{L}_2 r^{-1/2}.
\end{align}
We record the expansions
\begin{align}
\label{Uexp0}
U(r)& = a r\left(1-\dfrac{r^2}{8}+\dfrac{(8a^2+1)r^4}{192}\right)+O(r^7) \quad \hbox{for} \quad r\approx0, 
\\ 
\label{Uexpinfty}
U(r)&= 1-\dfrac{1}{2r^2}-\dfrac{9}{8r^4}+O(r^{-6}) \quad \qquad \quad \qquad  \hbox{as } \quad r\to+\infty,
\end{align}
which are slightly more precise than \eqref{asympt_U_infty} and \eqref{asympt_U_0};
in particular, $1-U^2 = r^{-2} + O(r^{-4})$ as $r \rightarrow \infty$.
Our main goal in this section is to construct the distorted Fourier Transform associated with $\mathcal{H}_2$. 

\subsection{General linear spectral theory}\label{SsecST}
We follow the general procedure in \cite{GZ,KST,KMS}.
Denote by 
\begin{align*}
\Phi_2(r,k^2) \quad \hbox{ and } \quad \Theta_2(r,k^2)
\end{align*}
a real-valued fundamental system of generalized eigenfunctions for the problem
\begin{align*}
\mathcal{H}_2 f = k^2 f, 
\end{align*}
with $\Phi_2 \in L_2((0,1))$, and let the (unique, up to normalization) Weyl solution at infinity be $\Psi_2(r,k^2)$;
this is such that $\mathcal{H}_2\Psi_2 = z \Psi_2$ with $\Psi_2(r,z) \in L^2$ close to infinity for $\Im z >0$.
In the sequel we always assume that $\Phi_2$ and $\Theta_2$ are chosen so that 
$W(\Theta_2,\Phi_2) = 1$, where $W(f,g) = fg' - f'g$.

The general theory of \cite{GZ} allows one to write the Fourier representation as 
\begin{align}\label{linfou}
f(r) = \int_{0}^\infty \Phi_2(r,k^2) \widetilde{f}(k) \, \rho_2(dk^2) \qquad   \widetilde{f}(k) = \int_{0}^\infty \Phi_2(r,k^2) f(r) \, dr,
\end{align}
where the spectral measure is 
given by 
\begin{align}\label{Fou01}
\rho_2\big((\lambda_1,\lambda_2])=\dfrac{1}{\pi}\lim_{\delta\to0^+}
  \lim_{\epsilon\to0^+}\int_{\lambda_1+\delta}^{\lambda_2+\delta}\mathrm{Im}\, m_2\big(\lambda+i\epsilon\big)d\lambda
\end{align}
where $m_2(\cdot)$, the so called Weyl-Titchmarsh $m$ function, is defined as 
\[
m_2(k^2)=\dfrac{W(\Theta_2(r,k^2),\Psi_2(r,k^2))}{W(\Psi_2(r,k^2),\Phi_2(r,k^2))}
  \qquad \hbox{so that} \qquad  C\Psi_2(r,k^2)=\Theta_2(r,k^2)+m_2(k^2)\Phi_2(r,k^2),
\]
for some $C=C(k^2)\neq0$. 
Note that the definition of $m_2$ does not depend on the normalization of the Weyl solution $\Psi_2$. 
Also note that the constant $C(k^2)$ is given by 
\[
C(k^2)=\dfrac{W(\Theta_2,\Phi_2)}{W(\Psi_2,\Phi_2)}=\dfrac{1}{W(\Psi_2,\Phi_2)},
\] 
having used $W(\Theta_2,\Phi_2)=1$. 
To calculate the spectral measure 
one can proceed as follows. 
First we connect the two Weyl solutions at zero and infinity respectively, 
namely, $\Phi_2$ and $\Psi_2$, by writing 
\begin{align}
\Phi_2(r,k^2) = 2 \Re \big( a_2(k^2) \Psi_2(r,k^2) \big),
\end{align}
for a certain function $a_2(k^2)  = W(\Phi_2,\overline{\Psi_2})/W(\Psi_2,\overline{\Psi_2})$.
Then, using the definition of $m_2(k^2)$ above, we calculate the Wronskian 
\begin{align*}
\vert C(k^2)\vert^2W\big(\Psi_2(r,k^2),\overline{\Psi_2(r,k^2)}\big)
& = W\Big(\Theta_2+m_2\Phi_2,\Theta_2+\overline{m}_2\Phi_2\Big)
\\ & = \overline{m}_2-m_2=-2i\mathrm{Im}\,m(k^2),
\end{align*}
having used the fact that $\Phi_2$ and $\Theta_2$ are both real-valued. 
Therefore, going back to the definition of $\rho_2$ we see that the density of the spectral measure is given by
\begin{align*}
\frac{d\rho_2}{dk}(k^2) = \dfrac{1}{\pi}\mathrm{Im}\,m(k^2)
  = \dfrac{W(\Psi_2,\overline{\Psi_2})}{-2i\pi \vert W(\Psi_2,\Phi_2)\vert^2}
  = \dfrac{1}{-2\pi i |a_2(k^2)|^2 W(\overline{\Psi_2},\Psi_2)}.
\end{align*}
By choosing a standard normalization so that $W(\overline{\Psi_2},\Psi_2) = 2i$ one gets
\begin{align}\label{Fou02}
\frac{d\rho_2}{dk}(k^2) = \dfrac{1}{4\pi |a_2(k^2)|^2}.
\end{align}
One of the advantages of this identity compared with the definition of $\rho_2$ in \eqref{Fou01}, 
is that calculating $a_2(k^2)$ is easier, at least at first order. 
In Proposition \ref{propSM2} we provide upper and lower bounds for $a_2(k^2)$ as well as for its derivatives.

\subsection{Region $rk \lesssim 1$}

\subsubsection{Fundamental system at zero}
We start by exhibiting a fundamental system for $\mathcal{H}_2 f = 0$. 
One advantage of this case compared with that of $\mathcal{H}_1$, which we will treat in the next section,
is that we already know an ``explicit'' solution to the equation $\mathcal{H}_2f=0$, 
which turns out to be the $L^2$ solution at zero. 
In fact, by definition $r^{1/2} \mathcal{L}_2 r^{-1/2} =\mathcal{H}_2$, and since $\mathcal{L}_2 U=0$ we get that
\begin{align}\label{Phi_0}
\mathcal{H}_2\Phi_2^{(0)} = 0, \qquad \Phi_2^{(0)}(r) := r^{1/2} U(r).
\end{align}
Using this solution we find another one, $\Theta_2^{(0)}(r)$, by imposing
$W( \Theta_2^{(0)},\Phi_2^{(0)}) = 1$, 
that is, by solving
\begin{align*}
(r^{1/2} U) \partial_r \Theta_2^{(0)} - \big(r^{1/2} U \big)' \Theta_2^{(0)} = -1.
\end{align*}
Imposing 
$\Theta_2^{(0)} (1) = 0$ we obtain
\begin{align*}
\Theta_2^{(0)} (r) := -\Phi_2^{(0)}(r) \int_1^r \frac{1}{(\Phi_2^{(0)}(s))^2} \, ds =- r^{1/2} U(r) \int_1^r \frac{1}{s \, U^2(s)} \, ds.
\end{align*}
Note that we have 
\[
(\Phi_2^{(0)}(1),(\Phi_2^{(0)})'(1)) = (U(1), U'(1) + \frac{1}{2}U(1) ) 
\quad \hbox{ and } \quad (\Theta_2^{(0)}(1),(\Theta_2^{(0)})'(1)) = (0, -1/U(1)).
\] 
For later computations, it will be helpful to keep in mind that 
\begin{align}\label{Phi0exp}
\Phi_2^{(0)} (r) = \left\{ 
\begin{array}{ll}
a r^{3/2}+O(r^{7/2}), &  r \lesssim 1, 
 \\
r^{1/2}+O(r^{-3/2}), &  r \gtrsim 1,
\end{array} 
\right.
\end{align}
having used \eqref{Uexp0} and \eqref{Uexpinfty}, as well as
\begin{align}\label{Theta0exp}
\Theta_2^{(0)} (r) = \left\{ 
\begin{array}{ll}
\dfrac{1}{2a} r^{-1/2}+O\big(r^{3/2}\log r\big), &  r \lesssim 1, 
 \\
- \sqrt{r} \Big(\log r+\eta_0\Big)+O(r^{-3/2}\log r),  &  r \gtrsim 1,
\end{array} 
\right.
\end{align}
where the $\log r$ term in the first asymptotic in \eqref{Theta0exp} for $r\ll1$ can 
be seen using \eqref{Uexp0} to expand
\[
\frac{1}{sU^2(s)} = 
  \dfrac{1}{s(as-\frac{a}{8}s^3 + O(s^5))^2} =\dfrac{1}{a^2s^3} + \dfrac{1}{4a^2s} + O(s), \qquad s \approx 0,
\]
whereas $\eta_0$ in the second asymptotic is 
\[
\eta_0:=\int_1^\infty\dfrac{(1-U^2(s))}{sU^2(s)}ds.
\]

\smallskip
\subsubsection{Asymptotic expansion of $\Phi_2(r,k^2)$}
Using the fundamental system of solutions for $k^2=0$ we now construct a power series representation of $\Phi_2(r,k^2)$ 
and $\Theta_2(r,k^2)$.  

\begin{lem}\label{lemFou1}
For all $r>0$, $k\geq 0$ we have the expansion
\begin{align}\label{Fou11}
\Phi_2(r,k^2) = \Phi_2^{(0)}(r) + \sqrt{r} \sum_{j\geq 1} (kr)^{2j} \Phi_{2,j}(r), \qquad \Phi_2^{(0)}(r) = \sqrt{r}U(r),
\end{align}
which converges absolutely. The expansion converges uniformly if $kr \leq 1$, and
$\Phi_{2,j}$ are smooth and satisfy, for some absolute $C>0$, for all $j\geq 1$ 
\begin{align}\label{Fou12}
\begin{split}
& | \Phi_{2,j}(u) | \leq \frac{C^j}{j!(j+1)!} u , \quad u \lesssim 1,
\\
& | \Phi_{2,j}(u) | \leq \frac{C^j}{((j-1)!)^2}  
	,  \hspace{0.05cm} \quad u \gtrsim 1.
\end{split}
\end{align}
Moreover 
\begin{align}\label{Fou13}
\Phi_{2,1}(u) = \left\{ 
\begin{array}{ll}
\dfrac{a}{8} r + O(r^3), &  r \lesssim  1,
\\ 
\\
\dfrac{1}{2} + O(r^{-2}\log^2r), &  r \gtrsim 1.
\end{array} 
\right.
\end{align}
In \eqref{Fou12} and \eqref{Fou13} we also have consistent symbol-type bounds for the derivatives, that is,
\begin{align}\label{Fou14}
\begin{split}
& | (u\partial_u)^m\Phi_{2,j}(u) | \leq \frac{C^j}{j!(j+1)!}u , \qquad u \lesssim 1,
\\
& | (u\partial_u)^m\Phi_{2,j}(u) | \leq \frac{C^j}{((j-1)!)^2}  
	,  \hspace{0.9cm}  u \gtrsim 1.
\end{split}
\end{align}
\end{lem}

\begin{proof}
In order to solve $\mathcal{H}_2\Phi_2(r,k^2) = k^2 \Phi_2(r,k^2)$ we begin by making the ansatz
\begin{align}
\Phi_2(r,k^2) = \sqrt{r} \sum_{j\geq 0} k^{2j} f_j(r), \qquad f_0(r) := U(r).
\end{align}
Once we show the convergence of such a power series, our ansatz shall hold and be proven. 
Substituting 
into the equation leads to the recurrence equation
\begin{align}\label{Fou1rec}
\mathcal{H}_2 \big(\sqrt{r} f_j(r)\big) = \sqrt{r} f_{j-1}(r), \quad j\geq 0, \quad f_{-1}(r) = 0.
\end{align}
Comparing with \eqref{Fou11} gives us $f_j(r) = r^{2j} \Phi_{2,j}(r)$. 
From \eqref{Fou1rec} and the fundamental solution (Green's function) for $\mathcal{H}_2$ given by 
\[
\Theta_2^{(0)}(s)\Phi_2^{(0)}(r)-\Theta_2^{(0)}(r)\Phi_2^{(0)}(s)  \quad \hbox{ for } \quad r>s,
\]
we get the recursion formula
\begin{align*}
\sqrt{r} f_j(r) = \int_0^r \Big( \Theta_2^{(0)}(s)\Phi_2^{(0)}(r)-\Theta_2^{(0)}(r)\Phi_2^{(0)}(s)\Big) \sqrt{s} f_{j-1}(s)\, ds.
\end{align*}
Note that this can be conveniently rewritten as 
\begin{align}\label{Fou1pr1}
\begin{split}
& f_j(r) 
= \int_0^r K_2(r,s) f_{j-1}(s)\, ds,
\qquad K_2(r,s) := U(r) \, s U(s) \Big( \int_s^r \frac{1}{t U^2(t)} \, dt \Big).
\end{split}
\end{align}
We immediately see from the last two identities that the kernel is non-negative and therefore $f_j(r) \geq 0$ for all $j$.
Moreover, using \eqref{Uexp0} and \eqref{Uexpinfty}, we see that 
\begin{align}\label{Fou1Kasy}
K_2(r,s)
= \left\{ 
\begin{array}{ll}
\dfrac{1}{2} r s^2 \cdot \big( s^{-2} - r^{-2} \big)+\dfrac{1}{4}rs^2\log\tfrac{r}{s} +r^3O(1-\tfrac{s^4}{r^4}) & s < r \ll 1,
\\
s \log\big(\tfrac{r}{s}\big)-\dfrac{1}{2r^2s}(r^2+s^2)\log\big(\tfrac{r}{s}\big)+O\big(\tfrac{1}{r}+\tfrac{1}{s}\big)    &  r > s \gg 1.
\end{array} 
\right.
\end{align}
\medskip
\noindent
{\it Proof of \eqref{Fou13}}.
Let us analyze $f_1$. From \eqref{Fou1pr1}, when $r \lesssim 1$, we can use \eqref{Uexp0} to see that
\begin{align*}
f_1(r) & = U(r) \int_0^r s U^2(s) \Big( \int_s^r \frac{1}{t U^2(t)} \, dt \Big) \, ds
\\ & =\int_0^r\Big(\dfrac{1}{2} r s^2 \cdot \big( s^{-2} - r^{-2} \big)+\dfrac{1}{4}rs^2\log\tfrac{r}{s} +O(r^3+s^3)\Big) a(s+O(s^3))ds  
  = \frac{a}{8} r^3 + O(r^5)
\end{align*}
where we have used the fact that \[
\int rs^3\log\big(\tfrac rs\big)ds=r\left(\dfrac{s^4}{16}-\dfrac{s^4}{4}\log \left(\dfrac{r}{s}\right)\right)+C,
\]
and hence this term becomes $O(r^5)$ after evaluation at $s=r$. Since $f_1(r) = r^{2} \Phi_{2,1}(r)$ the asymptotic formula for $r\lesssim 1$ in \eqref{Fou13} follows.

To prove the asymptotic formula for $r \gtrsim 1$ 
we use the expansion \eqref{Uexpinfty} to see that, for $r_0 \gtrsim 1$,
\begin{align}\label{canc_log_f1}
\int_{r_0}^r s U^2(s) \Big( \int_s^r \frac{1}{t U^2(t)} \, dt \Big) \, ds\nonumber
 & = \int_{r_0}^r s (1 + O(s^{-2})) \Big( \int_s^r t ^{-1}+ O(t^{-3}) \, dt \Big) \, ds 
 \\
 & = \int_{r_0}^r s\Big(1+O(s^{-2})\Big) \Big(\log \big(\tfrac{r}{s}\big)+O(r^{-2}-s^{-2})\Big) \, ds 
\\ &= \frac{r^2}{4} + O(\log^2 r),  \nonumber
\end{align}
where we have used the trivial calculation \[
\int s\log\big(\tfrac{r}{s}\big)ds=\dfrac{s^2}{4}+\dfrac{s^2}{2}\log\big(\tfrac{r}{s}\big)+C.
\]
and the $O(\log^2 r)$ term comes from integrating $s\log(\tfrac{r}{s})O(s^{-2})$. Since the integration from $0$ to $r_0$ gives a lower order contribution, it follows that 
$f_1(r)  = r^2/4 
	+ O(\log^2 r),$ for $r \gtrsim 1.$
In fact, regarding the contribution for $s\in(0,r_0)$ we have 
\begin{align*}
\int_0^{r_0}s U^2(s) \Big( \int_s^r \frac{1}{t U^2(t)} \, dt \Big) \, ds &\lesssim  O(1)+\int_1^r\dfrac{1}{tU^2(t)}dt = O(\log r).
\end{align*}
Since $f_1(r) = r^{2} \Phi_{2,1}(r)$ it follows that 
$\Phi_{2,1}(r) = 1/4 + O\big(r^{-2}\log^2r\big)$ which is the claimed 
\eqref{Fou13}.

\medskip
\noindent
{\it Proof of \eqref{Fou12}}.
To obtain the upperbound \eqref{Fou12} on $\Phi_{2,j}(r) = r^{-2j} f_j(r)$, $j\geq 2$, we proceed by induction. We assume that $f_k(r) \leq \frac{C^k}{k!(k+1)!} r^{2k+1}$ for all $k \leq j-1$, with a constant $C>0$ depending only on implicit constant in $r\lesssim1$, say on $r_*$ with $r\leq r_*$.  In fact, from similar computations to those used above for $f_1$ we obtain that
\begin{align*}
f_{j}(r) & = U(r) \int_0^r s U(s) \Big( \int_s^r \frac{1}{t U^2(t)} \, dt \Big) f_{j-1}(s) \, ds
  \\
  & \leq a(r + O(r^3)) \int_0^r a (s^2 + O(s^4)) \Big( \int_s^r \frac{1}{a^2 t^3} (1 + O(t^2)) \, dt \Big) 
    \frac{C^{j-1}}{(j-1)!j!} s^{2j-1} \, ds
  \\ & \leq \frac{C^{j-1}}{(j-1)!j!} \Big(r+O(r^3)\Big) \int_0^r \Big(s^{2j+1}+O(s^{2j+3})\Big) \Big(\dfrac{1}{2s^2} - \dfrac{1}{2r^2}+O\big(\log(\tfrac{r}{s})\big)\Big) \, ds 
  \\ & \leq \dfrac{C^{j-1}}{(j-1)!j!}\Big(r+O(r^3)\Big)\left(\dfrac{r^{2j}}{4j(j+1)}+O\left(\dfrac{r^{2j+2}}{4(j+1)^2}\right)\right)
  \\ & \leq \frac{C^j}{j!(j+1)!} r^{2j+1}, 
\end{align*}
for $C>1$ sufficiently large and $r\lesssim 1$. Observe that the logarithmic term satisfies \[
\int_0^r s^{2j+1}\log\big(\tfrac{r}{s}\big)ds=\dfrac{r^{2j+2}}{4(j+1)^2}.
\]

When $r \gtrsim 1$, we make the inductive hypothesis $f_k(r) \leq \frac{C^k}{((k-1)!)^2} r^{2k}$, 
for $k\leq j-1$, $j\geq 2$, with some $C$ depending only on the expansion \eqref{Uexpinfty}. Then, estimating as before yields
\begin{align*}
& U(r) \int_{r_0}^r s U(s) \Big( \int_s^r \frac{1}{t U^2(t)} \, dt \Big)  f_{j-1}(s) \, ds
 \\
 & \qquad  \leq \frac{C^{j-1}}{((j-2)!)^2}\big(1+O(\tfrac{1}{r^2})\big) \int_{r_0}^r \big(s+O(\tfrac{1}{s})\big) \Big(\log \big(\tfrac{r}{s}\big) +O\big(\tfrac{1}{r^2}-\tfrac{1}{s^2}\big)\Big) s^{2j-2}  \, ds 
 \\ & \qquad \leq \dfrac{C^{j-1}}{((j-2)!)^2}\big(1+O(\tfrac{1}{r^2})\big)\left(\dfrac{r^{2j}}{4j^2}+O\left(\dfrac{r^{2j-2}}{(j-1)^2}\right)\right)\leq \frac{C^j}{((j-1)!)^2} r^{2j}.
\end{align*}
This proves the induction for $f_j$ and gives the upperbound for $\Phi_{2,j}(u)$ with $u\gtrsim 1$ in \eqref{Fou12}.

\medskip
\noindent
{\it Proof of \eqref{Fou14}}.
The symbol-type bounds follow from similar arguments to those above, using an inductive procedure,
 taking advantage of the identity 
 \[
\partial_r\big(\sqrt{r}f_j(r)\big)=\int_0^r \Big( \Theta_2^{(0)}(s)\partial_r\Phi_2^{(0)}(r)
	- \partial_r\Theta_2^{(0)}(r)\Phi_2^{(0)}(s)\Big) \sqrt{s} f_{j-1}(s)\, ds,
\]
and using the natural bounds that hold for the derivatives $\partial_r\Phi_2^{(0)}(r)$ and $\partial_r\Theta_2^{(0)}(r)$. 
Note that after solving for $\partial_rf_j(r)$ we obtain 
\begin{align*}
&  \partial_r f_j(r)=-\dfrac{1}{2r}f_j(r) + \int_0^r \widetilde{K}(r,s) f_{j-1}(s) \, ds,
\\
& \widetilde{K}(r,s) := \dfrac{s^{1/2}}{r^{1/2}}\left(\dfrac{\Phi_2^{(0)}(s)}{\Phi_2^{(0)}(r)}+\partial_r\Phi_2^{(0)}(r) \Phi_2^{(0)}(s)\int_s^r\dfrac{dt}{\Phi_2^{(0)}(t)^2}\right), \qquad r>s>0.
\end{align*}
Thus, in the case $j=1$, for instance, we have the formula
\[
r \partial_rf_1(r)= - \dfrac{1}{2}f_1(r) + \int_0^r\Big(\dfrac{sU(s)}{U(r)}+\big(\sqrt{r}U'(r)+\tfrac{1}{2\sqrt{r}}U(r)\big)\sqrt{s}U(s)\int_s^r\dfrac{dt}{tU^2(t)}\Big)U(s)ds.
\]
In the previous line the leading order term is $-(1/2)f_1(r)$ when $r\lesssim 1$. Similarly, when $r\gtrsim1$, due to $U(r)=1+O(r^{-2})$ we obtain the claimed bounds on $r\partial_rf_1$ proceeding in the same fashion as before. 
Note that for the second derivative it is enough to use the equation \eqref{Fou1rec} satisfied by $f_j$ to obtain the desired bounds.
The same observation applies to all higher order derivatives.
\end{proof}

\begin{rem}
It is worth pointing out that in the calculation \eqref{canc_log_f1}, 
the potentially leading order term of the form $r^2\log r$ cancels out; 
if such a term were present, it would have had non-trivial consequences 
on the behavior of the spectral measure and impact all the decay estimates. 
\end{rem}

\subsection{Region $rk \gtrsim 1$}
In the next lemma we give asymptotics for the Weyl solution $\Psi_2(r,k^2)$ when $rk \gtrsim 1$. This result is very much analogous to that of Proposition 5.6 in  the work of Krieger-Schlag-Tataru \cite{KST},
so we will skip most of the details of the proof and refer the reader to that reference.

\begin{lem}\label{lemWeyl}
For all $r>0$, $k\geq 0$ with $k r \gtrsim 1$, a Weyl-Titchmarsh function of $\mathcal{H}_2$ at infinity is given by
\begin{align}\label{Fou21}
\Psi_{2} 
  (r,k^2) = k^{-1/2}e^{ikr}\sigma_{2,\infty}\left(kr,r\right)\qquad kr\gtrsim 1,
\end{align}
where $\sigma_{2,\infty}(q,r)$ is smooth in $q\gtrsim1$, and admits the asymptotic power series approximation 
\begin{align*}
\sigma_{2,\infty}(q,r)\approx \sum_{j=0}^\infty q^{-j} \Psi_{2,j}(r), \qquad \Psi_{2,0}(r)=1, 
  \qquad \Psi_{2,1}(r) = -\dfrac{i}{8}+O(r^{-3}).
\end{align*}
as $r\to+\infty$, in the sense that \begin{align}\label{Fou22}
\sup_{r>0}\left\vert (r\partial_r)^\alpha (q \partial_q)^{\beta}
  \Big(\sigma_{2,\infty}(q,r)-\sum_{j=0}^{j_0}q^{-j}\Psi_{2,j}(r)\Big) \right\vert \lesssim_{j_0,\alpha,\beta} q^{-j_0-1}. 
\end{align}
Moreover, for all $j\geq1$, the coefficient functions $\Psi_{2,j}(r)$ are zero-order symbols, that is, 
for all $m\in\N$
\begin{align}\label{Fou23}
\sup_{r>0}\big\vert (r\partial_r)^m\Psi_{2,j}(r)\big\vert \lesssim_{j,m} 1. 
\end{align}
\end{lem}


\begin{proof}
The proof is standard and follows a general procedure, see for example \cite{KST,KMS}. 
We claim that the Weyl solution of 
\[
-y''(r)+\dfrac{3}{4r^2}y(r)-(1-U^2(r))y(r)=k^2y(r)
\] 
for $kr\gtrsim1$, can be written as 
\[
\Psi_{2} 
  (r,k^2)= k^{-1/2}e^{ikr}\sigma_{2,\infty}(kr,r).
\]
By inverting the above relation we see that for $\Psi_{2}
$ 
to satisfy the differential equation,  $\sigma_{2,\infty}(kr,r)$ must satisfy the conjugated equation 
\[
\Big(-\partial_r^2-2ik\partial_r+\dfrac{3}{4r^2}-(1-U^2(r))\Big)\sigma_{2,\infty}(kr,r)=0.
\]
To solve this equation we make the following power series ansatz \[
\sigma_{2,\infty}(kr,r)=\sum_{j=0}^\infty k^{-j}f_j(r), \qquad f_j(r) = r^{-j} \Psi_{2,j}(r).
\]
Once we show the convergence of such a power series, our ansatz shall hold and be proven. 
Substituting the ansatz into the conjugated equation we obtain that, at least formally, 
the family of $f_j(r)$ must satisfy the recursive equation 
\begin{align*}
2if_{j+1}'=-f_{j}''+\dfrac{3}{4r^2}f_j-(1-U^2)f_j, \qquad f_0(r)=1.
\end{align*}
Integrating the last last equation from $r$ to $\infty$, assuming $f_j(r)\to0$ as $r\to\infty$ for all $j\geq 1$, we get
\begin{align}\label{lemWeylprrec}
f_{j}(r)=\dfrac{i}{2}f'_{j-1}(r)+\dfrac{i}{2}\int_r^\infty \Big(\tfrac{3}{4s^2}f_{j-1}(s)-\big(1-U^2(s)\big)f_{j-1}(s)\Big)ds.
\end{align}
Observe that, in particular, 
\begin{align}\label{Fou24}
f_1(r) & = \dfrac{3 i}{8r}-\dfrac{i}{2}\int_r^\infty \big(1-U^2(s)\big)ds =-\dfrac{i}{8r}+O(r^{-3}),
\end{align}
where we have used the fact that $1-U^2(r)= \tfrac{1}{r^2}+O(r^{-4})$ for $r\gtrsim 1$;
this gives the formula for $\Psi_{2,1}$. 
Based on \eqref{lemWeylprrec} one can then run the exact same proof of the one of Proposition 5.6 of \cite{KST}
to obtain \eqref{Fou22}-\eqref{Fou23}; 
indeed, that proof does not depend on the specific form of the potential, and it suffices to rely on the fact that 
\begin{align}\label{der1-U^2}
\partial_r^m(1-U^2(r))=(-1)^m(m+1)!r^{-m-2}+O((m+3)!r^{-m-4}),  \quad r \gtrsim 1,
\end{align}
which can be shown directly from \eqref{eq_vortex}. 
\end{proof}

\subsection{The spectral measure}\label{sec_specm_2}
Following the general theory of \cite{GZ} and some of the arguments in \cite{KST,KMS} we  
now derive asymptotics for the spectral measure associated to $\mathcal{H}_2$. 
Recall that
\begin{align}\label{sec_specm_2W}
W\big(\Theta_2(r,k^2),\Phi_2(r,k^2)\big)=1 
	\qquad \hbox{ and } \qquad W\big(\Psi_2(r,k^2),\overline{\Psi_2(r,k^2)}\big)=-2i,
\end{align}
where the first identity is by our definition, while the second follows from Lemma \ref{lemWeyl} at $r=\infty$.

\begin{prop}\label{propSM2}
The Weyl-Titchmarsh solution of $\mathcal{H}_2$ at zero, $\Phi_2$,
and the Weyl solution 
at infinity, $\Psi_2$, are related by 
\begin{align}\label{propSM0}
\Phi_2(r,k^2) = 2\Re (a_2(k^2) \Psi_2(r,k^2)),
\end{align}
where $a_2(k^2)$ is smooth, always non-zero, and satisfies
\begin{align}\label{propSMa}
a_2(k^2) \approx \jk^{-1},
\end{align}
with consistent symbol-type bounds on the derivatives 
\begin{align}\label{propSMa_dk}
\vert (k\partial_k)^na_2(k^2) \vert \approx_n \jk^{-1}. 
\end{align}
In particular, the spectral measure associated to $\mathcal{H}_2$ is given by
\begin{align}\label{propSMrho}
\frac{d\rho_2}{dk}(k^2) = \frac{1}{4\pi |a(k^2)|^2} \approx \langle k^2\rangle.
\end{align}
\end{prop}


\begin{proof}
The existence of $a_2$ giving the relation \eqref{propSM0} 
follows from the fact that $(\Psi_2,\overline{\Psi_2})$ is a fundamental system for $\mathcal{H}_2$
and that $\Phi_2$ is real-valued, and we also have
\begin{align}\label{SMW_a2}
a_2(k^2) := \dfrac{W(\Phi_2(\cdot,k^2),\overline{\Psi_2(\cdot,k^2)})}{W(\Psi_2(\cdot,k^2),\overline{\Psi_2(\cdot,k^2)})}
=\dfrac{i}{2}W(\Phi_2(\cdot,k^2),\overline{\Psi_2(\cdot,k^2)}),
\end{align}

In what follows, we will evaluate the Wronskian in \eqref{SMW_a2} 
in the region where we have asymptotic expansions for both $\Phi_2$ and $\Psi_2$, that is, 
when $rk=c \approx 1$ with $c$ small but fixed. 
In fact, by Lemma \ref{lemWeyl} we have that, for $rk = c \approx 1$,
\begin{align}\label{SMpr1}
|\Psi_2(ck^{-1},k^2)| \approx k^{-1/2},  
\qquad |\partial_r \Psi_2(ck^{-1},k^2)| \approx k^{1/2}.  
\end{align}

To estimate the function $\Phi_2$ 
we use the expansion \eqref{Fou11} and see that, for the same $c \approx 1$, 
\begin{align}\label{SMPhi}
\Phi_2(ck^{-1},k^2) & = \Phi_2^{(0)}(ck^{-1}) + \sqrt{c}k^{-1/2}  \sum_{j\geq 1} c^{2j} \Phi_{2,j}(ck^{-1}).
\end{align}
When $k \gtrsim 1$, using the asymptotics for small $r$ in \eqref{Phi0exp} 
and \eqref{Fou12}-\eqref{Fou13}, we obtain that
\begin{align}\label{SMpr2a}
\Phi_2(ck^{-1},k^2) & \lesssim k^{-3/2}, \quad k \gtrsim 1.
\end{align}
Instead, when $k \lesssim 1$, we use the asymptotics in \eqref{Phi0exp} for large $r$, as well as  the asymptotics for $\Phi_{2,j}$ in \eqref{Fou12}-\eqref{Fou13}, from which it follows that \begin{align}\label{SMprb}
\Phi_2(ck^{-1},k^2) & \lesssim k^{-1/2} 
	, \quad k \lesssim 1.
\end{align}
Again thanks to Lemma \ref{lemFou1} we have consistent bounds for the derivatives, namely,
\begin{align}\label{SMpr2der}
\partial_r \Phi_2(ck^{-1},k^2) \lesssim \left\{ 
\begin{array}{ll}
k^{1/2}, & k \lesssim 1, 
\\ k^{-1/2}, & k \gtrsim 1.
\end{array} 
\right.
\end{align}
Therefore, using \eqref{SMpr1}-\eqref{SMpr2der} we see that 
\begin{align}\label{SMpr5}
|a_2(k^2)| = \frac{1}{2}\big| \Phi_2(\cdot,k^2) \partial_r \overline{\Psi_2(\cdot,k^2)} 
  - \partial_r \Phi_2(\cdot,k^2) \overline{\Psi_2(\cdot,k^2)} \big| 
   \lesssim \left\{ 
\begin{array}{ll}
1, &  k \lesssim 1, 
\\ 
k^{-1},  & k \gtrsim 1.
\end{array} 
\right.
\end{align}

To obtain matching lower bounds we first observe, as in \cite{KMS}, that \eqref{propSM0} gives,
for all $r>0$,
\begin{align}\label{SMpr10}
|a_2(k^2)| \geq \frac{|\Phi_2(r,k^2)|}{2|\Psi_2(r,k^2)|} \gtrsim k^{1/2} |\Phi_2(r,k^2)|,
\end{align}
having used \eqref{SMpr1}, and hence it is enough to lower bound $\Phi_2$. To this end, we use again the expansion \eqref{SMPhi}.
First, for $k \gtrsim 1$ and $r = ck^{-1}$ with a suitable $c<1$ fixed and small,
we use \eqref{SMPhi} followed by \eqref{Phi0exp} and the first bound in \eqref{Fou12}, to see that 
\begin{align*}
|\Phi_2(ck^{-1},k^2)| & = \Big| \Phi_{2}^{(0)}(ck^{-1}) 
  + \sqrt{c}k^{-1/2} \sum_{j\geq 1} c^{2j} \Phi_{2,j}(ck^{-1}) \Big|
\\ & \geq \frac{a}{10} c^{3/2} k^{-3/2} - \sqrt{c} k^{-1/2} \sum_{j\geq 1} c^{2j} \frac{C^j}{j!(j+1)!} c k^{-1}
  \\ & \geq \frac{a}{10} c^{3/2} k^{-3/2}
  - c^{3/2} k^{-3/2} \big( e^{c^2C} -1 \big) \gtrsim k^{-3/2},
\end{align*}
provided $c$ is small enough depending on $a$ and $C$.
This gives $|a_2(k^2)| \gtrsim k^{-1}$ when $k \gtrsim 1$ consistently with \eqref{SMpr5}.


When $k \lesssim 1$ and $r = ck^{-1} \gtrsim 1$, we use the expansion \eqref{SMPhi} again, 
followed by the asymptotics \eqref{Phi0exp} for large $r$ and the second bound in \eqref{Fou12}, from where we get that
\begin{align}\label{SMpr11}
\begin{split}
|\Phi_2(r,k^2)| \geq \Phi_2^{(0)}(r) - \sqrt{r} \sum_{j\geq 1} c^{2j} \Phi_{2,j}(r)
  &\geq  \frac{1}{2} \sqrt{r} - \sqrt{r} \sum_{j\geq 1} c^{2j} \frac{C^j}{j!(j+1)!}
  \\ 
  & \geq \dfrac{1}{2}\sqrt{r}-\dfrac{1}{2}\big(e^{c^2C}-1\big)\sqrt{r}\gtrsim \sqrt{r},
\end{split}
\end{align}
provided $c<1$ is small enough. Expressed in terms of $k$, the lower bound above is $|\Phi_2(r,k^2)| \gtrsim k^{-1/2}$,
for $k \lesssim 1$. Inserting this in \eqref{SMpr10} 
we obtain $|a_2(k^2)| \gtrsim 1$ which matches the upperbound \eqref{SMpr5}.
This concludes the proof of the asymptotics \eqref{propSMa} for $a_2$.

To show that similar bounds hold for the derivatives we use Lemma \ref{lemWeyl} to see that both $\Psi_2(ck^{-1},k^2)$ and $(r\partial_r\Psi_2)(ck^{-1},k^2)$ can be express as $k^{-1/2}f(k^{-1})$ for some $f(\cdot)$ satisfying the symbol type bounds \[
\vert (r\partial_r)^nf(r)\vert \approx_n 1.
\]
On the other hand, from Lemma \ref{lemFou1} we see that both $\Phi_2(ck^{-1},k^2)$ and $(r\partial_r\Phi_2)(ck^{-1},k^2)$ can be written as $k^{-1/2}h(k^{-1})$, for some $h(\cdot)$ satisfying \[
\hbox{for all } \, r\gtrsim1, \quad \vert (r\partial_r)^nh(r)\vert \approx_n 1, \qquad \ \hbox{ and } \ \qquad \hbox{for } \, r\lesssim 1, \quad \vert (r\partial_r)^nh(r)\vert \approx_n r .
\]
Therefore, from \eqref{SMW_a2} and the above analysis we conclude 
that $a_2(k^2)$ can be written as a linear combination of functions of the form $f(k^{-1})h(k^{-1})$, 
from where we conclude \eqref{propSMa_dk}.

To conclude, we recall that the first identity in \eqref{propSMrho} 
follows from the general theory of \cite{GZ,KMS} (cf. \eqref{Fou01}-\eqref{Fou02}). 
The proof is complete.
\end{proof}

\subsection{Fourier transform associated to $\mathcal{L}_2$}\label{ssecFT2}
Let $\Phi_2(r,k^2)$ denote the eigenfunctions associated to $\mathcal{H}_2$ constructed above,
and recall that $\mathcal{H}_2 (\sqrt{r} f) = \sqrt{r} \mathcal{L}_2f$ for all $f \in \mathrm{Dom}(\mathcal{L}_2)$.
From the general theory of \cite{GZ,KST,KMS} we have, at least formally for nice enough $f$, the Fourier representation
\begin{align}\label{FT2}
\sqrt{r} f(r) = \int_{0}^\infty \Phi_2(r,\xi) \Big[ \int_{0}^\infty \Phi_2(s,\xi) (\sqrt{s}f )  \, ds
  \Big] \, \rho_2(d\xi),
\end{align}
where $\rho_2(d\xi) = \frac{d\rho_2}{d\xi}(\xi) \, d\xi$ with the formulas in Proposition \ref{propSMrho};
in particular we can re-write this as
\begin{align}\label{FT2'}
\begin{split}
& \sqrt{r} f(r) = \int_{0}^\infty \Phi_2(r,k^2) \wtF_2 (\sqrt{r}f ) (k) \, \frac{d\rho_2}{dk}(k^2) 2k \, dk, \qquad  
  \\
& \wtF_2 (\sqrt{r}g) (k) := \int_{0}^\infty \Phi_2(r,k^2) g(r) \, dr.
\end{split}
\end{align}
We also have the diagonalization property
\begin{align}\label{FT2diag0}
\wtF_2( n(\mathcal{H}_2) g )(k) = n(k^2) (\wtF_2 g)(k).
\end{align}
Here is a rigorous convergence result for the integrals in \eqref{FT2}:

\begin{prop}\label{propFT2}
With the definitions in \eqref{FT2}-\eqref{FT2'}. let $f \in L^2 (rdr) \cap C^2((0,\infty))$ such that
\begin{align}\label{propFT2as}
\int_0^\infty \big( r^{3/2} |f''(r)| + r^{1/2} |f'(r)| + r^{-1/2}|f(r)| ) \, dr = M < \infty.
\end{align}
Then, the limit
\begin{align}\label{propFT2conc1}
\wtF_2 (\sqrt{r}f) (\sqrt{\xi}) := \lim_{R \rightarrow \infty} \int_0^R \Phi_2(r,\xi) \sqrt{r} f(r) \, dr
\end{align}
exists for all $\xi > 0$ and satisfies
\begin{align}\label{propFT2conc2}
\int_{0}^\infty |\Phi_2(r,\xi)| |\wtF_2 (\sqrt{r}f) (\sqrt{\xi})| \, \rho_2(d\xi) \lesssim M.
\end{align}
\end{prop}

Note that in some the formulas we are using the frequency variables $\xi$,
with the understanding that $\xi $ plays the role of $k^2$.

\begin{proof}
The proof is similar to \cite{KST,KMS} but requires a slightly more careful estimation
of the generalized eigenfunctions. 
First, we write
\begin{align*}
& B_0(\xi) = \int_0^\infty \varphi_{\leq 0}(s^2\xi) \Phi_2(s,\xi) \sqrt{s} f(s) \, ds,
\\
& B_j(\xi) = \int_0^\infty \varphi_j(s^2\xi) \Phi_2(s,\xi) \sqrt{s} f(s) \, ds, \qquad j \geq 1.
\end{align*}
so that, formally at least, we have
\begin{align}\label{propFT2prsum}
\wtF_2 (\sqrt{r}f) (\sqrt{\xi}) = B_0(\xi) + \sum_{j\geq 1} B_j(\xi).
\end{align}

To estimate the above terms, first recall that, from Lemmas \ref{lemFou1} and \ref{lemWeyl}, 
we have 
($\xi = k^2$ here)
\begin{align}\label{propFT2pr0}
|\Phi_2(r,\xi)| \lesssim \left\{ 
\begin{array}{ll}
r^{3/2} \jr^{-1}, &  r^2\xi \lesssim 1,
\\
\xi^{-1/4} \jxi^{-1/2}, &  r^2\xi \gtrsim 1.
\end{array} 
\right.
\end{align}
See also the more precise \eqref{prKG2Phi}-\eqref{prKG2phiest}.
Also, note the elementary estimate, which we will use a few times below, $y \langle y \rangle^{-1} \jxi^{1/2} \lesssim 1$ whenever $y^2\xi \lesssim 1$.

Using \eqref{propFT2pr0}, it follows that 
\begin{align*}
|B_0(\xi)| \lesssim \int_0^\infty \varphi_{\leq0} (s^2\xi) s^2 \langle s \rangle^{-1} |f(s)| \, ds,
\end{align*}
and, therefore, using again  \eqref{propFT2pr0} and \eqref{propSMrho},
\begin{align*}
\int_0^\infty | \Phi(r,\xi) | |B_0(\xi)| \rho_2(d\xi) \,d\xi 
& \lesssim \int_0^\infty\int_0^\infty \varphi_{\leq 0}(r^2\xi)\varphi_{\leq 0}(s^2\xi) \jxi \frac{r^{3/2}}{\jr} \frac{s^2}{\langle s \rangle} |f(s)| \,d\xi ds
\\
& + \int_0^\infty\int_0^\infty \varphi_{\geq 0}(r^2\xi)\varphi_{\leq 0}(s^2\xi) \xi^{-1/4}  \jxi^{1/2} \frac{s^2}{\langle s \rangle} |f(s)| \,d\xi ds
:= A_1 + A_2.
\end{align*}
For the first term we have the estimate
\begin{align*}
A_1 & \lesssim \int_0^\infty \int_0^{C\min(s^{-2},r^{-2})} 
	r^{1/2} s  \,d\xi  \, |f(s)| \, ds
	\lesssim \int_0^\infty s^{-1/2} |f(s)\vert \, ds.
\end{align*}
For the second term, we can similarly bound
\begin{align*}
A_2 & \lesssim \int_0^\infty \int_0^{Cs^{-2}} 
	\xi^{-1/4} \, d\xi \, s |f(s)| \, ds	
\lesssim  \int_0^\infty s^{-1/2} |f(s)| \, ds.
\end{align*}
These last two estimates are consistent with the bound claimed in \eqref{propFT2conc2}. 

In the region $r^2\xi \gtrsim 1$ we use the description \eqref{propSM0} 
through the Weyl function from Lemma \ref{lemWeyl} to write, for $j \geq 1$
\begin{align*}
B_j(\xi) & = \int_0^\infty \varphi_j(r^2\xi) \xi^{-1/4}e^{i\sqrt{\xi}r} \sigma_{2,\infty}(\xi^{1/2}r,r) a_2(\xi) 
  \, \sqrt{r} f(r) \, dr
\\ & + \int_0^\infty \varphi_j(r^2\xi) \xi^{-1/4}e^{-i\sqrt{\xi}r} \overline{\sigma_{2,\infty}}(\xi^{1/2}r,r) 
  \overline{a_2}(\xi) \, \sqrt{r} f(r) \, dr.
\end{align*}
It suffices to look at the first of the two terms since they are nearly identical.
We can integrate by parts twice, using $a_2(\xi) \approx \jxi^{-1/2}$,
and the symbol properties \eqref{Fou22}-\eqref{Fou23}, to estimate
\begin{align}\label{propFT2prBj}
\begin{split}
|B_j(\xi)| & \lesssim \frac{1}{\xi^{5/4}} \frac{1}{\jxi^{1/2}}
	\int_0^\infty \big| \partial_r^2 \big( \varphi_j(r^2\xi) \sigma_{2,\infty}(\xi^{1/2}r,r) \, \sqrt{r} f(r) \big) \big| \, dr
	\\
	& \lesssim \frac{1}{\xi^{5/4}} \frac{1}{\jxi^{1/2}} \int_{r^2\xi \approx 2^j} 
		\big| r^{1/2} f''(r) \big| + \big| r^{-1/2} f'(r) \big| + \big| r^{-3/2} f(r) \big|\, dr.
\end{split}
\end{align}
Using the above estimate with 
\eqref{propSMrho}, we obtain
\begin{align}\label{propFT2pr5}
\begin{split}
& \sum_{j\geq 1} \int_0^\infty | \Phi(r,\xi) | |B_j(\xi)| \rho_2(d\xi) \,d\xi
\\
& \lesssim \sum_{j\geq 1} \int_0^\infty
	|\Phi(r,\xi)|  \frac{1}{\xi^{5/4}} \jxi^{1/2}  \int_{s^2\xi \approx 2^j} 
	\big| s^{1/2} f''(s) \big| + \big| s^{-1/2} f'(s) \big| + \big| s^{-3/2} f(s) \big|\, ds \,d\xi
\\
& \lesssim \int_0^\infty C(s) \cdot \sum_{j\geq 1} \Big[
	\int_{\xi \approx 2^j s^{-2}}  |\Phi(r,\xi)|  \frac{1}{\xi^{5/4}} \jxi^{1/2} \,d\xi \Big] \, ds,
\end{split}
\end{align}
where we have defined 
$$C(s) := \big| s^{1/2} f''(s) \big| + \big| s^{-1/2} f'(s) \big| + \big| s^{-3/2} f(s) \big|.$$
We can then estimate, using $|\Phi(r,\xi)| \lesssim \xi^{-1/4} \jxi^{-1/2}$
when $r^2\xi \gtrsim 1$ (see \eqref{lemWeyl} and \eqref{propSM2}),
\begin{align*}
& \sum_{j\geq 1}
  \int_{\xi \approx 2^j s^{-2}}  \varphi_{\geq 0}(r^2\xi) |\Phi(r,\xi)|  \frac{1}{\xi^{5/4}} \jxi^{1/2} \,d\xi
  \lesssim \sum_{j\geq 1}
  \int_{\xi \approx 2^j s^{-2}} \frac{1}{\xi^{3/2}} \,d\xi
  \lesssim s,
\end{align*}
and, when $r^2\xi \lesssim 1$ instead, using $|\Phi(r,\xi)| \lesssim r^{3/2} \jr^{-1}$,
\begin{align*}
& \sum_{j\geq 1}
  \int_{\xi \approx 2^j s^{-2}}  \varphi_{\leq 0}(r^2\xi) |\Phi(r,\xi)|  \frac{1}{\xi^{5/4}} \jxi^{1/2} \,d\xi
  \lesssim \sum_{j\geq 1}
  \int_{\xi \approx 2^j s^{-2}}  \varphi_{\leq 1}(r^2\xi) \frac{r^{3/2}}{\jr} 
  \frac{1}{\xi^{5/4}} \jxi^{1/2} \,d\xi
  \\ & 
  \lesssim  \sum_{j\geq 1}
  \int_{\xi \approx 2^j s^{-2}} \frac{1}{\xi^{3/2}} \,d\xi \lesssim s,
\end{align*}
having used again $r\jr^{-1} \jxi^{1/2} \lesssim 1$ for $r^2\xi \lesssim 1$. 
Plugging-in these last two bounds into \eqref{propFT2pr5} gives
\begin{align}\label{propFT2pr6}
\begin{split}
& \sum_{j\geq 1} \int_0^\infty | \Phi(r,\xi) | |B_j(\xi)| \rho_2(d\xi) \,d\xi 
  \lesssim \int_0^\infty C(s) s \, ds
\end{split}
\end{align}
which proves \eqref{propFT2conc2}.
The existence of the pointwise limit ($\xi>0$) in \eqref{propFT2conc1} is a consequence of
the convergence of the sum in \eqref{propFT2prsum}; we can see this 
by summing the right-hand side of \eqref{propFT2prBj} to obtain, using that $r^{-1} \lesssim \sqrt{\xi}$ 
on the support of the integral,
\begin{align*}
\begin{split}
\sum_{j\geq 1} |B_j(\xi)| & \lesssim \frac{1}{\xi^{5/4}} \frac{1}{\jxi^{1/2}} \int_{r^2\xi \gtrsim 1} 
  \big| r^{1/2} f''(r) \big| + \big| r^{-1/2} f'(r) \big| + \big| r^{-3/2} f(r) \big|\, dr
\\
  & \lesssim \frac{1}{\xi^{5/4}} 
  \int_0^\infty \big| r^{3/2} f''(r) \big| + \big| r^{1/2} f'(r) \big| + \big| r^{-1/2} f(r) \big|\, dr \lesssim_\xi M
\end{split}
\end{align*}
This concludes the proof.
\end{proof}


\medskip
\section{Linear spectral analysis for $\mathcal{L}_1$}\label{secL1}
Recall the definition of $\mathcal{L}_1$
and introduce the operator
\begin{align}\label{defH_1}
\mathcal{H}_1 := -\partial_r^2 + \tfrac{3}{4r^2} - (1-3U^2(r)), \qquad  \mathcal{H}_1 = r^{1/2} \mathcal{L}_1 r^{-1/2}.
\end{align}
We will try to follow the general arguments from the previous section. 
However, unlike the case of $\mathcal{L}_2$ treated before, 
here we do not have an explicit zero energy solution to begin with.
Therefore, to construct the generalized eigenfunctions satisfying 
\begin{align}\label{L1evalue}
\mathcal{H}_1 f = (k^2 + 2)f, \quad k\geq 0,  
\end{align}
we will need some additional steps. 
Another main difference
in this case compared to the previous one comes from the oscillations 
of the generalized eigenfunctions, which can be seen from the fact that, 
$\mathcal{H}_1 - 2 \approx -\partial_r^2 - (9/4)r^{-2}$ as $r\rightarrow \infty$, for which 
a fundamental system of real-valued solutions is given by \eqref{fsys_h1_inf_app}.
The oscillations make some of the following calculations more involved; 
in particular, showing lower bounds for the spectral measures requires a little more work.

\subsection{Fundamental solutions at zero energy}
We look for $\Phi_1^{(0)}$ solving
$\mathcal{H}_2 \Phi_1^{(0)} = 2(1 - U^2) \Phi_1^{(0)}$,
and expect that 
$\Phi_1^{(0)}(r) 
\approx r^{3/2}$ for small $r$.
For large $r$ we 
expect $\Phi_1^{(0)}(r)$ to be approximated by solutions of 
$-\partial_r^2 f - (9/4)r^{-2} f = 0$,
for which a fundamental system is 
\begin{align}\label{fsys_h1_inf_app}
y_1(r)=\sqrt{r}\cos\big(\sqrt{2}\ln(r)\big), \qquad y_2(r)=-\sqrt{r/2} \sin\big(\sqrt{2}\ln(r)\big).
\end{align}
Observe that the normalization was chosen so that $W[y_2,y_1](r)=1$. 

\begin{lem}\label{lem_p1_t1}
There exists a fundamental system of real-valued solutions of $\mathcal{H}_1f=2f$ with the following asymptotic behavior
\begin{align}\label{Phi1_0}
\Phi_{1}^{(0)} (r) & = \begin{cases}
r^{3/2}+O(r^{7/2}), & r\lesssim 1,
\\ c_1\sqrt{r}\cos\big(\sqrt{2}\ln(r)\big)+c_2\sqrt{r}\sin\big(\sqrt{2}\ln(r)\big)+O(r^{-3/2}),  &  r \gtrsim 1,
\end{cases}
\\ \Theta_{1}^{(0)} (r) & = \begin{cases}
 \tfrac{1}{2}r^{-1/2} + O(r^{3/2}), & r\lesssim 1,
\\ c_3\sqrt{r}\cos\big(\sqrt{2}\ln(r)\big)+c_4\sqrt{r}\sin\big(\sqrt{2}\ln(r)\big)+O(r^{-3/2}), &   r \gtrsim 1, 
\end{cases} \label{Theta1_0}
\end{align}
for some $c_1,...,c_4$ with $c_2c_3-c_1c_4=\tfrac{1}{\sqrt{2}}$.
Also, they are normalized so that $W[\Theta_{1}^{(0)},\Phi_{1}^{(0)}]=1$.
\end{lem}

\begin{proof}
We begin by constructing solutions at infinity using a perturbative argument.
We write \eqref{L1evalue} with $k=0$ 
as 
\begin{align}\label{L1_eqrewri}
-y''(r)-\dfrac{9}{4r^2}y(r)=-\Big(\dfrac{3}{4r^2}-(3-3U^2(r))+\dfrac{9}{4r^2}\Big)y(r)=:\widetilde{V}(r)y(r).
\end{align}
From the asymptotics of $U(r)$ it follows that 
\[
\widetilde{V}(r)=O(r^{-4}) \qquad r\gtrsim1.
\]
Using the fundamental system of solutions \eqref{fsys_h1_inf_app}, 
recalling that $W[y_2,y_1](r)=1$, 
we have that
\begin{align*}
\Phi_{1,\infty}^{(0)}(r) &=\sqrt{r}\cos\big(\sqrt{2}\ln(r)\big)+ \dfrac{1}{\sqrt{2}}\int_r^\infty \Big(\sqrt{rs}\cos\big(\sqrt{2}\ln(s)\big)\sin\big(\sqrt{2}\ln(r)\big)
\\ & \quad -\sqrt{rs}\cos\big(\sqrt{2}\ln(r)\big)\sin\big(\sqrt{2}\ln(s)\big)\Big)\widetilde{V}(s)\Phi_{1,\infty}^{(0)}(s)ds,
\\ \Theta_{1,\infty}^{(0)}(r)& = -\dfrac{\sqrt{r}}{\sqrt{2}} \sin\big(\sqrt{2}\ln(r)\big)+\dfrac{1}{\sqrt{2}}\int_r^\infty\Big(\sqrt{rs}\cos\big(\sqrt{2}\ln(s)\big)\sin\big(\sqrt{2}\ln(r)\big)
\\ & \quad -\sqrt{rs}\cos\big(\sqrt{2}\ln(r)\big)\sin\big(\sqrt{2}\ln(s)\big)\Big)\widetilde{V}(s)\Theta_{1,\infty}^{(0)}(s),
\end{align*}
is a fundamental system of solutions to equation \eqref{L1_eqrewri};
we claim that this can be written as 
\begin{align}\label{FS0infty}
\begin{split}
\Phi_{1,\infty}^{(0)}(r) & = \sqrt{r}\cos\big(\sqrt{2}\ln(r)\big)+\varepsilon_{1,\infty}(r),
\\ 
\Theta_{1,\infty}^{(0)}(r) & = -\dfrac{\sqrt{r}}{\sqrt{2}} \sin\big(\sqrt{2}\ln(r)\big)+\varepsilon_{2,\infty}(r),
\end{split}
\end{align}
with $\varepsilon_{1,\infty}(r), \varepsilon_{2,\infty}(r) = O(r^{-3/2})$ lower order terms.
Since the argument 
is similar in the two cases,
we only perform it in the case of $\Phi_{1,\infty}^{(0)}(r)$. 
Substituting the ansatz into the integral equation above we see that 
$\varepsilon_{1,\infty}(r)$ satisfies the Volterra equation 
\begin{align*}
\varepsilon_{1,\infty}(r) & = \dfrac{1}{\sqrt{2}}\int_r^\infty \Big(\sqrt{rs}\cos\big(\sqrt{2}\ln(s)\big)\sin\big(\sqrt{2}\ln(r)\big)
\\ & \quad -\sqrt{rs}\cos\big(\sqrt{2}\ln(r)\big)\sin\big(\sqrt{2}\ln(s)\big)\Big)\widetilde{V}(s)\left(\sqrt{s}\cos\big(\sqrt{2}\ln(s)\big)+\varepsilon_{1,\infty}(s)\right)ds.
\end{align*}
Thanks to the fast decay of $\widetilde{V}(s)=O(s^{-4})$, the leading order term in the last expression satisfies 
\begin{align*}
& \bigg\vert \int_r^\infty \Big(\sqrt{rs}\cos\big(\sqrt{2}\ln(s)\big)\sin\big(\sqrt{2}\ln(r)\big)
\\ & \quad -\sqrt{rs}\cos\big(\sqrt{2}\ln(r)\big)\sin\big(\sqrt{2}\ln(s)\big)\Big)\widetilde{V}(s)\sqrt{s}\cos\big(\sqrt{2}\ln(s)\big)ds\bigg\vert  \lesssim  \sqrt{r}\int_r^\infty s^{-3}ds \lesssim r^{-3/2}. 
\end{align*}
Therefore, 
we define $T(f):X\to X$ the map given by 
\begin{align*}
T(f)&:=\dfrac{1}{\sqrt{2}}\int_r^\infty \Big(\sqrt{rs}\cos\big(\sqrt{2}\ln(s)\big)\sin\big(\sqrt{2}\ln(r)\big)
\\ & \,\quad -\sqrt{rs}\cos\big(\sqrt{2}\ln(r)\big)\sin\big(\sqrt{2}\ln(s)\big)\Big)\widetilde{V}(s)\left(\sqrt{s}\cos\big(\sqrt{2}\ln(s)\big)+f(s)\right)ds,
\end{align*}
with $X := \{f\in C([r_0,\infty)): \, r^{3/2}f(r)\in L^\infty((r_0,\infty)) \}$ for $r_0\gg1$ large enough but fixed. 
By similar computations as before one sees that 
\[
\big\vert T[f](r)-T[g](r)\big\vert \lesssim r^{-2}\Vert f-g\Vert_{L^\infty(r,\infty)} .
\]
This implies that $\Vert Tf-Tg\Vert_X=\Vert r^{3/2}(Tf-Tg)\Vert_{L^\infty_r((r_0,\infty))}\lesssim r_0^{-1/2}\Vert f-g\Vert_X$,
hence $T$ is a contraction in $X$ for $r_0$ large enough. 
Then, a direct application of Banach's Fixed Point Theorem 
leads to the existence, uniqueness and decay of $\varepsilon_{1,\infty}(r)$.

Next, using the fundamental system of solutions at infinity \eqref{FS0infty}
we can construct the (real-valued) Weyl solution at zero in the whole half-line $(0,\infty)$, 
namely $\Phi_1^{(0)}(r)$, which satisfies
\begin{align}\label{asymp_phi_1_0}
\Phi_1^{(0)}(r)=\begin{cases}
r^{3/2}+O(r^{7/2}) & r\ll1 
\\ c_1\sqrt{r}\cos\big(\sqrt{2}\ln(r)\big)+c_2\sqrt{r}\sin\big(\sqrt{2}\ln(r)\big)+\varepsilon(r) & r\gg 1,
\end{cases}
\end{align}
for some $c_1,c_2\in\R$, with $\varepsilon(r)$ satisfying \[
\vert \varepsilon(r)\vert \lesssim r^{-3/2}, \qquad r\gg1.
\]
In fact, it is enough to see that around zero the equation \eqref{L1evalue} with $k=0$ is approximately 
$-y''+(3/4)r^{-2}y=0,$ 
whose fundamental system of solution is 
$y_1(r)=r^{3/2}$, 
$y_2(r)=(1/2)r^{-1/2}.$
where only $y_1$ is 
in $L^2\big((0,1)\big)$. 
Then, reasoning along similar lines as before, 
taking advantage of the fact that $1-U^2(r)=1-a^2r^2+O(r^4)$ for $r \ll 1$, 
it is not difficult to see that the Weyl real-valued solution of 
$\mathcal{H}_1f=2f$
can be written (up to a multiplicative constant), for  $r \ll 1$, as \[
\Phi_{1,0}^{(0)}(r)=r^{3/2}+\eta(r), \qquad \vert \eta(r)\vert=O(r^{7/2}),
\]
by a fixed point argument close to $r=0$. More specifically, writing the equation as \[
-y''(r)+\dfrac{3}{4r^2}y(r)=(3-3U^2(r))y(r)=:\widetilde{V}_0(r)y(r),
\]
then the map $T_0$ defined on $X_0=\{f\in C([0,r_0)): \, r^{-7/2}f(r)\in L^\infty((0,r_0)) \}$ with $r_0\ll1$, 
as 
\begin{align*}
T_0(f)&:=\dfrac{1}{2}\int_0^r \Big(r^{3/2}s^{-1/2}-s^{3/2}r^{-1/2}\Big)\widetilde{V}_0(s)\left(s^{3/2}+f(s)\right)ds
\end{align*} 
is a contraction on $X_0$. Then, by extending this solution up to $r=+\infty$, 
using the fact that $\Phi_1^{(0)}(r)$ solves the equation for all $r>0$, we conclude 
that there exist constants $c_1,c_2$ such that
\begin{align}\label{Phi1_0pr}
\Phi_1^{(0)}(r) = c_1 \Phi_{1,\infty}^{(0)}(r) - \sqrt{2}c_2 \Theta_{1,\infty}^{(0)}(r), \qquad r> r_0;
\end{align}
in view of \eqref{FS0infty} this gives \eqref{asymp_phi_1_0} 
with $\varepsilon = c_1\varepsilon_{1,\infty} - \sqrt{2}c_2\varepsilon_{2,\infty}$,
and \eqref{Phi1_0} follows.

Once $\Phi_1^{(0)}(r)$ is constructed, we can find $\Theta_1^{(0)}(r)$ 
by solving the equation 
$W[\Theta_1^{(0)},\Phi_1^{(0)}]=1,$ 
with $\Theta_1^{(0)}(1)=0,$
for $r\lesssim1$, namely, 
\begin{align}\label{Theta10}
\Theta_1^{(0)}(r)=-\Phi_1^{(0)}(r)\int_1^r (\Phi_1^{(0)}(s))^{-2} \, ds.
\end{align}
Extending this solution to $r=+\infty$, 
leads to the asymptotics stated in \eqref{Theta1_0},
where the condition on the constants, 
$c_2c_3-c_1c_4 = 1/\sqrt{2},$
comes from $W[\Theta_1^{(0)},\Phi_1^{(0)}]=1$ at infinity. 
From \eqref{Theta10} for $r$ close to zero we obtain the factor $\tfrac12$ 
in front of $\Theta_1^{(0)}$ appearing in \eqref{Theta1_0} for $r\ll1$.
This concludes the proof of \eqref{Theta1_0} and the lemma.
\end{proof}

\subsection{Region $rk \lesssim 1$}
In this subsection we seek to construct a power series expansion for $\Phi_1(r,k^2)$
for all $r,k>0$. 
As before, we proceed perturbatively using the fundamental system for $k^2=0$ we found above.

\begin{lem}\label{Phi1Op1_1}
For all $r>0$, $k> 0$ we have the expansion
\begin{align}\label{Phi1Op1_2}
\Phi_1(r,k^2) = \Phi_{1}^{(0)}(r) + \sqrt{r} \sum_{j\geq 1} (rk)^{2j} \Phi_{1,j}(r),
\end{align}
which converges absolutely; the expansion converges uniformly if $rk \leq 1$, and
$\Phi_{1,j}$ are smooth. In the case of $j=1$ we have
\begin{align}\label{Phi1Op1_10}
\Phi_{1,1}(u) = \begin{cases} 
\tfrac18r + O(r^3), &  r \lesssim 1,
\\ \tfrac{1}{12}\Big((c_1-\sqrt{2}c_2)\cos(\sqrt{2}\ln r)+(\sqrt{2}c_1+c_2)\sin(\sqrt{2}\ln r)\Big) + O(\tfrac{1}{r^2}), &  r \gtrsim 1.
\end{cases}
\end{align}
Moreover, for some absolute $C>0$, for all $j\geq 1$
\begin{align}\label{Phi1Op1_3}
| \Phi_{1,j}(u) | \leq\frac{C^j}{j!(j+1)!} u , \quad u \lesssim 1,
\end{align}
while for $r\gtrsim1$ the following holds
\begin{align}\label{Phi1Op1_8}
\Phi_{1,j}(r)=c_{1,j}\cos(\sqrt{2}\ln(r))+c_{2,j}\sin(\sqrt{2}\ln(r))+O\big(\tfrac{1}{(j-2)!}r^{-2}\big), \qquad r\gtrsim1,
\end{align}
with constants satisfying 
\begin{align}\label{Phi1Op1_9}
\vert c_{1,j}\vert+\vert c_{2,j}\vert \lesssim  \dfrac{\vert c_1\vert+\vert c_2\vert}{2^{2j}(j!)^2}.
\end{align}
In \eqref{Phi1Op1_3} we also have consistent symbol-type bounds for the derivatives, that is,
\begin{equation}\label{Phi1Op1_4}
\begin{aligned}
\qquad \vert (r\partial_r)^m\Phi_{1,j}(r)\vert&\lesssim r,  & \qquad r\lesssim 1,
\end{aligned}
\end{equation}
whereas, for \eqref{Phi1Op1_8}, we have 
\begin{align}\label{Phi1Op1_12}
(r\partial_r)^m\Phi_{1,j}(r) 
  = \widetilde{c}_{1,j,m}\cos(\sqrt{2}\ln(r))+\widetilde{c}_{2,j,m}\sin(\sqrt{2}\ln(r))+O\big(\tfrac{1}{(j-2)!}r^{-2}\big), \qquad r\gtrsim1,
\end{align}
with constants satisfying
\begin{align}\label{Phi1Op1_12'}
\vert \widetilde{c}_{1,j,m}\vert+\vert \widetilde{c}_{2,j,m}\vert \lesssim  \dfrac{C^m}{((j-1)!)^2}.
\end{align}
\end{lem}

\begin{proof}
In order to solve $\mathcal{H}_1\Phi_1(r,k^2) = (k^2+2) \Phi_1(r,k^2)$ we begin by making the ansatz
\begin{align*}
\Phi_1(r,k^2) = \sqrt{r} \sum_{j\geq 0} k^{2j} f_j(r), \qquad f_0(r) := r^{-1/2}\Phi_1^{(0)}(r).
\end{align*}
Substituting this 
into the equation we are led to the relation
\begin{align}\label{Phi1Op1_5}
\mathcal{H}_1 (\sqrt{r} f_j(r))-2\sqrt{r}f_j = \sqrt{r} f_{j-1}(r), \quad j\geq 0, \quad f_{-1}(r) = 0.
\end{align}
Comparing with \eqref{Phi1Op1_2} 
we have $f_j(r) = r^{2j} \Phi_j(r)$.
From \eqref{Phi1Op1_5} and the fundamental solution (Green's function) 
for $\mathcal{H}_1f=2f$ given by
\begin{align}\label{Phi1Op1_1W}
\Theta_1^{(0)}(r)\Phi _1^{(0)}(s) - \Theta_1^{(0)}(s)\Phi_1^{(0)}(r)  \quad  \hbox{ for }  \quad r>s,
\end{align}
we get the recursion formula
\begin{align}\label{Phi1Op1_13}
\sqrt{r} f_j(r) = \int_0^r 
\Big( \Theta_1^{(0)}(s)\Phi_1^{(0)}(r)-\Theta_1^{(0)}(r)\Phi _1^{(0)}(s)\Big) \sqrt{s} f_{j-1}(s)\, ds.
\end{align}
Recalling \eqref{Theta10}, 
it follows that the recursion for $f_j(r)$ can be conveniently rewritten as 
\begin{align}\label{Phi1Op1_6}
\begin{split}
& f_j(r) 
= \int_0^r K(r,s) f_{j-1}(s)\, ds,
\\
& K(r,s) := \big(\tfrac{s}{r}\big)^{1/2}\Phi_1^{(0)}(r)\Phi_1^{(0)}(s) \int_s^r \dfrac{1}{\Phi_1^{(0)}(t)^2} \, dt, 
  \quad r\lesssim1.
\end{split}
\end{align}
Observe that $K(r,s)$ is strictly positive for $r\ll 1$. 
Using the asymptotics in \eqref{Phi1_0}-\eqref{Theta1_0} for small and large $r$, we obtain 
\begin{align}\label{Phi1Op1_7}
K(r,s)
= \left\{ 
\begin{array}{ll}
\dfrac{r}{2} \big( 1 - \tfrac{s^2}{r^2} \big)+\widetilde{c}  rs^2\log\tfrac{r}{s} +r^3O(1-\tfrac{s^4}{r^4}) & s < r \lesssim 1,
\\ 
\\
\dfrac{s}{\sqrt{2}}\sin\big(\sqrt{2}\ln(\tfrac{r}{s})\big) +O(s^{-1})&  r > s \gtrsim 1,
\end{array} 
\right.
\end{align}
for some constant $\widetilde{c}\in\R$;
to calculate the asymptotic for $r>s\gtrsim 1$ we have used the explicit asymptotics 
for $\Theta_1^{(0)}(r)$, given by \eqref{Theta1_0}, in the formula \eqref{Phi1Op1_1W}
(cfr. \eqref{Phi1Op1_13}),
elementary trigonometric identities, 
and
$c_2c_3-c_1c_4=1/\sqrt{2}.$

%

\smallskip
\noindent
{\it Proof of \eqref{Phi1Op1_10}}.
Let us analyze $f_1$. From \eqref{Phi1Op1_6}, when $r \lesssim 1$, we can use \eqref{Phi1_0} again to see that
\begin{align}\label{Phi1Op1_f1_rsmall}
f_1(r) & = r^{-1/2} \Phi_{1}^{(0)}(r) \int_0^r \Phi_1^{(0)}(s)^2 \Big( \int_s^r \frac{1}{\Phi_1^{(0)}(t)^2} \, dt \Big) \, ds \nonumber
  \\
  & = (r + O(r^3)) \int_0^r (s^3 + O(s^5)) \Big( \int_s^r \frac{1}{t^3} (1 -2ct^2+O(t^4)) \, dt \Big) \, ds \nonumber
  \\ & = \big(r+O(r^3)\big) \int_0^r \big(s^3 + O(s^5)\big) \frac{1}{2}(s^{-2} - r^{-2}-4c\log\tfrac{r}{s}+O\big(r^2+s^2\big)) \, ds \nonumber
  \\ & = \frac{1}{8} r^3 + O(r^5),
\end{align}
where the constant $c$ in $-2ct^2$ above is some constant $c\in\R$ coming from the $O(r^{7/2})$ term in \eqref{Phi1_0}. 
Since $f_1(r) = r^{2} \Phi_{1,1}(r)$ the asymptotic formula for $r\lesssim 1$ in \eqref{Phi1Op1_10} follows. 

When $r \gg 1$, we use the \eqref{Phi1_0} 
and \eqref{Phi1Op1_7} 
for $r_0 \gtrsim 1$ large enough and fixed, to infer that
\begin{align}\label{Phi1Op1_f1_rlarge}
\int_{r_0}^r K(r,s)f_0(s)\, ds
 & = \dfrac{1}{\sqrt{2}}\int_{r_0}^r  \Big(s^{1/2} \sin\big(\sqrt{2}\ln(\tfrac{r}{s}) \big) +s^{-1/2}O(s^{-1}+sr^{-2}) \Big)\Phi_1^{(0)}(s)ds 
 \\ 
 & = \dfrac{1}{12\sqrt{2}}\Big((\sqrt{2}c_1-2c_2)\cos(\sqrt{2}\ln r)+(2c_1+\sqrt{2}c_2)\sin(\sqrt{2}\ln r)\Big)r^2 +O(1). \nonumber
\end{align}
where $c_1$ and $c_2$ are the constants in $\Phi_1^{(0)}(r)$. 
Since the integration from $0$ to $r_0$ gives a bounded contribution, 
it follows that 
\begin{align*}
f_1(r) & = \dfrac{1}{12}\Big((c_1-\sqrt{2}c_2)\cos(\sqrt{2}\ln r)+(\sqrt{2}c_1+c_2)\sin(\sqrt{2}\ln r)\Big)r^2 
	+ O(1), \qquad r \gtrsim 1,
\end{align*}
which gives the second asymptotic formula for $\Phi_{1,1}(r,k^2)$.
 
\medskip

\noindent
{\it Proof of \eqref{Phi1Op1_8}-\eqref{Phi1Op1_9}}.
We proceed by induction having already proved the case $n=1$.
Assuming there exist some constants such that 
\[
f_{n-1}(r)=\Big(c_{1,n-1}\cos(\sqrt{2}\ln r)+c_{2,n-1}\sin(\sqrt{2}\ln r)\Big)r^{2(n-1)}
  +O\left(\tfrac{1}{(n-3)!}r^{2(n-2)}\right),
\]
substituting into the identity for $f_n(r)$ in terms of $f_{n-1}(r)$, for $r>r_0\gg1$, we find that \begin{align*}
f_n(r) & =\int_{r_0}^r K(r,s)f_{n-1}(s)ds
\\ & = \dfrac{1}{\sqrt{2}}\int_{r_0}^r \Big(s\sin\big(\sqrt{2}\ln(\tfrac rs)\big)+O(s^{-1})\Big) \times 
\\ & \quad \times \Big(c_{1,n-1}s^{2(n-1)}\cos(\sqrt{2}\ln s)+c_{2,n-1}s^{2(n-1)}\sin(\sqrt{2}\ln s)+O\big(\tfrac{s^{2(n-2)}}{(n-3)!}\big)\Big)ds
\\ & = \dfrac{r^{2n}}{4n(2+n^2)}\Big((nc_{1,n-1}-\sqrt{2}c_{2,n-1})\cos(\sqrt{2}\ln r)
\\ & \quad +(\sqrt{2}c_{1,n-1}+nc_{2,n-1})\sin(\sqrt{2}\ln r)\Big)+O\left(\dfrac{r^{2(n-1)}}{(n-2)!}\right),
\end{align*} 
where we have used the following explicit antiderivatives \begin{align*}
&\int s^{2n-1}\sin\big(\sqrt{2}\ln(\tfrac rs)\big)\cos(\sqrt{2}\ln s)ds
\\ & \quad \ = \dfrac{s^{2n}}{4n(2+n^2)}\bigg(\sqrt{2}n\cos\big(\sqrt{2}\log(\tfrac{r}{s^2})\big)+(2+n^2)\sin(\sqrt{2}\log r)+n^2\sin\big(\sqrt{2}\log(\tfrac{r}{s^2})\big)\bigg)+C,
\end{align*}
and
\begin{align*}
\\ &\int s^{2n-1}\sin\big(\sqrt{2}\ln(\tfrac rs)\big)\sin(\sqrt{2}\ln s)ds
\\ & \quad = -\dfrac{s^{2n}}{4n(2+n^2)}\bigg(\sqrt{2}n\sin\big(\sqrt{2}\ln(\tfrac{r}{s^2}\big)+(2+n^2)\cos(\sqrt{2}\log r)-n^2\cos\big(\sqrt{2}\log(\tfrac{r}{s^2})\big)\bigg)+C,
\end{align*}
which give
\begin{align*}
& \int s\sin\big(\sqrt{2}\ln(\tfrac rs)\big)\Big(c_{1}s^{2(n-1)}\cos(\sqrt{2}\ln s)+c_{2}s^{2(n-1)}\sin(\sqrt{2}\ln s)\Big)ds
\\ 
& \qquad = \dfrac{s^{2n}}{4n(2+n^2)}\bigg(n(\sqrt{2}c_1+c_2n)\cos\big(\sqrt{2}\ln(\tfrac{r}{s^2})\big)
\\ 
& \qquad +(2+n^2)\big(c_1\sin(\sqrt{2}\ln r)-c_2\cos(\sqrt{2}\ln r)\big)
  +n(c_1n-c_2\sqrt{2})\sin\big(\sqrt{2}\ln(\tfrac{r}{s^2})\big)\bigg)+C.
\end{align*}
Observe that, in particular, the terms containing an additional multiplier factor of $n^2$ 
cancel each other when $s=r$ since 
\[
c_2n^2\cos\big(\sqrt{2}\ln(\tfrac{r}{r^2})\big)+c_1n^2\sin(\sqrt{2}\ln r)
  -c_2n^2\cos(\sqrt{2}\ln r)+c_1n^2\sin\big(\sqrt{2}\ln (\tfrac{r}{r^2})\big)=0.
\]
In the previous calculation of $f_n(r)$, the error term is $O\big(\tfrac{r^{2(n-1)}}{(n-2)!}\big)$. 
Regarding the constants $c_{1,n}$ and $c_{2,n}$, from our previous computations it 
follows that they satisfy the recursion 
\begin{align*}
c_{1,n} & = \dfrac{1}{4n(2+n^2)}\big(nc_{1,n-1}-\sqrt{2}c_{2,n-1}\big),
\qquad c_{2,n} = \dfrac{1}{4n(2+n^2)}\big(\sqrt{2}c_{1,n-1}+nc_{2,n-1}\big),
\end{align*}
which in turn implies \eqref{Phi1Op1_9}.

\smallskip
\noindent
{\it Proof of \eqref{Phi1Op1_3}}.
To obtain the upperbound on $\Phi_j(r) = r^{-2j} f_j(r)$, $j\geq 2$, for $r \lesssim  1$, we proceed by induction. 
Assuming that \[
f_k(r) \leq \frac{C^k}{k!(k+1)!} r^{2k+1} \qquad \hbox{ and } \qquad k \leq j-1,
\]
we see that computations similar to those used above for $f_1$ give
\begin{align*}
f_{j}(r) & = \int_0^r K(r,s) f_{j-1}(s) \, ds  
\\ & \leq \int_0^r  \left(\dfrac{r}{2}(1-\tfrac{s^2}{r^2})+\widetilde{c}  rs^2\log\tfrac{r}{s} +r^3O(1-\tfrac{s^4}{r^4})\right)\dfrac{C^{j-1}}{(j-1)!j!}s^{2j-1} \, ds
  \\ & =\dfrac{C^{j-1}}{2(j-1)!j!}\left(\dfrac{1}{2j}-\dfrac{1}{2j+2}\right)r^{2j+1}+O\left(\dfrac{r^{2j+3}}{j!(j+1)!}\right) \leq \dfrac{C^j}{j!(j+1)!}r^{2j+1},
\end{align*}
where we have used that the logarithmic term satisfies \[
\int_0^r r  s^{2j+1} \log\big(\tfrac{r}{s}\big)ds = \dfrac{  r^{2j+3}  }{4(j+2)^2}.
\]
which concludes the proof of \eqref{Phi1Op1_3}. 
To see that $f_j$ is smooth on $(0,\infty)$ one can proceed by induction.
In fact, the smoothness of $f_0(r)$ is direct and follows from that of $\Phi_1^{(0)}(r)$ 
along with the identity $f_0(r)=r^{-1/2}\Phi_1^{(0)}(r)$. 
Then, the general inductive case follows from the recursive representation \eqref{Phi1Op1_7}
using the smoothness of $K(r,s)$ on $\{(r,s)\in\R^2: \, 0<s<r\}$. 

\medskip
\noindent
{\it Proof of \eqref{Phi1Op1_4}-\eqref{Phi1Op1_12}}. 
To prove the symbol-type bounds for $r \gtrsim 1$,
we begin by differentiating the recursive formula \eqref{Phi1Op1_13}, to get that
\begin{align}\label{Phi1Op1_14}
\begin{split}
\partial_r f_j(r) & =
-\dfrac{1}{2r}f_j(r)+\dfrac{1}{\sqrt{r}}\int_0^r\Big(\Theta_1^{(0)}(s)\partial_r\Phi_1^{(0)}(r)
  -\partial_r\Theta_1^{(0)}(r)\Phi_1^{(0)}(s)\Big)\sqrt{s}f_{j-1}(s)ds
\\ & = -\dfrac{1}{2r}f_j(r)+\int_0^r \widetilde{K}(r,s)f_{j-1}(s)ds.
\end{split}
\end{align}
By direct computations, using \eqref{Phi1_0}, \eqref{Theta1_0} 
as well as the fact that $c_1c_4-c_2c_3=1/\sqrt{2}$, we find that the kernel in this case is
\begin{align*}
\widetilde{K}(r,s)=\dfrac{s}{4r}\Big(4\cos\big(\sqrt{2}\ln(\tfrac{r}{s})\big)+\sqrt{2}\sin\big(\sqrt{2}\ln(\tfrac{r}{s})\big)\Big)+O\big(\tfrac{1}{sr}+\tfrac{s}{r^3}\big), \qquad r>s\gtrsim 1.
\end{align*}
Thus, recalling that we just proved that \begin{align*}
f_{j-1}(r)&=\Big(c_{1,j-1}\cos(\sqrt{2}\ln r)+c_{2,j-1}\sin(\sqrt{2}\ln r)\Big)r^{2(j-1)}+O\left(\tfrac{1}{(j-3)!}r^{2(j-2)}\right)
\\ & =:\widetilde{f}_{j-1}(r)+O\left(\dfrac{r^{2(j-2)}}{(j-3)!}\right),
\end{align*} 
using the 
antiderivatives
\begin{align*}
& \int\dfrac{s}{r}\cos\big(\sqrt{2}\ln(\tfrac{r}{s})\big)\widetilde{f}_{j-1}(s)ds 
\\ & \qquad =\dfrac{s^{2j}}{4j(2+j^2)r}\Big(j(jc_{1,j-1}-\sqrt{2}c_{2,j-1})\cos\big(\sqrt{2}\ln(\tfrac{r}{s^2})\big)+c_{1,j-1}(2+j^2)\cos\big(\sqrt{2}\ln r\big)
\\ & \qquad \quad -j(\sqrt{2}c_{1,j-1}+jc_{2,j-1})\sin\big(\sqrt{2}\ln(\tfrac{r}{s^2})\big)+c_{2,j-1}(2+j^2)\sin\big(\sqrt{2}\ln r\big)\Big)+C,
\end{align*}
and
\begin{align*}
& \int\dfrac{s}{r}\sin\big(\sqrt{2}\ln(\tfrac{r}{s})\big)\widetilde{f}_{j-1}(s)ds 
\\ & \qquad =\dfrac{s^{2j}}{4j(2+j^2)r}\Big(j(\sqrt{2}c_{1,j-1}+jc_{2,j-1})\cos\big(\sqrt{2}\ln(\tfrac{r}{s^2})\big)-c_{2,j-1}(2+j^2)\cos\big(\sqrt{2}\ln r\big)
\\ & \qquad \quad +j(jc_{1,j-1}-\sqrt{2}c_{2,j-1})\sin\big(\sqrt{2}\ln(\tfrac{r}{s^2})\big)+c_{1,j-1}(2+j^2)\sin\big(\sqrt{2}\ln r\big)\Big)+C,
\end{align*}
we conclude that, for $r_0\gtrsim1$ large enough,
the leading order contribution coming from the integral term in \eqref{Phi1Op1_14} is given by 
\begin{align*}
& \int_{r_0}^r\widetilde{K}(r,s)f_{j-1}(s)ds
\\ &\qquad =\dfrac{r^{2j-1}}{4j(2+j^2)}\bigg(\left(c_{1,j-1}(2+\sqrt{2}j+2j^2)-c_{2,j-1}(2+\sqrt{2}j)\right)\cos\big(\sqrt{2}\ln r\big)
\\ & \qquad  \quad +\left(c_{1,j-1}(2+\sqrt{2}j)+c_{2,j-1}(2+\sqrt{2}j+2j^2)\right)\sin\big(\sqrt{2}\ln r\big)\bigg)+O(r^{2j-3})
\\ & \qquad = r^{2j-1}\Big(\widetilde{c}_{1,j}\cos\big(\sqrt{2}\ln r\big)+\widetilde{c}_{2,j}\sin\big(\sqrt{2}\ln r\big)\Big)+O\left(\dfrac{r^{2j-3}}{(j-2)!}\right),
\end{align*}
where the coefficients satisfy \begin{align*}
\vert \widetilde{c}_{1,j}\vert+\vert \widetilde{c}_{2,j}\vert \lesssim \dfrac{\vert c_{1,j-1}\vert+\vert c_{2,j-1}\vert}{j}.
\end{align*}
Therefore, using that 
$\partial_r\Phi_{1,j}(r)=r^{-2j}\partial_rf_j-2jr^{-1}\Phi_{1,j}$ 
we conclude \eqref{Phi1Op1_12} for $m=1$.

The general case follows from noticing that 
\begin{align*}
& \Theta_1^{(0)}(s)\partial_r^m\Phi_1^{(0)}(r) - \partial_r^m \Theta_1^{(0)}(r)\Phi_1^{(0)}(s)
\\ & 
\qquad =\dfrac{\textbf{c}_{1,m}}{r^{(2m-1)/2}}\cos\big(\sqrt{2}\ln(\tfrac{r}{s})\big)
  + \dfrac{\textbf{c}_{2,m}}{r^{(2m-1)/2}}\sin\big(\sqrt{2}\ln(\tfrac{r}{s})\big)
  + O\big(r^{-(2m+3)/2}\big),
\end{align*}
where the coefficients grow as $\vert \textbf{c}_{1,m}\vert+\vert \textbf{c}_{2,m}\vert\lesssim C^m$,
along similar lines to those above, and hence we can omit the details. 

Finally, \eqref{Phi1Op1_4} follows from a similar (in fact simpler) argument. 
\end{proof}

\subsection{Region $rk \gtrsim 1$}
The next lemma is the analogue of Lemma \ref{lemWeyl} from the previous section.
The proof is almost exactly the same, so we will skip most of the details,
and once again refer the reader to Proposition 5.6 in \cite{KST}.

\begin{lem}\label{lemWeyl1}
For any $k^2>2$, the Weyl-Titchmarsh solution $\Psi_1(r,k^2)$ associated to $\mathcal{H}_1$
is of the form \begin{align*}
\Psi_1(r,k^2) = k^{-1/2}e^{ikr}\sigma_{1,\infty}\left(kr,r\right), \qquad kr\gtrsim 1.
\end{align*}
where $\sigma_{1,\infty}(q,r)$ is smooth in $q\gtrsim1$, 
and admits the asymptotic power series approximation 
\begin{align}\label{lemWeyl1exp}
\sigma_{1,\infty}(q,r)\approx \sum_{j=0}^\infty q^{-j} \Psi_{1,j}(r), 
  \qquad \Psi_{1,0}(r)=1, \qquad \Psi_{1,1}(r)=-\dfrac{9i}{8}+O(r^{-3}).
\end{align}
as $r\to\infty$, in the sense that 
\[
\sup_{r>0}\left\vert (r\partial_r)^\alpha (q \partial_q)^{\beta}
\Big(\sigma_{1,\infty}(q,r)-\sum_{j=0}^{j_0}q^{-j} \Psi_{1,j}(r)\Big) \right\vert \lesssim_{j_0,\alpha,\beta} q^{-j_0-1}. 
\]
Moreover, for all $j\geq1$, the coefficient functions $\Psi_{1,j}(r)$ are zero-order symbols, that is, 
\begin{align}\label{Psi1Op1_1}
\sup_{r>0}\big\vert (r\partial_r)^k \Psi_{1,j}(r) \big\vert < + \infty,
\end{align}
for all $k\in\N$, and they are analytic at infinity.
\end{lem}

\begin{proof}
As before, we make the ansatz  
$\Psi_1(r,k^2)= k^{-1/2}e^{ikr}\sigma_{1,\infty}\big(kr,r\big),$
and look for $\sigma_{1,\infty}(kr,r)$ satisfying 
the conjugated equation 
\[
\Big(-\partial_r^2-2ik\partial_r+\dfrac{3}{4r^2}-(3-3U^2(r))\Big)\sigma_{1,\infty}(kr,r)=0.
\]
To solve the last equation we make the power series ansatz 
\[
\sigma_{1,\infty}(kr,r)=\sum_{j=0}^\infty k^{-j}f_j(r).
\]
and substituting it into the conjugated equation we obtain 
\begin{align}\label{lemWeyl1conj}
-2if_0'(r)+\sum_{j=0}^\infty \Big(-2if_{j+1}'+\left(-f_j''+\tfrac{3}{4r^2}f_j-(3-3U^2)f_j\right)\Big)k^{-j}=0,
\end{align}
hence,  
\begin{align*}
f_0(r)=1 \qquad \hbox{and}\qquad 2if_{j+1}'=-f_{j}''+\dfrac{3}{4r^2}f_j-(3-3U^2)f_j.
\end{align*}
Integrating the last last equation from $r$ to $\infty$, 
assuming $f_j(r)\to0$ as $r\to+\infty$ for all $j\geq 1$, we get 
\begin{align}\label{lemWeyl1prrec}
f_{j}(r)=\dfrac{i}{2}f'_{j-1}(r)+\dfrac{i}{2}\int_r^\infty \Big(\tfrac{3}{4s^2}f_{j-1}(s)-\big(3-3U^2(s)\big)f_{j-1}(s)\Big)ds.
\end{align}
Observe that, in particular, 
\begin{align}\label{lemWeyl1_f1}
f_1(r)&=\dfrac{3 i}{8r}-\dfrac{i}{2}\int_r^\infty \big(3-3U^2(s)\big)ds =-\dfrac{9i}{8r}+O(r^{-3}),
\end{align}
where we have used the fact that $3-3U^2(r)= \tfrac{3}{r^2}+O(r^{-4})$ for $r\gg1$.
Based on \eqref{lemWeyl1prrec} one can see that the proofs of \eqref{lemWeyl1exp} and \eqref{Psi1Op1_1} 
can proceed identically to the proofs of \eqref{Fou22} and \eqref{Fou23} (see the analogous recursion
\eqref{lemWeylprrec})
and the proof of Proposition 5.6 of \cite{KST}, using again the estimates \eqref{der1-U^2} for the potential.

\end{proof}

\subsection{The spectral measure}
We now derive asymptotics for the spectral measure associated to $\mathcal{L}_1$.
The main difficulty is the oscillatory nature as $r \rightarrow \infty$
of the $L^2((0,1))$ Weyl solution constructed in Lemma \ref{Phi1Op1_1}.
The following proposition is the main result of this section.

\begin{prop}\label{SMOp1_1}
The generalized eigenfunctions associated to $\mathcal{H}_1$ are given by
\begin{align}\label{SMOp1_2}
\Phi_1(r,k^2) = 2\Re (a_1(k^2) \Psi_1(r,k^2))
\end{align}
where $\Psi_1$ is the Weyl solution at infinity from Lemma \ref{lemWeyl1}, and $a_1(k^2)$ is smooth,
always non-zero, and satisfies
\begin{align}\label{SMOp1_3}
a_1(k^2) \approx \jk^{-1}.
\end{align}
Moreover, $a_1(k^2)$ satisfies the symbol type bounds
\begin{align}\label{SMOp1_13}
\vert (k\partial_k)^na_1(k^2)\vert \lesssim 
  \jk^{-1}. 
\end{align}
In particular, the spectral measure associated to $\mathcal{H}_1$ is given by
\begin{align}\label{SMOp1_4}
\frac{d\rho_1}{dk}(k^2) = \frac{1}{4\pi |a_1(k^2)|^2} \approx \langle k^2\rangle,
\end{align}
with consistent symbol-type upperbounds on the derivatives.
\end{prop}

\begin{proof}
As in Section \ref{sec_specm_2}, from the definition of $a_1(k^2)$, and Lemma \ref{lemWeyl1} at infinity, 
one has that $a_1(k^2)$ is given by the following Wronskian:
\begin{align}\label{SMOp1_15}
a_1(k^2) =\dfrac{W(\Phi_1(\cdot,k^2),\overline{\Psi_1(\cdot,k^2)})}{W(\Psi_1(\cdot,k^2),\overline{\Psi_1(\cdot,k^2)})}= \frac{i}{2} W(\Phi_1(\cdot,k^2),\overline{\Psi_1(\cdot,k^2)}).
\end{align}

\medskip
{\it Upper bounds}. 
As in Lemma \ref{propSM2}, to upper bound $a_1$ 
we evaluate this Wronskian in the region where we have asymptotic expansions for both $\Phi_1$ and $\Psi_1$,
that is, when $kr \approx 1$. 
From Lemma \ref{lemWeyl1} we obtain that, for $kr = c \approx 1$,
\begin{align}\label{SMOp1_5}
|\Psi_1(ck^{-1},k^2)| \approx k^{-1/2},   
\qquad |\partial_r \Psi_1(ck^{-1},k^2)| \approx k^{1/2},   
\qquad k > 0.
\end{align}
%
From the expansion \eqref{Phi1Op1_2}, for $c \approx 1$, we see that
\begin{align}\label{SMOp1_6}
\Phi_1(ck^{-1},k^2) & = \Phi_1^{(0)}(ck^{-1}) + \sqrt{c}k^{-1/2} 
  \sum_{j\geq 1} c^{2j} \Phi_j(ck^{-1}).
\end{align}
Therefore, when $k \gtrsim 1$, using the asymptotics for small $r$ in \eqref{Phi1_0} 
and \eqref{Phi1Op1_2}-\eqref{Phi1Op1_3}, we have
\begin{align}\label{SMOp1_7}
\Phi_1(ck^{-1},k^2) & \lesssim k^{-3/2}, \quad k \gtrsim 1.
\end{align}
When instead $k \lesssim 1$, we use the asymptotics for large $r$ in \eqref{Phi1_0} 
and \eqref{Phi1Op1_2}-\eqref{Phi1Op1_8}, and see that
\begin{align}\label{SMOp1_8}
\Phi_1(ck^{-1},k^2) & \lesssim k^{-1/2} 
	, \quad k \lesssim 1.
\end{align}
Using \eqref{Phi1Op1_4} we infer that consistent bounds hold for the derivatives:
\begin{align}\label{SMOp1_9}
\partial_r \Phi_1(ck^{-1},k^2) \lesssim \left\{ 
\begin{array}{ll}
k^{1/2}, &  k \lesssim 1, 
\\  k^{-1/2},  & k \gtrsim 1.
\end{array} 
\right.
\end{align}
Plugging the above estimates into the right-hand side of \eqref{SMOp1_15} we can upper bound 
\begin{align}\label{SMOp1_10}
|a_1(k^2)| = \frac{1}{2}\big| \Phi_1(\cdot,k^2) \partial_r \overline{\Psi_1(\cdot,k^2)} 
  - \partial_r \Phi_1(\cdot,k^2) \overline{\Psi_1(\cdot,k^2)} \big| 
   \lesssim \left\{ 
\begin{array}{ll}
1, &  k \lesssim 1, 
\\
k^{-1},  & k \gtrsim 1.
\end{array} 
\right.
\end{align}
This is consistent with \eqref{SMOp1_3}. 
%

\medskip
{\it Lower bounds}. 
Now we seek to lower bound $a_1$. 
This case is more subtle 
than that of $a_2$ due to the fact the the leading order terms in $\Phi_1^{(0)}(r)$ and $\Phi_1(r,k^2)$ 
may vanish. First of all 
from \eqref{SMOp1_2} we infer that, for all $r>0$,
\begin{align}\label{SMOp1_11}
|a_1(k^2)| \geq \frac{|\Phi_1(r,k^2)|}{2|\Psi_1(r,k^2)|} \gtrsim k^{1/2} |\Phi_1(r,k^2)|,
\end{align}
having used \eqref{SMOp1_5}. 
Then, recall that from Lemma \ref{lemWeyl1}
calculating at $r=\infty$, we have $W(\Psi_1,\overline{\Psi_1})=-2i$,
which in turn implies $\mathrm{Im}(\Psi_1(r,k^2)\partial_r\overline{\Psi_1(r,k^2)} )=-1$;
combining this last identity with the fact that $\Phi_1(r,k^2)$ is real-valued, we obtain that 
\begin{align} \label{SMOp1_16}
&\mathrm{Im}\big[\partial_r\Psi_1(r,k^2)W\big(\Phi_1(r,k^2),\overline{\Psi_1(r,k^2)}\big)\big]
= -\partial_r\Phi_1(r,k^2).
\end{align}
Going back to the definition of $a_1(k^2)$ 
taking the Wrosnkian of  \eqref{SMOp1_2} against $\overline{\Psi_1(r,k^2)}$ 
and then multiplying the resulting equation by $\partial_r\Psi_1(r,k^2)$, we obtain that 
\[
\vert a_1(k^2)\vert \geq \dfrac{\vert \partial_r\Phi_1(r,k^2)\vert}{2\vert \partial_r\Psi_1(r,k^2)\vert},
\]
which in turn implies that 
\begin{align}\label{SMOp1_12}
\vert a_1(k^2)\vert \geq \frac{|\Phi_1(r,k^2)|}{2|\Psi_1(r,k^2)|} 
  \vee \dfrac{\vert \partial_r\Phi_1(r,k^2)\vert}{2\vert \partial_r\Psi_1(r,k^2)\vert} \gtrsim k^{1/2} |\Phi_1(r,k^2)| \vee    k^{-1/2}\vert \partial_r\Phi_1(r,k^2)\vert
\end{align}
having used Lemma \ref{lemWeyl1} again.  

\medskip
\noindent
{\it Case 1: Large frequencies}.
For $k \gtrsim 1$ and $r = ck^{-1}$, with a suitable small but fixed constant $c<1$ to be chosen, 
we can proceed as in Lemma \ref{propSM2}:
we use \eqref{SMOp1_6} followed by \eqref{asymp_phi_1_0} and \eqref{Phi1Op1_3}, to see that 
\begin{align*}
|\Phi_1(ck^{-1},k^2)| & = \Big| \Phi_1^{(0)}(ck^{-1}) 
  + \sqrt{c}k^{-1/2} \sum_{j\geq 1} c^{2j} \Phi_{1,j}(ck^{-1}) \Big|
  \\
  & \geq \frac{1}{10} c^{3/2} k^{-3/2}
  - \sqrt{c} k^{-1/2} \sum_{j\geq 1} c^{2j} \frac{C^j}{2^{2j}j!(j+1)!} c k^{-1}
  \\
  & \geq \frac{1}{10} c^{3/2} k^{-3/2}
  - \dfrac{1}{8}c^{3/2} k^{-3/2} \big( e^{c^2C} -1 \big) \gtrsim k^{-3/2},
\end{align*}
provided $c$ is small enough depending on $C$.
Inserting this bound into \eqref{SMOp1_11} gives $|a_1(k^2)| \gtrsim k^{-1}$ when $k \gtrsim 1$,
which is consistent with \eqref{SMOp1_3}.

\medskip
\noindent
{\it Case 2: Small frequencies}.
In the case $k \ll c<1$ with $c$ fixed, and $r = ck^{-1} \gg 1$, 
we need to be careful with the oscillations of $\Phi_1(r,k^2)$ 
and exploit \eqref{SMOp1_12}, which we use in the form
\begin{align}\label{SMOp1_12'}
\vert a_1(k^2)\vert 
  \gtrsim (rk)^{1/2} \big| \frac{1}{\sqrt{r}} \Phi_1(r,k^2) \big| 
     + (rk)^{-1/2} \big| \sqrt{r} \partial_r\Phi_1(r,k^2) \big|;
\end{align}
the idea is to use that, if the leading order term in $\Phi_1(r,k^2)$ is equal to zero, 
then, the leading order term in $\partial_r\Phi_1(r,k^2)$ is not.

First, observe that, 
from \eqref{Phi1Op1_2}, \eqref{Phi1Op1_8} as well as the fast decay \eqref{Phi1Op1_9} for the coefficients, 
when $r=ck^{-1}\gg1$, we can write 
\begin{align}\label{c1c2til}
\begin{split}
& \frac{1}{\sqrt{r}} \Phi_1(r,k^2) = \frac{1}{\sqrt{r}} \Phi_1^{(0)}(r) + 
  \sum_{j\geq1}c^{2j}\Phi_{1,j}(r)
\\ & \qquad \qquad \qquad = 
  C_1 \cos\big(\sqrt{2}\ln(r)\big) + 
  C_2 \sin\big(\sqrt{2}\ln(r)\big) + O(r^{-2}),
\\
& 
  C_1 := c_1 + \sum_{j=1}^\infty c^{2j}c_{1,j},
  \qquad 
  C_2 := c_2 + \sum_{j=1}^\infty c^{2j}c_{2,j}.
\end{split}
\end{align}
Note that $C_i = c_i + O(c^2)$, $i=1,2$, and 
so these constants can be made arbitrarily close to $c_1$ and $c_2$, respectively, 
by making $c<1$ small enough. 

Second, we claim that the following differentiated version of \eqref{Phi1_0} (see also \eqref{asymp_phi_1_0}) holds:
\begin{align}\label{drPhi10}
\begin{split}
\sqrt{r} \partial_r\Phi_1^{(0)}(r) = \dfrac{1}{2
  }\big(c_1+2\sqrt{2}c_2\big)\cos\big(\sqrt{2}\ln(r)\big)
  + \dfrac{1}{2
  }\big(c_2-2\sqrt{2}c_1\big)\sin\big(\sqrt{2}\ln(r)\big) 
	+ \varepsilon'(r), \\ \varepsilon'(r) = O(r^{-2}), \qquad r \gtrsim 1.
\end{split}
\end{align}
To see this, it suffices to show that $\sqrt{r} \partial_r\varepsilon(r) = O(r^{-2})$
where $\varepsilon$ is the remainder in the asymptotic \eqref{Phi1_0},
that is 
\begin{align*}
\varepsilon := 
c_1\varepsilon_{1,\infty} - \sqrt{2}c_2\varepsilon_{2,\infty};
\end{align*}
cfr. \eqref{Phi1_0pr} and \eqref{FS0infty}
in the proof of Lemma \ref{lem_p1_t1}.
Let us use the same definitions for $\wt{V}$ in \eqref{L1_eqrewri},
so that $\mathcal{H}_1 - 2 = -\partial_r^2 - 9/(4r^2) - \wt{V}$,
and the notation for $(\Phi_{1,\infty}^{(0)},\Theta_{1,\infty}^{(0)})$ and
$(\varepsilon_{1,\infty},\varepsilon_{2,\infty})$ in \eqref{FS0infty};
in particular, recall that $\varepsilon_{i,\infty} = O(r^{-3/2})$, $i=1,2$, and
that $\partial_r^\alpha \wt{V} = O(r^{-4-\alpha})$, $\alpha=0,1$, for $r \gtrsim 1$.
Then, since we have
\begin{align*}
\big( -\partial_r^2 - \frac{9}{4r^2} \big) \varepsilon_{1,\infty} & = 
 \wt{V}(r) \big(\sqrt{r}\cos\big(\sqrt{2}\ln(r) +  \varepsilon_{1,\infty}\big),
\end{align*}
it follows that
\begin{align*}
\big( -\partial_r^2 - \frac{9}{4r^2} \big) \partial_r \varepsilon_{1,\infty} & = 
 \wt{V}(r) \partial_r \varepsilon_{1,\infty} + O(r^{-9/2}).
\end{align*}
From this last equation we can then use a fixed point argument almost identical 
to the one performed after \eqref{FS0infty}
to see that $\partial_r \varepsilon_{1,\infty} = O(r^{-5/2})$. 
The same argument applies to $\partial_r \varepsilon_{2,\infty}$,
and we conclude the validity of \eqref{drPhi10}.

Next, differentiating \eqref{Phi1Op1_2} we can write
\begin{align}\label{xyz}
\begin{split}
\sqrt{r} \partial_r\Phi_1(r,k^2) & = \sqrt{r} \partial_r \Phi_1^{(0)}(r) 
	+ \sum_{j\geq1}(rk)^{2j} r\partial_r\Phi_{1,j}(r) + R(r,k),
\\
R(r,k) & := \sum_{j\geq1}(2j+1/2)(rk)^{2j}\Phi_{1,j}(r).
\end{split}
\end{align}
Upon evaluating at $rk = c \ll 1$, using \eqref{Phi1Op1_8} we see that $R$ is a lower order term
satisfying
\begin{align}\label{xyzR}
\begin{split}
& R(r,cr^{-1}) 
  = c_{1,R} \cos(\sqrt{2}\ln(r)) + c_{2,R} \sin(\sqrt{2}\ln(r)) + O(r^{-2}),
  \qquad c_{1,R}, \, c_{2,R} = O(c^2).
\end{split}
\end{align}
Then, using \eqref{Phi1Op1_12} and \eqref{drPhi10} to evaluate the leading order term in \eqref{xyz}, 
  we see that,
for $rk = c$,
\begin{align}\label{xyz'}
\begin{split}
\sqrt{r} \partial_r\Phi_1(r,k^2) & = A \cos\big(\sqrt{2}\ln(r)\big) + B \sin\big(\sqrt{2}\ln(r)\big) 
  + O(r^{-2}),
\end{split}
\end{align}
where we have defined
\begin{align}\label{xyzconstants}
\begin{split}
A & := \dfrac{1}{2}(c_1+2\sqrt{2}c_2) + \sum_{j\geq 1} c^2 \widetilde{c}_{1,j,1}
  + c_{1,R} = \dfrac{1}{2}(c_1+2\sqrt{2}c_2) + O(c^2),
\\
B & := \dfrac{1}{2}(c_2-2\sqrt{2}c_1) + \sum_{j\geq 1} c^2 \widetilde{c}_{2,j,1} 
  + c_{2,R} = \dfrac{1}{2}(c_2-2\sqrt{2}c_1) + O(c^2),
\end{split}
\end{align}
also owing to \eqref{Phi1Op1_12'}.

Plugging \eqref{c1c2til} and \eqref{xyz'} into \eqref{SMOp1_12'}, and using the shorthand $\sqrt{2}\ln r = x$,
we have 
\begin{align}\label{SMOp1_12''}
\begin{split}
| a_1(k^2) |
  & \gtrsim \big| C_1 
  \cos x + C_2 
  \sin x \big|
  + \big| A \cos x + B \sin x \big| + O(r^{-2})
\end{split}
\end{align}
at $rk = c < 1$;
we then see that, for $r$ large enough, $|a_1(k^2)| \gtrsim 1$,
since the function of $x$ that appears in \eqref{SMOp1_12''} is periodic and never vanishes
in view of the fact that
\begin{align}
C_2 A - C_1 B 
 = \sqrt{2} (c_1^2 + c_2^2) + O(c^2) \neq 0,
\end{align}
provided $c$ is small enough. This concludes the proof of the lower bounds,
and gives \eqref{SMOp1_3}.

\medskip
{\it Proof of \eqref{SMOp1_13}}. 
Finally, to show that similar bounds hold for the derivatives, 
we proceed as in Lemma \ref{propSM2}. 
First, we use Lemma \ref{lemWeyl1} to see that both $\Psi_1(ck^{-1},k^2)$ and $(r\partial_r\Psi_1)(ck^{-1},k^2)$ 
can be expressed as $k^{-1/2}f(k^{-1})$ for some $f(\cdot)$ satisfying the symbol type bounds \[
\vert (r\partial_r)^nf(r)\vert   \approx_n 1.
\]
On the other hand, using Lemma \ref{Phi1Op1_1} we infer that a similar property holds for $\Phi_1(r,k^2)$. 
More specifically, from Lemma \ref{Phi1Op1_1} we see that
both $\Phi_1(ck^{-1},k^2)$ and $(r\partial_r\Phi_1)(ck^{-1},k^2)$ can be written 
as $k^{-1/2}h(k^{-1})$  and $k^{-1/2}g(k^{-1})$, respectively, 
for some $h(\cdot)$ and $g(\cdot)$ satisfying 
\begin{align*}
\hbox{for all } \, r>0, \quad \vert (r\partial_r)^nh(r)\vert \leq \widetilde{c}_n, 
  \qquad \ \hbox{ and } \ \qquad \hbox{for } \, r\lesssim 1, \quad \vert (r\partial_r)^nh(r)\vert \approx_n r ,
\\ 
\hbox{for all } \, r>0, \quad \vert (r\partial_r)^ng(r)\vert \leq \widetilde{c}_n, 
  \qquad \ \hbox{ and } \ \qquad \hbox{for } \, r\lesssim 1, \quad \vert (r\partial_r)^ng(r)\vert \approx_n r .
\end{align*}
Therefore, from \eqref{SMOp1_15} and the above analysis we conclude that $a_1(k^2)$ 
can be written as a linear combination of functions of the form $ f(k^{-1})h(k^{-1})$ and $f(k^{-1})g(k^{-1})$, 
from where we conclude the claimed upper bounds on the derivatives of $a_1(k^2)$.
\end{proof}

\begin{rem}[The case $\vert n\vert\geq 2$]\label{section3_nvortex}
Now we briefly show some of the calculations needed for the vortex of degree $n$. 
Let us discuss Lemma \ref{Phi1Op1_1} first. In the proof of \eqref{Phi1Op1_10},  
for $r\lesssim1$, the kernel in \eqref{Phi1Op1_7} now writes \[
K_n(r,s)=\tfrac{s}{2n}(rs)^{-n}(r^{2n}-s^{2n})\Big(1+O(r^2+s^2)\Big),
\]
and $f_{0,n}(r)=r^{-1/2}\Phi_{1,n}^{(0)}(r)$ obeys the bound $f_{0,n} = r^n+O(r^{n+2})$ for $r\lesssim1$,
consistently with the fundamental set of solutions already shown in Remark \ref{section1_nvortex}. 
Hence, the calculation of $f_1$ in \eqref{Phi1Op1_f1_rsmall} is now replaced by
\begin{align}\label{check_nvor_proof4}
f_{1,n}(r) & = \int_0^r K_n(r,s)f_{0,n}(s)ds = \frac{1}{4+4n} r^{2+n} + O(r^{4+n}).
\end{align}
On the other hand, for $r>s\gtrsim1$, the kernel is now replaced by  $K_n=\tfrac{s}{\sqrt{2}}\sin\big(n\sqrt{2}\ln(\tfrac{r}{s})\big) +O(s^{-1})$, 
and hence the calculation in \eqref{Phi1Op1_f1_rlarge} needs to be replaced by
\begin{align}\label{check_nvor_proof1}
& \int_{r_0}^r K_n(r,s)f_0(s)\, ds
 \\ &  \qquad  = \dfrac{n}{4\sqrt{2}(1+2n^2)}\Big((\sqrt{2}c_1-2c_2n)\cos(n\sqrt{2}\ln r)+(2c_1n+\sqrt{2}c_2)\sin(n\sqrt{2}\ln r)\Big)r^2 +O(1). \nonumber
\end{align}
The remaining estimates are adapted similarly so, for the sake of brevity, we omit them.

Regarding  the region $rk\gtrsim 1$, the proof of Lemma \ref{lemWeyl1} needs some trivial modifications;
in particular, \eqref{lemWeyl1conj} needs to be replaced by \[
-2if_0'(r)+\sum_{j=0}^\infty \Big(-2if_{j+1}'+\left(-f_j''+\dfrac{1}{r^2}\big(n^2-\tfrac{1}{4}\big)f_j-(3-3U_n^2)f_j\right)\Big)k^{-j}=0.
\]
yielding (compare with \eqref{lemWeyl1_f1}), \[
f_{1,n}(r)=\dfrac{(n^2-\tfrac{1}{4}) i}{2r}-\dfrac{i}{2}\int_r^\infty \big(3-3U_n^2(s)\big)ds =-\dfrac{(1+8n^2)i}{8r}+O(r^{-3}).
\]
Finally, regarding the spectral measure (Lemma \ref{SMOp1_1}), first of all note that in the present case we have (compare with \eqref{SMOp1_7}-\eqref{SMOp1_9}) \begin{align}\
\Phi_{1,n}(ck^{-1},k^2)\lesssim \begin{cases}
k^{-1/2}, & k\lesssim 1,
\\ k^{-1/2-n}, & k\gtrsim 1,
\end{cases}  \quad \hbox{ and } \quad \partial_r\Phi_{1,n}(ck^{-1},k^2)\lesssim \begin{cases}
k^{1/2}, & k\lesssim 1,
\\ k^{1/2-n} & k\gtrsim 1,
\end{cases}\label{check_nvor_proof5}
\end{align}
and hence \eqref{SMOp1_10} for $\vert n\vert\geq2$ becomes \begin{align}\label{check_nvor_proof6}
\vert a_{1,n}(k^2)\vert \lesssim \begin{cases}
1, & k\lesssim 1,
\\ k^{-n}, & k\gtrsim 1,
\end{cases}
\end{align}
yielding the upper bound for $\vert a_{1,n}\vert$. 
The lower bound is adapted similarly and yields the same result. 
Therefore, the spectral measure satisfies $\rho_{1,n}'(k^2)\approx \langle k^2\rangle^n$ (compare with \eqref{SMOp1_4}).
\end{rem}

\subsection{Fourier transform associated to $\mathcal{L}_1$}\label{ssecFT1}
Let $\Phi_1(r,k^2)$ denote the eigenfunctions associated to $\mathcal{H}_1$ constructed above,
and recall that $\mathcal{H}_1(\sqrt{r}f) = \sqrt{r}\mathcal{L}_1f$, for all $f \in \mathrm{Dom}(\mathcal{L}_1)$.
From the general theory of \cite{GZ}, at least formally, for nice enough $f = P_c^1f$ we have the Fourier representation
\begin{align}\label{FT1}
\sqrt{r} f(r) = \int_{0}^\infty \Phi_1(r,\xi) \wtF_1 (\sqrt{r}f ) (\sqrt{\xi}) \, \rho_1(d\xi), \qquad  
  \wtF_1 (\sqrt{r}g)(\sqrt{\xi}) := \int_{0}^\infty \Phi_1(r,\xi) \sqrt{r} g(r) \, dr,
\end{align}
(here $\xi$ plays the role of our $k^2$ generalized eigenvalue)
and the diagonalization property
\begin{align}\label{FT1diag0}
\wtF_1( n(\mathcal{H}_1) g )(k) = n(k^2) (\wtF_1 g)(k)
\end{align}

A rigorous convergence result for the integrals above is the following:

\begin{prop}\label{propFT1}
With the definitions in \eqref{FT1}, let $f \in L^2 (rdr) \cap C^2((0,\infty))$ such that
\begin{align*}
\int_0^\infty \big( r^{3/2} |f''(r)| + r^{1/2} |f'(r)| + r^{-1/2}|f(r)| ) \, dr = M < \infty.
\end{align*}
Then, the limit
\begin{align*}
\wtF_2 (\sqrt{r}f) (\sqrt{\xi}) := \lim_{R \rightarrow \infty} \int_0^R \Phi_2(r,\xi) \sqrt{r} f(r) \, dr
\end{align*}
exists for all $\xi > 0$ and satisfies
\begin{align*}
\int_{0}^\infty |\Phi_2(r,\xi)| |\wtF_2 (\sqrt{r}f) (\sqrt{\xi})| \, \rho_2(d\xi) \lesssim M.
\end{align*}
\end{prop}

We omit the proof since it is completely analogous to the proof of Proposition \ref{propFT2}, 
given that the generalized eigenfunctions 
$\Phi_1$ satisfy properties identical  to those used there for the generalized eigenfunctions $\Phi_2$;
see Lemma \ref{lemWeyl1} and Lemma \ref{Phi1Op1_1}, 
noticing that the additional logarithmic oscillations for $r\gtrsim 1$ (with $rk \lesssim 1$)
play no role for the relevant estimates in this argument.



\medskip
\section{Decay estimates}\label{secdecay}
In this section we establish 
dispersive decay estimates for linear solutions of 
the second-order in time evolution problems associated with the operators $\mathcal{L}_1$ and $\mathcal{L}_2$.
Consider the initial value problems, for $j=1,2$,
\begin{align}\label{KGj}
\partial_t^2 v_j + \mathcal{L}_j v_j = 0, \qquad v_j(0) = f, \quad \partial_t v_j(0) = g,
\end{align}
with $(v_1(0),\partial_tv_1(0)) = (P_c^1 v_1(0), P_c^1 \partial_tv_1(0))$; 
we will often omit the indexes $j=1,2$ for convenience and just denote $v=v_j$. Let $\wtF_j$, $j=1,2$, denote the Fourier transform associated with the operator $\mathcal{H}_j$, 
constructed in Sections \ref{secL2} and \ref{secL1}, and recall that
$\mathcal{H}_j \sqrt{r} = \sqrt{r} \mathcal{L}_j$. 
After applying the Fourier transform, in terms of the variable 
\begin{align*}
\wt{z_j} := \wtF_j (\sqrt{r} v_j), \qquad z_j := \sqrt{r}v_j, 
\end{align*}
we write \eqref{KGj} as
\begin{align}\label{FTKG1}
\partial_t^2 \wt{z_1} + (k^2+2) \wt{z_1} = 0,  \qquad \wt{z_1}(0) = \wtF_1(\sqrt{r}f), 
  \quad \partial_t\wt{z_1}(0) = \wtF_1(\sqrt{r}g),
\end{align}
when $j=1$, and as 
\begin{align}\label{FTKG2}
\partial_t^2 \wt{z_2} + k^2 \wt{z_2} = 0,  \qquad \wt{z_2}(0) = \wtF_2(\sqrt{r}f), 
  \quad \partial_t\wt{z_2}(0) = \wtF_2(\sqrt{r}g),
\end{align}
when $j=2$.
Note that \eqref{FTKG1} is a Klein-Gordon type evolution, while \eqref{FTKG2} is wave-like.

\subsection{Littlewood-Paley projections}
We define the Littlewood-Paley `projection' to frequencies $k \approx 2^\ell$, $\ell \in \Z$
relative to the operator $\mathcal{H}_j$, $j=1,2$, as follows:
let $f = f(r) \in L^2$, define the projection $P_\ell^{\mathcal{H}_j} f$ by 
\begin{align}\label{LPF1}
\begin{split}
\wtF_j (P_\ell^{\mathcal{H}_j} f) (k) & = \varphi_\ell(k) (\wtF_j f)(k)
 = \varphi_\ell(k) \int_0^\infty \Phi_j(r,k^2) f(r) \, dr 
\end{split}
\end{align}
that is,
\begin{align}\label{LPF1'}
\begin{split}
& P_\ell^{\mathcal{H}_j} f (r) = \int_0^\infty K_{j,\ell}(r,s) f(s) \, ds,
\\
& K_{j,\ell}(r,s) := \int_0^\infty \Phi_j(r,k^2) \Phi_j(s,k^2)  \, \varphi_\ell(k) \, 2k \rho_j'(k^2) \, dk,
\end{split}
\end{align}
recall the notation from Subsection \ref{secnot}.
We have the usual partition of unity property
\begin{align*}
\sum_{\ell \in \Z} P_\ell^{\mathcal{H}_j} f = f \quad \mbox{a.e.},
\end{align*}
as well as orthogonality: $P_\ell^{\mathcal{H}_j} P_{\ell'}^{\mathcal{H}_j} = 0$ if $|\ell - \ell' | > 1$.
Recalling the definition of the cutoff $\varphi_I$, $I\subset[0,\infty]$, from Subsection \ref{secnot},
we similarly define
\begin{align}\label{LPF1'I}
\begin{split}
& P_I^{\mathcal{H}_j} f (r) = \int_0^\infty K_{j,I}(r,s) f(s) \, ds,
\\
& K_{j,I}(r,s) := \int_0^\infty \Phi_j(r,k^2) \Phi_j(s,k^2)  \, \varphi_I(k) \, 2k \rho_j'(k^2) \, dk.
\end{split}
\end{align}

We also define a `projection' relative to the operators $\mathcal{L}_j$, $j=1,2$, 
as follows:
let $h = h(r) \in L^2$, define 
\begin{align}\label{LPF1L}
\begin{split}
P_\ell^{\mathcal{L}_j} f (r) = r^{-1/2} P_\ell^{\mathcal{H}_j} (r^{1/2} f ),
\end{split}
\end{align}
that is,
\begin{align}\label{LPF1'L}
\begin{split}
& P_\ell^{\mathcal{L}_j} f (r) := \int_0^\infty \frac{1}{\sqrt{rs}} K_{j,\ell}(r,s) f(s) \, s  ds,
\end{split}
\end{align}
where the kernel $K_{j,\ell}$ is the same as in \eqref{LPF1'}.
As before we have 
\begin{align}
\sum_{\ell \in \Z} P_\ell^{\mathcal{L}_j} f = f \quad \mbox{a.e.}
\end{align}
and the orthogonality property $P_\ell^{\mathcal{L}_j} P_{\ell'}^{\mathcal{L}_j} = 0$ if $|\ell - \ell' | > 1$.
Similarly to the above, we also define, for $I\subset[0,\infty]$ 
\begin{align*}P_I^{\mathcal{L}_j} f (r) := r^{-1/2} P_I^{\mathcal{H}_j} (r^{1/2} f )
  = \int_0^\infty \frac{1}{\sqrt{rs}} K_{j,I}(r,s) f(s) \, s  ds.
\end{align*}
In Subsection \ref{ssecLP} we prove $L^p$ bounds for these projections.



\subsection{Second order dynamics for $\mathcal{L}_1$}\label{ssecKG}
Writing out the solution to \eqref{FTKG1}, and using the Fourier representation 
from Proposition \ref{propFT1}, 
we have 
\begin{align*}
\begin{split}
\wt{z_1}(t,k) & = e^{it\sqrt{k^2+2}} \frac{1}{2}\Big( \wtF_1(\sqrt{r}f)(k) + \frac{1}{i\sqrt{k^2+2}}\wtF_1(\sqrt{r}g)(k) \Big)
\\
& + e^{-it\sqrt{k^2+2}} \frac{1}{2}\Big( \wtF_1(\sqrt{r}f)(k) - \frac{1}{i\sqrt{k^2+2}}\wtF_1(\sqrt{r}g)(k) \Big)
\\
& = e^{it\sqrt{k^2+2}} \frac{1}{2} \int_{0}^\infty \Phi_1(r,k^2) \Big( \sqrt{r}f + \frac{1}{i\sqrt{k^2+2}} \sqrt{r}g \Big) \, dr
\\
& + e^{-it\sqrt{k^2+2}} \frac{1}{2} \int_{0}^\infty \Phi_1(r,k^2) \Big( \sqrt{r}f - \frac{1}{i\sqrt{k^2+2}} \sqrt{r}g \Big) \, dr
\end{split}
\end{align*}
so that we can represent the solution of \eqref{KGj} as
\begin{align}\label{ssecKG1}
\begin{split}
& v(t,r) = \frac{1}{\sqrt{r}} 
  \int_0^\infty \Phi_1(r,k^2) e^{it\sqrt{k^2+2}} \wt{F_+}(k) \, 2k \rho_1^\prime(k^2) \, dk
  \\
  & \qquad \qquad + \frac{1}{\sqrt{r}} 
  \int_0^\infty \Phi_1(r,k^2) e^{-it\sqrt{k^2+2}} \wt{F_-}(k) \, 2k \rho_1^\prime(k^2) \, dk,
  \\
& \wt{F_{\pm}}(k) := \frac{1}{2}\Big( \wtF_1(\sqrt{r}f)(k) \pm \frac{1}{i\sqrt{k^2+2}}\wtF_1(\sqrt{r}g)(k) \Big)
  = \frac{1}{2}\wt{\mathcal{F}_1}\Big( \sqrt{r}f \pm \frac{1}{i\sqrt{\mathcal{H}_1}} \sqrt{r}g \Big)(k) ,
\end{split}
\end{align}
and as
\begin{align}\label{ssecKG1'}
\begin{split}
v(t,r) & = \int_0^\infty 
  \Big(K_+(t,r,s) f(s) + K_+'(t,r,s) g(s)\Big) \, s \, ds 
  \\
  & + \int_0^\infty 
  \Big(K_-(t,r,s) f(s) - K_-'(t,r,s) g(s)\Big) \, s \, ds 
\end{split}
\end{align}
where the kernels are defined by
\begin{align}\label{Ker}
K_\pm(t,r,s) & :=  \frac{1}{\sqrt{rs}} 
  \int_0^\infty \Phi_1(r,k^2) \Phi_1(s,k^2) e^{\pm it\sqrt{k^2+2}} \, k \rho_1^\prime(k^2) \, dk,
\\
\label{Ker'}
K_\pm'(t,r,s) & := \frac{1}{\sqrt{rs}} \int_0^\infty
  \frac{1}{i\sqrt{k^2+2}} \Phi_1(r,k^2) \Phi_1(s,k^2) e^{\pm it\sqrt{k^2+2}} \, k \rho_1^\prime(k^2) \, dk.
\end{align}
Recall that $\rho_1'$ is given in Proposition \ref{SMOp1_1},
and that we can effectively think of $\rho_1'(k^2) \approx 
\langle k \rangle^2$.

We are going to use both types of representations, namely \eqref{ssecKG1} and \eqref{ssecKG1'},
to establish different types of dispersive estimates. The former representation 
will be used to obtain estimates in terms of the transform of the data in Section \ref{WDecKG_Sec}; 
this type of estimate is often more suitable to tackle nonlinear problems where oscillations 
- which are more naturally seen in Fourier space - play an important role.
The second representation is more convenient to obtain physical space type estimates, which we do in this section.

According to the notation above we have
\begin{align}\label{ssecKG1'l}
\begin{split}
P_\ell^{\mathcal{L}_1} v(t,r) & = \int_0^\infty 
  \Big(K_{+,\ell}(t,r,s) P_{[\ell-2,\ell+2]}^{\mathcal{L}_1} f(s) + K_{+,\ell}'(t,r,s) P_{[\ell-2,\ell+2]}^{\mathcal{L}_1} g(s)\Big) \, s \, ds 
  \\
  & + \int_0^\infty 
  \Big(K_{-,\ell}(t,r,s) P_{[\ell-2,\ell+2]}^{\mathcal{L}_1} f(s) - K_{-,\ell}'(t,r,s) P_{[\ell-2,\ell+2]}^{\mathcal{L}_1} g(s)\Big) \, s \, ds 
\end{split}
\end{align}
where the kernels with the index $\ell$, that is, $K_{\pm,\ell}$ and $K_{+,\ell}'$ 
are defined as the kernel $K_{\pm}$ and $K_{\pm}'$ in \eqref{ssecKG1'}
with an additional $\varphi_\ell(k)$ cutoff in the integrals.
Let us verify the identity \eqref{ssecKG1'l}; it suffices to do this for the kernel $K_+$
since the others obey the same formulas.
We have, according to our Fourier representation, that
\begin{align*}
\int_0^\infty K_+(t,r,s) f(s) \, s \, ds  = \frac{1}{\sqrt{r}} e^{it \sqrt{\mathcal{H}_1+2}} (\sqrt{r} f).
\end{align*}
Applying the projection with respect to $\mathcal{L}_1$, according to the definition \eqref{LPF1L}, we get
\begin{align}\label{ssecKG1'lpr}
\begin{split}
& P_\ell^{\mathcal{L}_1} \big[ \frac{1}{\sqrt{r}} e^{it \sqrt{\mathcal{H}_1+2}} (\sqrt{r} f) \big]
  =  \frac{1}{\sqrt{r}} P_\ell^{\mathcal{H}_1} \big[e^{it \sqrt{\mathcal{H}_1+2}} (\sqrt{r} f) \big]
  \\
  & =  \frac{1}{\sqrt{r}} e^{it \sqrt{\mathcal{H}_1+2}} P_\ell^{\mathcal{H}_1} 
    \big[ P_{[\ell-2,\ell+2]}^{\mathcal{H}_1} (\sqrt{r} f) \big]
  =  \frac{1}{\sqrt{r}} e^{it \sqrt{\mathcal{H}_1+2}} P_\ell^{\mathcal{H}_1} 
  \big[ \sqrt{r} \, P_{[\ell-2,\ell+2]}^{\mathcal{L}_1} f \big]
\end{split}
\end{align}
having used the orthogonality property. 
Finally we notice that
\begin{align*}
\frac{1}{\sqrt{r}} e^{it \sqrt{\mathcal{H}_1+2}}  P_\ell^{\mathcal{H}_1} (\sqrt{r} h ) = 
 \int_0^\infty K_{+,\ell}(r,s) h(s) \, s ds;
\end{align*}
this proves \eqref{ssecKG1'l} from \eqref{ssecKG1'}.

\smallskip
\begin{prop}\label{propKG1dec}
Let $v$ be a solution to $v_{tt} + \mathcal{L}_1 v = 0$, with $(v(0),\partial_t v(0)) = (P_c^1f,P_c^1g)$ such that 
\begin{align}\label{propKG1decas}
\sum_{\ell} 
(1+2^{2\ell}) {\big\| P_\ell^{\mathcal{L}_1} f \big\|}_{L^1(rdr)} 
  + (1+2^{\ell}) {\big\| P_\ell^{\mathcal{L}_1} g \big\|}_{L^1(rdr)}
\leq 1.
\end{align}
Then
\begin{align*}
|v(t,r)| \lesssim \frac{1}{t}.
\end{align*}
\end{prop}

\begin{proof}
In view of the representation \eqref{ssecKG1'} and the assumption \eqref{propKG1decas},
it suffices to show the following estimates for the kernels in 
the expressions \eqref{ssecKG1'l}:
\begin{align}\label{prKG1Ker0}
\sup_{r,s\geq 0}\big| K_{\pm,\ell}(t,r,s) \big| \lesssim (1+2^{2\ell}) t^{-1}, 
  \qquad \big| K_{\pm,\ell}'(t,r,s) \big| \lesssim (1+2^{\ell}) t^{-1}.
\end{align}
Since these two estimate are almost identical, we will only give the details for the proof 
of the first one. Also, we only consider the case with index $+$, since the one with index $-$
can be dealt with in an identical fashion.

By symmetry we may assume, without loss of generality, that $r\geq s$.
We let $\chi = \chi(x)$, $x\geq 0$, be a smooth cutoff function equal to $1$ in $[0,1]$, 
decreasing, and vanishing for $x \geq 2$, and denote $\chi^c := 1-\chi$. We split the kernel into three 
main pieces:
\begin{align}
\nonumber
K_{+,\ell}(t,r,s) & = K_{1,\ell}(t,r,s) + K_{2,\ell}(t,r,s) + K_{3,\ell}(t,r,s),
\\
\label{prKG1Ker1}
K_{1,\ell}(t,r,s) & :=  \frac{1}{\sqrt{rs}} 
  \int_0^\infty \Phi_1(r,k^2) \Phi_1(s,k^2) e^{
  it\sqrt{k^2+2}} \, \varphi_\ell(k) k \rho_1^\prime(k^2) \, 
  \chi^c(sk) \, dk,
\\
\label{prKG1Ker2}
K_{2,\ell}(t,r,s) & :=  \frac{1}{\sqrt{rs}} 
  \int_0^\infty \Phi_1(r,k^2) \Phi_1(s,k^2) e^{it\sqrt{k^2+2}} \, \varphi_\ell(k) k \rho_1^\prime(k^2) \, 
  \chi(rk)\chi(sk) \, dk,
\\
\label{prKG1Ker3}
K_{3,\ell}(t,r,s) & :=  \frac{1}{\sqrt{rs}} 
  \int_0^\infty \Phi_1(r,k^2) \Phi_1(s,k^2) e^{it\sqrt{k^2+2}} \, \varphi_\ell(k) k \rho_1^\prime(k^2) \, 
  \chi^c(rk)\chi(sk) \, dk.
\end{align}
Our main aim is to prove 
\begin{align}\label{prKG1main}
\sup_{r,s\geq 0} \big| K_{j,\ell}(r,s) \big| \lesssim (1+2^{2\ell}) t^{-1}, \qquad j=1,2,3.    
\end{align}

Before proceeding to estimate each piece we observe that, from Lemmas \ref{Phi1Op1_1} and \ref{lemWeyl1},
we can write 
%
\begin{align}\label{prKG1Phi}
\Phi_1(r,k^2) = \left\{ 
\begin{array}{ll}
r^{3/2} \big( 1 + \phi_{<}(r,k)), &  rk \lesssim 1, \quad r \lesssim 1,
\\
\\
r^{1/2} \big( 1 + \phi_{>}(r,k)), &  rk \lesssim 1, \quad r \gtrsim 1,
\\
\\
\dfrac{1}{\sqrt{k}} \big[ a_1(k^2) e^{irk} \big( 1 + \phi_+(r,k)) 
  + \overline{a_1}(k^2) e^{-irk} \big(1 + \phi_-(r,k)) \big], &  rk \gtrsim 1
\end{array} 
\right.
\end{align}
where $\phi \in \{ \phi_{<}, \phi_{>}\phi_-,\phi_+ \}$ 
is a generic function satisfying the symbol-type estimates
\begin{align}\label{prKG1phi}
|\partial_k^\alpha \phi(r,k)| \lesssim k^{-\alpha}, \quad \alpha = 0,1.
\end{align}
Also, recall that $a_1(k^2) \approx \langle k \rangle^{-1}$ with compatible symbol-type estimates
for the derivatives. 
To see the validity of \eqref{prKG1Phi} it suffices to inspect the statement 
of Lemma \ref{Phi1Op1_1}, and use Lemma \ref{lemWeyl1}. 


%
%

\medskip
\noindent
{\it Estimate of $K_{1,\ell}$.}
In view of our assumptions, on the support of $K_{1,\ell}$ we have $rk \geq sk \gtrsim 1$.
Using \eqref{prKG1Phi} we see that the desired estimate \eqref{prKG1main} reduces to the same estimate for
the kernels 
\begin{align}
\begin{split}
L_{\eps_1\eps_2}(t,r,s) := 
  \int_0^\infty \frac{1}{\sqrt{rk}} \frac{1}{\sqrt{sk}} a_1^{\eps_1}(k^2) e^{\eps_1 irk} \big( 1 + \phi_{\eps_1}(r,k))
  a_1^{\eps_2}(k^2) e^{\eps_2 isk} \big( 1 + \phi_{\eps_2}(s,k))  
  \\
  \times e^{it\sqrt{k^2+2}} \, \varphi_\ell(k) k \rho_1^\prime(k^2) \, \chi^c(sk) \, dk, \qquad \eps_1,\eps_2 \in \{+,-\},
\end{split}
\end{align}
where we use the notation $f^+=f$ and $f^-=\bar{f}$.
It suffices to look at the case $(\eps_1,\eps_2) = (-,+)$ since the other cases are similar or easier.
Using $|a_1(k^2)|^2 \rho_1^\prime(k^2) = 1/(4\pi)$, see \eqref{SMOp1_4}, we can write $4\pi L_{-+}=L$ with
\begin{align}\label{prKG1L}
\begin{split}
L(t,r,s) & := \int_0^\infty \frac{1}{\sqrt{rk}} \frac{1}{\sqrt{sk}} e^{i t S(r,s,k,t)} 
  A(r,s,k) \,  \varphi_\ell(k) k \, dk,
\\
S(r,s,k,t) & := \frac{k}{t}(s-r) + \sqrt{k^2+2},
\\
A(r,s,k) & := \big( 1 + \phi_-(r,k)) \big( 1 + \phi_{+}(s,k)) \chi^c(sk).
\end{split}
\end{align}
We then aim to prove 
\begin{align}\label{prKG1main1}
\sup_{r,s\geq 0} \big| L(t,r,s) \big| \lesssim (1+ 2^{2\ell}) t^{-1}. 
\end{align}
From Proposition \ref{SMOp1_1} and \eqref{prKG1phi}, we have the symbol-type estimates
\begin{align}\label{prKG15}
|\partial_k^\alpha A(r,s,k)| \lesssim k^{-\alpha}, \qquad \alpha = 0,1,2.
\end{align}


\smallskip
\noindent
{\it Proof of \eqref{prKG1main1}}.
First, observe that we may restrict ourselves to the case $k \geq Ct^{-1/2}$,
that is $2^\ell \gtrsim t^{-1/2}$, otherwise the bound follows easily from \eqref{prKG15} and $rk,sk \geq 1$. 

According to \eqref{prKG1L} we have
\begin{align*}
& \partial_k S(r,s,k,t) = \frac{1}{t}(s-r) + \frac{k}{\sqrt{k^2+2}},
\qquad \partial_k^2 S(r,s,k,t) = (k^2+2)^{-3/2} \approx (1+2^{3\ell})^{-1}.
\end{align*}
In what follows we will often omit some the arguments $(r,s,k,t)$ for ease of notation,
and when this causes no confusion.
The phase $S$ has a unique stationary point, which we denote by $k_0$, whenever $\frac{1}{t}(s-r) \in (-1,0)$. 
In what follows we may assume that the stationary point is in the support 
of the integral for otherwise 
the desired bound can be obtained more easily. 
In particular, we have $k_0 \approx 2^\ell$ and $r \gtrsim t k/\sqrt{k^2+2}$.

Let $q_0$ be the smallest integer such that $2^{q_0} \geq t^{-1/2}(1+2^{3\ell/2})$. 
We decompose the support of the integral according to the distance of $k$ to $k_0$ as follows:
\begin{align}\label{prKGmain1q}
\begin{split}
& L(t,r,s) = \sum_{q \geq q_0}^{\ell+10} L_{q}(t,r,s), 
\\
& L_{q}(t,r,s) := \int_0^\infty \frac{1}{\sqrt{rk}} \frac{1}{\sqrt{sk}} e^{i t S} A(r,s,k)
  \varphi_q^{(q_0)}(k-k_0) \,  \varphi_\ell(k) k \, dk.
\end{split}
\end{align}
Note that the estimate for the term with $q=q_0$ follows by direct integration, 
using also that $\sqrt{k/r} \lesssim t^{-1/2} (k^2+2)^{1/4}$:
\begin{align*}
\vert L_{q_0}\vert \lesssim t^{-1/2} 
  (1+2^{\ell/2}) 
  2^{q_0} \lesssim t^{-1}(1+2^{2\ell}),
\end{align*}
having used that $2^{q_0}\approx t^{-1/2}(1+2^{3\ell/2})$. 

Next, observe that on the support of $L_q$ we have
\begin{align}\label{prKG1main1S}
| \partial_k S| \approx |k-k_0| \jk^{-3} \approx 2^q (1+2^{3\ell})^{-1} \gtrsim t^{-1/2}(1+2^{3\ell})^{-1/2}. 
\end{align}
Integrating by parts in $k$ gives
\begin{align*}
L_{q}(t,r,s) = \frac{i}{t} \int_0^\infty e^{i t S} \partial_k \Big[ \frac{1}{\partial_k S}
	\frac{1}{\sqrt{r}} \frac{1}{\sqrt{s}} \,  A \, \varphi_q(k-k_0) \,  \varphi_\ell(k) \Big] \, dk,
\end{align*}
so that (in what follows we will sometimes  omit the variable $t$ for lighter notation)
\begin{align}\label{prKG1main1q}
\begin{split}
| L_{q} | \lesssim t^{-1} (I + II + III), \qquad
I(r,s) & = \int_0^\infty 
  \frac{|\partial_k^2 S|}{(\partial_k S)^2}
	\frac{1}{\sqrt{r}} \frac{1}{\sqrt{s}} \, | A | \, \varphi_q(k-k_0) \, \varphi_\ell(k) \, dk,	
\\
II(r,s) & = \int_0^\infty \frac{1}{|\partial_k S|}
	\frac{1}{\sqrt{r}} \frac{1}{\sqrt{s}} \, | A |
	\, \varphi_\ell(k) \, \partial_k \varphi_q(k-k_0) \, dk,
\\
III(r,s) & = \int_0^\infty \frac{1}{|\partial_k S|}
	\frac{1}{\sqrt{r}} \frac{1}{\sqrt{s}} \, \big| \partial_k \big[ A
	\, \varphi_\ell(k) \, \big] \big| \varphi_q(k-k_0) \, dk.
\end{split}
\end{align}

Since $r \gtrsim t k/\sqrt{k^2+2}$ and $sk \gtrsim 1$, using also \eqref{prKG1main1S}
and \eqref{prKG15} we obtain
\begin{align*}
| I(r,s) | \lesssim \int_0^\infty \jk^3 2^{-2q}
	\frac{1}{\sqrt{t k}} \sqrt{1+k} \frac{1}{\sqrt{s}} \chi^c(sk) \varphi_q(k-k_0) \, \varphi_\ell(k) \, dk 
	\\
	\lesssim (1+2^{3\ell}) 2^{-q} \cdot t^{-1/2}(1+2^{\ell/2}).
\end{align*}
Then, summing over $q\geq q_0$ gives a bound by $1+2^{2\ell}$.

Similarly, we can estimate the second term in \eqref{prKG1main1q}, 
using $r \gtrsim t k/\sqrt{k^2+2}$ and $sk \gtrsim 1$, by
\begin{align*}
| II(r,s) | \lesssim \int_0^\infty \jk^3 2^{-q}
	\frac{1}{\sqrt{t k}} \sqrt{1+k} \frac{1}{\sqrt{s}}  2^{-q} \varphi_{[q-2,q+2]}(k-k_0) \, \varphi_\ell(k) \, dk 
	\\
	\lesssim (1+2^{3\ell}) 2^{-q} \cdot t^{-1/2}(1+2^{\ell/2}),
\end{align*}
which suffices.
For the third term in \eqref{prKG1main1q} we use \eqref{prKG15} to bound it by 
\begin{align*}
| III(r,s) | \lesssim \int_0^\infty \jk^3 2^{-q}
	\frac{1}{\sqrt{t k}} \sqrt{1+k} \frac{1}{\sqrt{s}} \varphi_q(k-k_0) \, 2^{-\ell} \varphi_{[\ell-2,\ell+2]}(k) \, dk 
	\\
	\lesssim (1+2^{3\ell}) 2^{-\ell} \cdot t^{-1/2}(1+2^{\ell/2}),
\end{align*}
which is more than sufficient since $\ell \geq q-10$.
This concludes the proof of \eqref{prKG1main1} and therefore 
the proof of \eqref{prKG1main} for $j=1$. 

\medskip
\noindent
{\it Estimate of $K_{2,\ell}$.}
On the support of $K_{2,\ell}$ we have $sk\leq rk \lesssim 1$ so we only use the first two 
expansions in \eqref{prKG1Phi}. In particular, we can write
\begin{align}\label{prKG1Ker2'}
K_{2,\ell}(t,r,s) & :=
  \int_0^\infty \frac{r}{\jr} \frac{s}{\langle s \rangle} (1+\phi(r,k)) (1+\phi(s,k)) \,
  e^{it\sqrt{k^2+2}} \, \varphi_\ell(k) k \rho_1^\prime(k^2) \, \chi(rk)\chi(sk) \, dk,
\end{align}
%
where we are using $\phi(y,k)$ to indicate a generic function satisfying \eqref{prKG1phi}.
%
Integrating by parts using $\partial_k e^{it\sqrt{k^2+2}} = e^{it\sqrt{k^2+2}} k/\sqrt{k^2+2}$ gives
\begin{align*}
\big| K_{2,\ell}(t,r,s) \big| \lesssim \frac{1}{t}  \int_0^\infty \Big|\frac{r}{\jr} \frac{s}{\langle s \rangle}
  \partial_k \Big[ \sqrt{k^2+2} \, (1+\phi(r,k)) (1+\phi(s,k))
  \, \varphi_\ell(k) \rho_1^\prime(k^2) \, \chi(rk)\chi(sk) \Big]\Big| \, dk ,
\end{align*}
so that
\begin{align}\label{prKG120}
| K_{2,\ell}(t,r,s) | & \lesssim t^{-1} (I + II),
\\
I(r,s) & = \int_0^\infty \frac{r}{\jr} \frac{s}{\langle s \rangle}\big| \partial_k \big[ (1+\phi(r,k)) (1+\phi(s,k)) \, \chi(rk)\chi(sk) \big] \big| \, \sqrt{k^2+2} \varphi_\ell(k)\rho_1^\prime(k^2) dk, \nonumber
\\
II(r,s) & = \int_0^\infty \frac{r}{\jr} \frac{s}{\langle s \rangle}\big| \partial_k \big[ \sqrt{k^2+2} \,
  \, \varphi_\ell(k)\rho_1^\prime(k^2) \big] \, (1+\phi(r,k)) (1+\phi(s,k)) \big| \chi(rk)\chi(sk) dk. \nonumber
\end{align}
We estimate $I(r,s)$ using 
$|\partial_k \phi(y,k)| \lesssim y$,
which can be seen using Proposition \ref{Phi1Op1_1},
the estimate from Proposition \ref{SMOp1_1} for $\rho_1'(k^2)$, 
and the fact that $s \langle s \rangle^{-1} \jk, r \jr^{-1} \jk \lesssim 1$,
\begin{align*}
I(r,s) & \lesssim 
  \dfrac{r}{\langle r\rangle }\dfrac{s}{\langle s\rangle} 
  r \int_0^\infty \sqrt{k^2+2} \, \varphi_\ell(k) \rho_1^\prime(k^2) \, \chi(rk)\chi(sk) \, dk \lesssim 1+2^{\ell},
\end{align*}
and, similarly,
\begin{align*}
II(r,s) & \lesssim \int_0^\infty\dfrac{r}{\langle r\rangle }\dfrac{s}{\langle s\rangle} \Big| \partial_k \big[ \sqrt{k^2+2} \,
  \, \varphi_\ell(k)\rho_1^\prime(k^2) \big] \Big|  \, dk \lesssim 1+2^{\ell}.
\end{align*}

\medskip
\noindent
{\it Estimate of $K_{3,\ell}$.}
The estimate \eqref{prKG1main} for $j=3$ can be obtained by a combination of the 
arguments used in the previous two cases. On the support of the integral
we have $sk \lesssim 1 \lesssim rk$, so that, using the expansions in \eqref{prKG1Phi}
we can reduce to estimating an expression of the form
\begin{align}\label{prKG1Ker3'}
\begin{split}
M(t,r,s) =  
  \int_0^\infty \frac{1}{\sqrt{rk}} \overline{a_1(k^2)} e^{-irk} \big( 1 + \phi(r,k)) 
  \, \frac{s}{\langle s \rangle}  (1+\phi(s,k)) 
  \, e^{it\sqrt{k^2+2}} \, 
  \\ \times \varphi_\ell(k) k \rho_1^\prime(k^2) \,  \chi^c(rk)\chi(sk) \, dk
\end{split}
\end{align}
where, as before, $\phi$ denotes a generic function obeying the bounds \eqref{prKG1phi} 
and we have only considered the main term with the exponential $e^{-irk}$, disregarding the other with the opposite sign
which gives lower order contributions.
Similarly to \eqref{prKG1L} and \eqref{prKGmain1q} we let 
\begin{align}\label{prKG1main3q}
\begin{split}
& M(t,r,s) = \sum_{q \geq q_0}^{\ell+10} M_{q}(t,r,s), 
\\
& M_{q}(t,r,s) := \int_0^\infty \frac{\sqrt{k}}{\sqrt{r}} e^{i t T(r,s,k,t)}  n(k) B(r,s,k) \, 
 \varphi_q^{(q_0)}(k-k_0) \varphi_\ell(k) \, dk,
\end{split}
\end{align}
where $q_0$ is the smallest integer such that $2^{q_0} \geq t^{-1/2}(1+2^{3\ell/2})$
and we denote
\begin{align*}
\begin{split}
T(r,s,k,t) & := -\frac{k}{t}r + \sqrt{k^2+2},
\\
n(k) & := \overline{a(k^2)} \rho_1^\prime(k^2) \cdot \jk^{-1},
\\
B(r,s,k) & := \big( 1 + \phi(r,k)) \big( 1 + \phi(s,k)) \chi^c(rk)\chi(sk) \cdot \frac{s}{\langle s \rangle} \jk.
\end{split}
\end{align*}
As before, we work under the assumption that the phase is stationary,
and let $k_0$ denote the unique solution to $\partial_k T = 0$. 
In particular, we also have $r \gtrsim tk/(k+1)$. 
It is not hard to see, by direct integration as in the case of $K_{1,\ell}$,
that $|M_{q_0}| \lesssim t^{-1}(1+2^{2\ell})$.

To deal with the cases $q>q_0$, we first observe that, in view of \eqref{SMOp1_4}
and the fact that on the support of the integral $sk\lesssim 1$,
we have estimates similar to those in \eqref{prKG15} and \eqref{prKG1main1S}:
\begin{align}\label{prKG15'}
\begin{split}
& |\partial^\alpha_k n(k)| \lesssim k^{-\alpha},
\qquad |\partial_k^\alpha B(r,s,k)| \lesssim k^{-\alpha}, \qquad \alpha = 0,1,2;
\\
& | \partial_k T(r,s,k,t)| \approx 2^q (1+2^{3\ell})^{-1} \gtrsim t^{-1/2}(1+2^{3\ell})^{-1/2}.
\end{split}
\end{align}
We can then proceed as we did after \eqref{prKG1main1S} up to some small adjustments.
An integration by parts leads us to estimate
\begin{align}\label{prKG1main3q'}
\begin{split}
| M_{q} | \lesssim t^{-1} (I + II + III),
\qquad
I(r,s) & = \int_0^\infty 
  \frac{|\partial_k^2 T|}{(\partial_k T)^2}
	\frac{\sqrt{k}}{\sqrt{r}} \, | n \,  B | \, \varphi_q(k-k_0) \, \varphi_\ell(k) \, dk,	
\\
II(r,s) & = \int_0^\infty \frac{1}{|\partial_k T|}
	\frac{\sqrt{k}}{\sqrt{r}} \, | n \,  B |
	\, \varphi_\ell(k) \, \partial_k \varphi_q(k-k_0) \, dk,
\\
III(r,s) & = \int_0^\infty \frac{1}{|\partial_k T|}
	\frac{\sqrt{k}}{\sqrt{r}} \, \big| \partial_k \big[ n \,  B
	\, \varphi_\ell(k) \, \big] \big| \varphi_q(k-k_0) \, dk.
\end{split}
\end{align}
Since $r \gtrsim t k/\sqrt{k^2+2}$, using also \eqref{prKG15'} we can bound
\begin{align*}
| I(r,s) | \lesssim \int_0^\infty \jk^3 2^{-2q}
	\frac{\sqrt{1+k}}{\sqrt{t}}  \varphi_q(k-k_0) \, \varphi_\ell(k) \, dk 
	\lesssim (1+2^{3\ell}) 2^{-q} \cdot t^{-1/2}(1+2^{\ell/2});
\end{align*}
summing over $q\geq q_0$ gives a bound by $1+2^{2\ell}$. 
The second term in \eqref{prKG1main3q} is estimated similarly
\begin{align*}
| II(r,s) | \lesssim \int_0^\infty \jk^3 2^{-q}
	\frac{\sqrt{1+k}}{\sqrt{t}}  2^{-q} \varphi_{[q-2,q+2]}(k-k_0) \, \varphi_\ell(k) \, dk 
	\\
	\lesssim (1+2^{3\ell}) 2^{-q} \cdot t^{-1/2}(1+2^{\ell/2}),
\end{align*}
which suffices.
For the third term in \eqref{prKG1main3q} we use \eqref{prKG15'} to bound it by 
\begin{align*}
| III(r,s) | \lesssim \int_0^\infty \jk^3 2^{-q}
	\frac{\sqrt{1+k}}{\sqrt{t}} \varphi_q(k-k_0) \, 2^{-\ell} \varphi_{[\ell-2\ell+2]}(k) \, dk 
	\\
	\lesssim (1+2^{3\ell}) 2^{-\ell} \cdot t^{-1/2}(1+2^{\ell/2}),
\end{align*}
which is more than sufficient. 
The last three bounds, \eqref{prKG1main3q}, and \eqref{prKG1main3q'} give $|M(r,s)| \lesssim t^{-1}
(1+2^{2\ell})$ as desired.
This concludes the proof of \eqref{prKG1main} for $j=3$, and of the proposition.
\end{proof}

\begin{rem}[The case $\vert n\vert\geq 2$]\label{section4_nvortex}
Again, we briefly show some of the calculations needed for the vortex of degree $n$. In Proposition \ref{propKG1dec}, the Weyl solution in \eqref{prKG1Phi} now writes \begin{align}\label{check_nvor_proof8}
\Phi_{1,n}(r,k^2) = \left\{ 
\begin{array}{ll}
r^{1/2+n} \big( 1 + \phi_{<}(r,k)), &  rk \lesssim 1, \quad r \lesssim 1,
\\
r^{1/2} \big( 1 + \phi_{>}(r,k)), &  rk \lesssim 1, \quad r \gtrsim 1,
\\
\dfrac{1}{\sqrt{k}} \big[ a_{1,n}(k^2) e^{irk} \big( 1 + \phi_+(r,k)) 
  + \overline{a_{1,n}}(k^2) e^{-irk} \big(1 + \phi_-(r,k)) \big], &  rk \gtrsim 1
\end{array} 
\right.
\end{align}
where $a_{1,n}\approx \langle k\rangle^{-n}$, and we use $\phi(r,k)$ to indicate a generic function satisfying \eqref{prKG1phi}. Then, in the case $sk\leq rk\lesssim1$, for example, the integral $K_{2,\ell}$ in \eqref{prKG1Ker2'} needs to be replaced by \begin{align*}
K_{2,\ell,n}(t,r,s) & :=
  \int_0^\infty \frac{r^n}{\jr^n} \frac{s^n}{\langle s \rangle^n} (1+\phi(r,k)) (1+\phi(s,k)) \,
  e^{it\sqrt{k^2+2}} \, \varphi_\ell(k) k \rho_1^\prime(k^2) \, \chi(rk)\chi(sk) \, dk.
\end{align*}
However, noticing that $r^ns^n\langle r\rangle^{-n}\langle s\rangle^{-n}\langle k\rangle^{2n}\lesssim 1$,
on the support of the integral,
we deduce that all calculations for $K_{2,\ell,n}$ follow identically those for $K_{2,\ell}$ in\eqref{prKG1Ker2'}.

If $sk\lesssim 1\lesssim rk$, the integral $M(t,r,s)$ in \eqref{prKG1Ker3'} when $\vert n\vert\geq 2$ is replaced by \begin{align*}
M_n(t,r,s) =  
  \int_0^\infty \frac{1}{\sqrt{rk}} \overline{a_{1,n}(k^2)} e^{-irk} \big( 1 + \phi(r,k)) 
  \, \dfrac{s^n}{\langle s\rangle^n}  (1+\phi(s,k)) 
  \, e^{it\sqrt{k^2+2}} \, 
  \\ \times \varphi_\ell(k) k \rho_{1,n}'(k^2) \,  \chi^c(rk)\chi(sk) \, dk.
\end{align*}
Then, it is enough to notice that 
$| \overline{a_{1,n}}(k^2) s^n \langle s\rangle^{-n} \rho_{1,n}'(k^2)| 
\lesssim s^n \langle s\rangle^{-n}\langle k\rangle^n\lesssim 1,$
and hence the same bounds deduced for $M(t,r,s)$ in \eqref{prKG1Ker3'} apply in the case $\vert n\vert\geq2$.
In all other cases, once all modifications have been made, all the resulting integrals
in Proposition \ref{propKG1dec} when  $\vert n\vert\geq2$ ultimately reduce to the same estimates as those for $n=1$, 
just as in the two cases shown above.
\end{rem}



\subsection{Second order dynamics for $\mathcal{L}_2$}\label{ssecKG2}
Writing out the solution to \eqref{FTKG2}, and using the Fourier representation 
from Proposition \ref{propFT2}, 
we have 
\begin{align}
\begin{split}
\wt{z_2}(t,k) & = e^{itk} \frac{1}{2}\Big( \wtF_2(\sqrt{r}f)(k) + \frac{1}{ik}\wtF_2(\sqrt{r}g)(k) \Big)
+ e^{-itk} \frac{1}{2}\Big( \wtF_2(\sqrt{r}f)(k) - \frac{1}{ik}\wtF_2(\sqrt{r}g)(k) \Big)
\\
& = e^{itk} \frac{1}{2} \int_{0}^\infty \Phi_2(r,k^2) \Big( \sqrt{r}f + \frac{1}{ik} \sqrt{r}g \Big) \, dr
+ e^{-itk} \frac{1}{2} \int_{0}^\infty \Phi_2(r,k^2) \Big( \sqrt{r}f - \frac{1}{ik} \sqrt{r}g \Big) \, dr,
\end{split}
\end{align}
so that we can represent the linear solution of \eqref{KGj} when $j=2$ as
\begin{align}\label{ssecKG20}
\begin{split}
& v(t,r) = \frac{1}{\sqrt{r}} 
  \int_0^\infty \Phi_1(r,k^2) e^{itk} \wt{F_+}(k) \, 2k \rho_2^\prime(k^2) \, dk
  \\
  & \qquad \qquad + \frac{1}{\sqrt{r}} 
  \int_0^\infty \Phi_2(r,k^2) e^{-itk} \wt{F_-}(k) \, 2k \rho_2^\prime(k^2) \, dk,
  \\
& \wt{F_{\pm}}(k) := \frac{1}{2}\Big( \wtF_2(\sqrt{r}f)(k) \pm \frac{1}{ik}\wtF_2(\sqrt{r}g)(k) \Big)
  = \frac{1}{2}\wt{\mathcal{F}_2}\Big( \sqrt{r}f \pm \frac{1}{i\sqrt{\mathcal{H}_2}} \sqrt{r}g \Big)(k) ,
\end{split}
\end{align}
and as
\begin{align}\label{ssecKG2'}
\begin{split}
v(t,r) & = \int_0^\infty 
  \Big(K_+(t,r,s) f(s) + K_+'(t,r,s) g(s)\Big) \, s \, ds 
  \\
  & + \int_0^\infty 
  \Big(K_-(t,r,s) f(s) - K_-'(t,r,s) g(s)\Big) \, s \, ds 
\end{split}
\end{align}
where the kernels are defined by
\begin{align}\label{Ker2}
K_\pm(t,r,s) & :=  \frac{1}{\sqrt{rs}} 
  \int_0^\infty \Phi_2(r,k^2) \Phi_2(s,k^2) e^{\pm itk} \, k \rho_2^\prime(k^2) \, dk,
\\
\label{Ker2'}
K_\pm'(t,r,s) & := \frac{1}{\sqrt{rs}} \int_0^\infty
  \frac{1}{ik} \Phi_2(r,k^2) \Phi_2(s,k^2) e^{\pm itk} \, k \rho_2^\prime(k^2) \, dk.
\end{align}
As before, for the frequency projections we have
\begin{align}\label{ssecKG2'l}
\begin{split}
P_\ell^{\mathcal{L}_2} v(t,r) & = \int_0^\infty 
  \Big(K_{+,\ell}(t,r,s) P_{[\ell-2,\ell+2]}^{\mathcal{L}_2} f(s) + K_{+,\ell}'(t,r,s) 
    P_{[\ell-2,\ell+2]}^{\mathcal{L}_2} g(s)\Big) \, s \, ds 
  \\
  & + \int_0^\infty 
  \Big(K_{-,\ell}(t,r,s) P_{[\ell-2,\ell+2]}^{\mathcal{L}_2} f(s) 
  - K_{-,\ell}'(t,r,s) P_{[\ell-2,\ell+2]}^{\mathcal{L}_2} g(s)\Big) \, s \, ds 
\end{split}
\end{align}
where the kernels with the index $\ell$ are defined as in \eqref{Ker2} and \eqref{Ker2'}
with an additional $\varphi_\ell(k)$ cutoff in the integrals.
Note that we are using the same letter for the kernels as in the previous subsection,
but this should cause no confusion here. 

\smallskip
\begin{prop}\label{propKG2dec}
Let $v$ be a solution to $v_{tt} + \mathcal{L}_2 v = 0$, with $(v(0),\partial_t v(0)) = (f,g)$
such that
\begin{align}
\sum_{\ell} 2^\ell (1+2^{\ell/2}) {\big\| P_\ell^{\mathcal{L}_2} f \big\|}_{L^1(rdr)} 
  + (1+2^{\ell/2}) {\big\| P_\ell^{\mathcal{L}_2} g \big\|}_{L^1(rdr)} \leq 1.
\end{align}
Then, for $t \geq 1$,
\begin{align}\label{KG2decconc}
|v(t,r)| \lesssim \frac{1}{\sqrt{t}} \frac{1}{\sqrt{|t-r|+1}}.
\end{align}
\end{prop}


\begin{proof}
In view of the representation \eqref{ssecKG2'} it suffices to show the following estimates for 
the kernels in 
the expressions \eqref{Ker2}-\eqref{Ker2'}: 
\begin{align}\label{prKG2Ker0}
\begin{split}
& \sup_{r,s\geq 0} \big| K_{\pm,\ell}(t,r,s) \big| \lesssim 2^\ell (1+2^{\ell/2}) \frac{1}{\sqrt{t}} \frac{1}{\sqrt{|t-r|+1}},
\\
& \sup_{r,s\geq 0} \big| K_{\pm,\ell}'(t,r,s) \big| \lesssim (1+2^{\ell/2}) 
  \frac{1}{\sqrt{t}} \frac{1}{\sqrt{|t-r|+1}}.
\end{split}
\end{align}
Since the two kernels are the same up to a factor of $1/k$, it is enough to prove the estimate for $K_{\pm,\ell}$.
We adopt similar notation as in the proof of Proposition \ref{propKG1dec}.
By symmetry we may assume, without loss of generality, that $r\geq s$. 
We let $\chi = \chi(x)$, $x\geq 0$ be a smooth function equal to $1$ in $[0,1]$, 
decreasing, and vanishing for $x \geq 2$, let $\chi^c := 1-\chi$ and decompose the kernel as follows:
\begin{align}
\nonumber
K_{+,\ell}(t,r,s) & = K_{1,\ell}(t,r,s) + K_{2,\ell}(t,r,s) + K_{3,\ell}(t,r,s),
\\
\label{prKG2Ker1}
K_{1,\ell}(t,r,s) & :=  \frac{1}{\sqrt{rs}} 
  \int_0^\infty \Phi_2(r,k^2) \Phi_2(s,k^2) e^{itk} \, \varphi_\ell(k) k \rho_2^\prime(k^2) \, 
  \chi^c(sk) \, dk,
\\
\label{prKG2Ker2}
K_{2,\ell}(t,r,s) & :=  \frac{1}{\sqrt{rs}} 
  \int_0^\infty \Phi_2(r,k^2) \Phi_2(s,k^2) e^{itk} \, \varphi_\ell(k) k \rho_2^\prime(k^2) \, 
  \chi(rk)\chi(sk) \, dk,
\\
\label{prKG2Ker3}
K_{3,\ell}(t,r,s) & :=  \frac{1}{\sqrt{rs}} 
  \int_0^\infty \Phi_2(r,k^2) \Phi_2(s,k^2) e^{itk} \, \varphi_\ell(k) k \rho_2^\prime(k^2) \, 
  \chi^c(rk)\chi(sk) \, dk.
\end{align}
It then suffices to show that
\begin{align}\label{prKG2main}
\sup_{r,s\geq 0} \big| K_{j,\ell}(t,r,s) \big| \lesssim 2^\ell (1+2^{\ell/2}) \, t^{-1/2} (1+|r-t|)^{-1/2}, \qquad j=1,2,3.
\end{align}

Similarly to \eqref{prKG1Phi}, we can use Lemmas \ref{lemFou1} and \ref{lemWeyl} to write 
\begin{align}\label{prKG2Phi}
\Phi_2(r,k^2) = \left\{ 
\begin{array}{ll}
r^{3/2} \big( 1 + \phi_{<}(r,k)), &  rk \lesssim 1, \quad r \lesssim 1,
\\
\\
r^{1/2} \big( 1 + \phi_{>}(r,k)), &  rk \lesssim 1, \quad r \gtrsim 1,
\\
\\
\dfrac{1}{\sqrt{k}} \big[ a_2(k^2) e^{irk} \big( 1 + \phi_+(r,k)) 
  + \overline{a_2}(k^2) e^{-irk} \big(1 + \phi_-(r,k)) \big], &  rk \gtrsim 1
\end{array} 
\right.
\end{align}
where $\phi \in \{ \phi_{<}, \phi_{>}\phi_-,\phi_+ \}$ is a generic function satisfying 
\begin{align}\label{prKG2phiest}
|\partial_k^\alpha \phi(y,k)| \lesssim k^{-\alpha}, \quad \alpha = 0,1,2,
\end{align}
which follows from an inspection of the statement of Lemma \ref{lemFou1}.
Also, recall that $a_2(k^2) \approx \langle k \rangle^{-1}$ with compatible symbol-type estimates
for the derivatives, see Proposition \ref{propSM2}.


\medskip
\noindent
{\it Estimate of $K_{1,\ell}$.}
On the support of $K_{1,\ell}$ we have $rk \geq sk \gtrsim 1$, and
using \eqref{prKG2Phi} we see that the desired estimate \eqref{prKG1main} reduces to the same estimate for
the kernels 
\begin{align*}
\begin{split}
L_{\eps_1\eps_2}(t,r,s) := 
  \int_0^\infty \frac{1}{\sqrt{rk}} \frac{1}{\sqrt{sk}} a_2^{\eps_1}(k^2) e^{\eps_1 irk} \big( 1 + \phi_{\eps_1}(r,k))
  a_2^{\eps_2}(k^2) e^{\eps_2 isk} \big( 1 + \phi_{\eps_2}(s,k))  
  \\
  \times e^{itk} \, \varphi_\ell(k) k \rho_2^\prime(k^2) \, \chi^c(sk) \, dk, \qquad \eps_1,\eps_2 \in \{+,-\},
\end{split}
\end{align*}
where, as before, we are denoting $f^+=f$ and $f^-=\bar{f}$.
It suffices to look at the case $(\eps_1,\eps_2) = (-,+)$ since the other cases are similar or easier.
Using that $|a_2(k^2)|^2 \rho_2^\prime(k^2) = 1/(4\pi)$, see \eqref{propSMrho}, we write $4\pi L_{-+} = L$ where
\begin{align}\label{prKG2L}
\begin{split}
L(t,r,s) & := \int_0^\infty \frac{1}{\sqrt{rk}} \frac{1}{\sqrt{sk}} e^{i S(r,s,k,t)} 
  A(r,s,k) \,  \varphi_\ell(k) k \, dk,
\\
S(r,s,k,t) & := k( t-r + s),
\\
A(r,s,k) & := \big( 1 + \phi_-(r,k)) \big( 1 + \phi_{+}(s,k)) \chi^c(sk),
\end{split}
\end{align}
and aim to show
\begin{align}\label{prKG2main1}
\sup_{r,s\geq 0} \big| L(t,r,s) \big| \lesssim 2^\ell (1+2^{\ell/2}) \, t^{-1/2} (1+|r-t|)^{-1/2}.
\end{align}
Recall that, in view of \eqref{propSMrho} and \eqref{prKG2phiest}, we have the symbol type estimates
\begin{align}\label{prKG15''} 
|\partial_k^\alpha A(r,s,k)| \lesssim k^{-\alpha}, \qquad \alpha = 0,1,2.
\end{align}


\smallskip
\noindent
{\it Case $r\geq t/3$.}
First, we trivially have
\begin{align*}
|L(t,r,s)|\lesssim  \frac{1}{\sqrt{r}} \int_0^\infty \frac{1}{\sqrt{sk}} \varphi_\ell(k) \sqrt{k} \, \chi^c(sk)\, dk
\lesssim \frac{1}{\sqrt{r}} 2^{3\ell/2},
\end{align*}
having used that $sk\gtrsim1$. Hence, if $|t-r| \leq 1$ this already suffices.
If $2^\ell \leq (1+|t-r|)^{-1}$, the above is also bounded by the right-hand side of \eqref{prKG2main} as desired.

When $2^\ell \geq (1+|t-r|)^{-1}$, we do not bound as above, instead, we first look at the case
when $s \geq |t-r|/2$ and use the factor $1/\sqrt{s}$ to bound
\begin{align*}
|L(t,r,s)|\lesssim  \frac{1}{\sqrt{r}} \frac{1}{\sqrt{s}} \int_0^\infty 
  \varphi_\ell(k) \, \chi^c(sk)\, dk
  \lesssim \frac{1}{\sqrt{t}}  \frac{1}{\sqrt{|t-r|}} 2^{\ell},
\end{align*}
which is sufficient.
Whereas if $s \leq |t-r|/2$, we integrate by parts using
$|\partial_k S| = |t-r+s| \geq |t-r|/2,$ to obtain
\begin{align*}
| L(t,r,s) | = \frac{1}{|t-r+s|} \frac{1}{\sqrt{r}}
  \Big| \int_0^\infty e^{i S} 
  \partial_k \Big[\frac{1}{\sqrt{sk}} \, A(r,s,k) \, \sqrt{k} \varphi_\ell(k) \Big]\, dk \Big| 
  \lesssim \frac{1}{|t-r|} \frac{1}{\sqrt{t}} \, 2^{\ell/2},
\end{align*}
which is enough since $|t-r| \gtrsim 2^{-\ell}$.

\smallskip
\noindent
{\it Case $r \leq t/3$.}
In this case, since $s\leq r \leq t/3$, we have 
$|\partial_k S| = |t-r+s| \geq t/3,$
and integration by parts gives
\begin{align*}
| L(t,r,s) | = \frac{1}{|t-r+s|}
  \Big| \int_0^\infty e^{i S} 
  \partial_k \Big[  \frac{1}{\sqrt{rk}}\frac{1}{\sqrt{sk}} \,  A(r,s,k) \, k \varphi_\ell(k) \Big]\, dk \Big| 
  \lesssim \frac{1}{t} \, 2^{\ell}.
\end{align*}
This bound is consistent with \eqref{prKG2main1} since $t \geq t-r \gtrsim 1$, 
and concludes the proof of \eqref{prKG2main} when $j=1$.

\medskip
\noindent
{\it Estimate of $K_{2,\ell}$.}
On the support of $K_{2,\ell}$ we have $sk\leq rk \lesssim 1$ so we only use the first two 
expansions in \eqref{prKG2Phi}. In particular, we can write
\begin{align}\label{prKG2Ker2'}
K_{2,\ell}(t,r,s) & :=
  \int_0^\infty \frac{r}{\jr} \frac{s}{\langle s \rangle} (1+\phi(r,k)) (1+\phi(s,k)) \,
  e^{itk} \, \varphi_\ell(k) k \rho_2^\prime(k^2) \, \chi(rk)\chi(sk) \, dk,
\end{align}
where $\phi$ is a generic function satisfying 
$|\partial_k^\alpha \phi| \lesssim k^{-\alpha},$ for $\alpha = 0,1,2$, consistently with \eqref{prKG2phiest}.

Note first that, using \eqref{propSMrho}, we can bound
\begin{align*}
\big| K_{2,\ell}(t,r,s) \big| & \lesssim
  \int_0^\infty \frac{s}{\langle s \rangle} \frac{r}{\langle r \rangle} \, \jk^2 
  \big| k  \, \varphi_\ell(k) 
  \, \chi(rk)\chi(sk) \, dk |
  \lesssim 2^\ell \min\big(r^{-1}, 2^{\ell}\big).
\end{align*}
Then, if $r \geq t/2$, so that we also have $|r-t| \lesssim r$,
the bound above is stronger than the right-hand side of \eqref{prKG2main1}.
In the case $r \leq t/2$, we use integration by parts instead, to obtain
\begin{align*}
\big| K_{2,\ell}(t,r,s) \big| \lesssim \frac{1}{t} \Big| \int_0^\infty e^{i t k} 
  \frac{r}{\jr} \frac{s}{\langle s \rangle} 
  \partial_k \Big[ (1+\phi(r,k)) (1+\phi(s,k))
  \\ 
  \times \, k \varphi_\ell(k) \rho_2^\prime(k^2) \, \chi(rk)\chi(sk) \Big] \, dk \Big|
  \lesssim \frac{1}{t} 2^\ell 
\end{align*}
which suffices 
since $t \gtrsim |t-r| \gtrsim 1$. This proves \eqref{prKG2main1} for $j=2$.

\medskip
\noindent
{\it Estimate of $K_{3,\ell}$.}
On the support of the integral we have $sk \lesssim 1 \lesssim rk$, so that, 
using the expansions in \eqref{prKG2Phi} we can reduce to estimating an expression of the form
\begin{align}\label{prKG2Ker3'}
\begin{split}
M(r,s) =  
  \int_0^\infty \frac{1}{\sqrt{rk}} a_2^{-}(k^2) e^{-irk} \big( 1 + \phi(r,k))
  \, \frac{s}{\langle s \rangle} (1+\phi(s,k)) 
  \, e^{itk} \\ \varphi_\ell(k) k \rho_2^\prime(k^2) \, 
  \chi^c(rk)\chi(sk) \, dk
\end{split}
\end{align}
where, as before, $\phi$ denotes a generic function obeying the bounds \eqref{prKG1phi}.
Note that in \eqref{prKG2Ker3'} we have chosen the $-$ sign for the exponential $e^{-irk}$
since the case of $e^{irk}$ would correspond to a non-stationary phase, and hence it is easier to handle.
We can then rewrite \eqref{prKG2Ker3'} as 
\begin{align}\label{prKG2Ker3''}
M(r,s) = \int_0^\infty \frac{\sqrt{k}}{\sqrt{r}} e^{i(t -r)k}  n(k) B(r,s,k) \, \varphi_\ell(k) \, dk,
\end{align}
where
\begin{align*}
\begin{split}
n(k) & := a_2^-(k^2) \rho_2^\prime(k^2) \cdot \jk^{-1},
\\
B(r,s,k) & := \big( 1 + \phi(r,k)) \big( 1 + \phi(s,k)) \chi^c(rk)\chi(sk) \cdot \frac{s}{\langle s \rangle} \jk,
\end{split}
\end{align*}
and the following estimates hold
\begin{align}\label{prKG25'}
& |\partial^\alpha n(k)| \lesssim \jk^{-\alpha},
\qquad |\partial_k^\alpha B(r,s,k)| \lesssim k^{-\alpha}, \qquad \alpha = 0,1,2.
\end{align}

Let us first look at the case $r \in (t/2,2t)$.
Observe that we have the trivial bound $|M(r,s)| \lesssim r^{-1/2} 2^{3\ell/2}$.
This suffices if $|r-t|\lesssim 1$, or if $2^\ell \leq (1+|r-t|)^{-1}$. 
When instead $|r-t|\gtrsim 1$ and $2^\ell \gtrsim |r-t|^{-1}$,
we integrate by parts in \eqref{prKG2Ker3''} and obtain
\begin{align*}
|M(r,s)| \lesssim \frac{1}{\sqrt{r}} \frac{1}{|t-r|}
  \int_0^\infty \big| \partial_k \big[ \sqrt{k} n(k) B(r,s,k) \, \varphi_\ell(k) \big] \big| \, dk
  \\
  \lesssim \frac{1}{\sqrt{r}} \frac{1}{|t-r|} 2^{\ell/2}
  \lesssim \frac{1}{\sqrt{r}} \frac{1}{\sqrt{|t-r|+1}} 2^{\ell}.
\end{align*}

When $r \not\in (t/2,2t)$, integration by parts gives
\begin{align*}
|M(r,s)| \lesssim \frac{1}{|t-r|}
  \int_0^\infty \big| \partial_k \big[ \frac{1}{\sqrt{rk}} \, k  n(k) B(r,s,k) \, \varphi_\ell(k) \big] \big| \, dk
  \lesssim  \frac{1}{|t-r|} 2^\ell,
\end{align*}
having used that $rk \gtrsim 1$ on the support of the integral;
since $|r-t| \gtrsim t$ the above bound is consistent with the inequality \eqref{prKG2main1}.
This concludes the proof of the proposition.
\end{proof}

\medskip
\section{Weighted-type estimates}\label{WDecKG_Sec}
In this section we establish weighted-type decay estimates 
where the norms on the right-hand side are measured in the distorted Fourier space. 

\subsection{Klein-Gordon case}
We first estimate solutions associated with $\mathcal{L}_1$. 

\begin{prop}\label{WDecKG1}
Let $v$ be a solution to $ \partial_t^2 v + \mathcal{L}_1 v = 0$, 
with initial conditions $(v(0),v_t(0)) = (P_c^1f,P_c^1g)$.
Let 
\begin{align}\label{WDecKG2}
F(k) &:= \frac{1}{a_1(k)} \wtF_1 (\sqrt{r} f)(k), 
  \qquad \hbox{and} \qquad G(k) := \frac{1}{a_1(k)} \wtF_1 (\sqrt{r} g)(k),
\end{align}
where $ \wtF_1$ is the Fourier transform associated to the operator $\mathcal{H}_1 = r^{1/2} \mathcal{L}_1 r^{-1/2}$. 
Then, the following inequality holds
\begin{align}\label{WDecKG3}
\begin{split}
|v(t,r)| &\lesssim \frac{1}{t} {\| \jk^2 F \|}_{L^\infty_k}
 + \frac{1}{t^{5/4}} {\| \jk^{11/4}\partial_k F \|}_{L^2_k} 
 + \frac{1}{t^{7/4}} {\|  \jk^{17/4}\partial_k^2 F \|}_{L^2_k}
 \\ & \qquad +\frac{1}{t} {\| \jk  G \|}_{L^\infty_k}
 + \frac{1}{t^{5/4}} {\| \jk^{7/4}\partial_k G  \|}_{L^2_k} 
 + \frac{1}{t^{7/4}} {\|  \jk^{13/4}\partial_k^2 G  \|}_{L^2_k}.
\end{split}
\end{align}
\end{prop}


\smallskip
We refer to Lemma \ref{WDecKG4} for some estimates on Fourier norms,
such as those appearing in the above statement, in terms of more standard norms in physical space. 

\smallskip
\begin{proof}[Proof of Proposition \ref{WDecKG1}]
Recall that 
\begin{align*}
v(t,r)
&=\dfrac{1}{\sqrt{r}}\int_0^\infty \Phi_1(r,k^2)e^{it\sqrt{k^2+2}}\widetilde{F}_+(k)2k\rho_1'(k^2)dk
\\ &\quad +\dfrac{1}{\sqrt{r}}\int_0^\infty \Phi_1(r,k^2)e^{-it\sqrt{k^2+2}}\widetilde{F}_-(k)2k\rho_1'(k^2)dk
\end{align*}
with $\widetilde{F}_+$ and $\widetilde{F}_-$ defined as in \eqref{ssecKG1}. 
The claimed estimates for the terms with $\widetilde{F}_+$ and $\widetilde{F}_-$ are identical, 
so we will only give the details for $\widetilde{F}_+$, and disregard the other;
furthermore, the proof for the term with $(k^2+2)^{-1} \wtF_1 (\sqrt{r}g)$ 
is essentially the same as that for the term with $\wtF_1 (\sqrt{r} f)$, 
and thus we will restrict ourselves to estimating only
the contribution from $\wtF_1 (\sqrt{r} f)$. 

Introducing a cutoff function $\chi$ as in the previous proofs, with the above reductions,
and with the notation \eqref{WDecKG2}, 
we see that it suffices to estimate the following two terms 
\begin{align}\label{WDecKG7}
v^{(1)}(t,r) & := \frac{1}{\sqrt{r}} \int_0^\infty  \Phi_1(r,k^2) 
  \chi^c(rk) e^{it\sqrt{k^2+2}}\, F(k) \, 2k a_1(k^2)\rho_1^\prime(k^2) \, dk,
\\
\label{WDecKG8}
v^{(2)}(t,r) & :=  \frac{1}{\sqrt{r}} \int_0^\infty \Phi_1(r,k^2) 
  \chi(rk) e^{it\sqrt{k^2+2}} \,  F(k) \, 2k a_1(k^2)\rho_1^\prime(k^2) \, dk.
\end{align}

\smallskip
\noindent
{\it Estimate of $v^{(1)}$}.
Using the expansion \eqref{prKG1Phi} for $\Phi_1$ when $rk \gtrsim 1$ 
reduces matters to estimating by the right-hand side of \eqref{WDecKG3}, 
the following expression
\begin{align}\label{WDecKG_v1_w}
w(t,r) & := \frac{1}{\sqrt{r}} \int_0^\infty \frac{1}{\sqrt{k}} e^{it S(t,r,k)} (1+\phi(r,k))
  \chi^c(rk) \, F(k) \, 2k \, dk,
  \ \quad S(t,r,k) := \sqrt{k^2+2} - \frac{rk}{t},
\end{align}
with $\phi(r,k)$ as in \eqref{prKG1phi}, 
and having used $|a_1(k^2)|^2\rho_1^\prime(k^2) = 1/(4\pi)$ 
according to \eqref{SMOp1_4}.
The contribution of the term associated with $e^{irk}$ 
is non-stationary and easier to estimate, so we have disregarded it. 

Note that we may restrict the support of the integral to $k \gg t^{-1/2}$,
otherwise the bound follows by direct integration: 
\[
\left\vert \int_0^{ct^{-1/2}}\dfrac{1}{\sqrt{rk}}e^{itS(t,r,k)}\big(1+\phi(r,k)\big)\chi^c(rk)F(k)2kdk\right\vert 
  \lesssim \Vert F\Vert_{L^\infty_k}\int_0^{t^{-1/2}} k\,dk\lesssim \dfrac{1}{t}\Vert F\Vert_{L^\infty_k},
\]
having also used that $rk\gtrsim1$. We can thus insert a cutoff $\chi^c(ck\sqrt{t})$ 
for some sufficiently small $c>0$ in the integral in the definition of $w$. 
Moreover, for stationary points to exists we require $r/t\in(0,1)$, 
and, from now, focus on this case.

Let $k_0 := \sqrt{2}r/\sqrt{t^2-r^2}$ denote the stationary point of $S(t,r,k)$, 
and let us restrict to the main stationary contribution $k/k_0 \approx 1$ 
inserting a corresponding cutoff $\varphi_{[-2,2]}(k/k_0)$. 
We disregard 
the non-stationary contribution coming from the complementary region, 
which gives lower order terms or terms that are easier to estimate.

With the notation for Littlewood-Paley cutoffs in Subsection \ref{secnot}, we split $w$ as
\begin{align}
\label{WDecKG9}
\begin{split}
w(t,r) & = \sum_{q \geq q_0} I_q(t,r), 
\\
I_q(t,r) & := \frac{1}{\sqrt{r}} \int_0^\infty \frac{1}{\sqrt{k}} e^{itS(t,r,k)} 
  m(r,k) \, F(k) \, 2k \, \varphi_q^{(q_0)}(k-k_0)\, dk, 
  \\
\qquad \qquad m(r,k) & := \varphi_{[-2,2]}(k/k_0) (1+\phi(r,k)) \chi^c(rk) \chi^c(ck\sqrt{t}), \qquad  k_0 = \sqrt{2}r/\sqrt{t^2-r^2},
\end{split}
\end{align}
where $q_0$ is the smallest integer such that $2^{q_0} \geq  t^{-1/2}\langle k_0\rangle^{3/2}$ 
(and we omit the dependence of $m$ on $t$).
Note that we must have $k \approx k_0 \gtrsim 2^q$ on the support of $I_q$, as well as, by Taylor expansion,
\begin{align*}
|\partial_k S(t,r,k)| & \approx 2|k-k_0|\langle k_0\rangle^{-3} \approx 2^q\langle k_0\rangle^{-3},
\end{align*}
and, according to Lemma \ref{lemWeyl1}, $|\partial_k^\alpha m(r,k) | \lesssim k^{-\alpha}$. 

The estimate for the case $q=q_0$ 
follows from direct integration,
\begin{align*}
|I_{q_0}(t,r) | & \lesssim \frac{1}{\sqrt{r}} \int_0^\infty \sqrt{k} \, \chi^c(rk) 
  \, \varphi_{\leq q_0}(k-k_0)\, \jk^{-2}dk \, {\| \jk^{2}F \|}_{L^\infty_k}  
  \\ & \lesssim t^{-1/2}\langle k_0\rangle^{-3/2} 2^{q_0}\Vert \jk^2 F\Vert_{L^\infty_k}\lesssim t^{-1} {\| \jk^{2}F \|}_{L^\infty_k},
\end{align*}
having used that $r = k_0 t/\sqrt{k_0^2+2}$ and $2^{q_0}\approx t^{-1/2}\langle k_0\rangle^{3/2}$;
the bound so obtained is consistent with \eqref{WDecKG5}.

In the case $q>q_0$, 
we integrate by parts in $k$ in the integral \eqref{WDecKG9} twice, from where we obtain 
\begin{align}\label{WDecKG10}
\begin{split}
|I_q| & \lesssim t^{-2} | A_q | + t^{-2} | B_q | + t^{-2} |C_q| 
+ t^{-2} R, \qquad q>q_0,
\\
A_q(t,r) & := \frac{1}{\sqrt{r}} \int_0^\infty e^{itS(t,r,k)} 
 \sqrt{k}\, \partial_k^2 \Big( \frac{1}{(\partial_k S)^2} \Big)\, m(r,k) \, F(k) \, k \, 
  \varphi_q(k-k_0) \, dk, 
\\
B_q(t,r) & := \frac{1}{\sqrt{r}} \int_0^\infty  e^{itS(t,r,k)} 
  \frac{\sqrt{k}}{(\partial_kS)^2}\, m(r,k) \, \partial_k^2 F(k)  \, \varphi_q(k-k_0)\, dk, 
\\
C_q(t,r) & := \frac{1}{\sqrt{r}} \int_0^\infty  e^{itS(t,r,k)} 
 \sqrt{k}\, \partial_k \Big( \frac{1}{(\partial_kS)^2}\Big)   \,  m(r,k) \, \partial_k F(k) \, \varphi_q(k-k_0) \, dk,
\end{split}
\end{align}
where $R$ denotes all the terms where the symbol $m(k)$ or $\varphi_q(k-k_0)$ are differentiated at least one time, which are easier to estimate, since $|\partial_k m| \lesssim k^{-1}$ and $k \gtrsim 2^q$,
where $2^{-q}$ is the loss from differentiating $(\partial_k S)^{-1}$ or the cutoff $\varphi_q$.

We begin by estimating the first term in \eqref{WDecKG10}, 
\begin{align*}
|A_q(t,r)| &\lesssim \int_{0}^\infty 
 \left\vert  \partial_k^2\Big(\dfrac{1}{(\partial_kS)^2}\Big) \, \sqrt{k/r} \, \varphi_q(k-k_0)  \, F(k)\right\vert dk \, 
  \\ &\lesssim 2^{-4q}\langle k_0\rangle^6\int_0^\infty       \, \sqrt{k/r} \, \varphi_q(k-k_0)  \, \vert F(k)\vert \, dk \, 
  \\ &\lesssim t^{-1/2} 2^{-3q} \langle k_0\rangle^{4+1/2}{\| \jk^2 F \|}_{L^\infty_k}
\end{align*}
having used that, on the support of the integral, $k \approx k_0 = r\sqrt{k_0^2+2}/t$. 
Summing this last bound over $2^q \gtrsim t^{-1/2}\langle k_0\rangle^{3/2}$ yields the bound 
\[
|A_q(t,r)| \lesssim t^{-1/2} 
  (t^{1/2}\langle k_0\rangle^{-3/2} )^3 \langle k_0\rangle^{4+1/2}{\| \jk^{2}F \|}_{L^\infty_k}
  \approx t{\| \jk^{2}F \|}_{L^\infty_k},
\]
which, together with the first inequality in \eqref{WDecKG9}, is consistent with the bound \eqref{WDecKG3}.

The second term in \eqref{WDecKG10} is bounded similarly using Cauchy-Schwarz inequality
\begin{align*}
|B_q(t,r)| & \lesssim 2^{-2q} \langle k_0\rangle^6 \int_{0}^\infty 
  \sqrt{k/r} \, \varphi_q(k-k_0) \,\vert \partial_k^2 F(k)\vert  \, dk 
 \\ & \lesssim t^{-1/2}\langle k_0\rangle^{9/4}2^{-3q/2}  {\| \jk^{17/4}\partial_k^2 F \|}_{L^2_k},
\end{align*}
summing over $2^q \gtrsim t^{-1/2}\langle k_0\rangle^{3/2}$ gives a bound of $t^{1/4} {\| \jk^{17/4}\partial_k^2 F \|}_{L^2_k}$, which, in view of the first inequality in \eqref{WDecKG10} is consistent with the right-hand side of \eqref{WDecKG3}.

Finally, for the third term, proceeding as in the last case we find that
\begin{align*}
|C_q(t,r)| & \lesssim 2^{-3q}\langle k_0\rangle^6\int_0^\infty 
 \sqrt{k/r}  \, \varphi_q(k-k_0)\, |\partial_k F(k)| \,dk 
   \\ & \lesssim t^{-1/2}2^{-5q/2}\langle k_0\rangle^{15/4}  {\| \jk^{11/4}\partial_k F \|}_{L^2_k},
\end{align*}
where the numerology above comes from $6+\tfrac12-\tfrac{11}{4}=\tfrac{15}{4}$. 
Then, summing over $2^q \gtrsim t^{-1/2}\langle k_0\rangle^{3/2}$ 
gives a bound by $t^{3/4}\Vert \jk^{11/4}\partial_kF\Vert_{L^2}$,
which is again consistent with the desired \eqref{WDecKG3}. 
This concludes 
the proof for $w$.


\medskip

\noindent
{\it Estimate of $v^{(2)}$}.
In this case $rk \lesssim 1$ and the asymptotics in \eqref{prKG2Phi} reduce matters to estimating
by to the right-hand side of \eqref{WDecKG3} the following expression:
\begin{align}\label{WDecKG_v2_h}
h(t,r) & := 
\int_0^\infty k e^{it\sqrt{k^2+2}} \frac{r}{\jr} (1+\phi(r,k)) \jk \chi(rk) \, F(k) \, dk,
\end{align}
where we have replaced $a_1(k^2)\rho_1^\prime(k^2)$ with $\jk$, consistently with Proposition \ref{SMOp1_1},
and $\phi$ satisfies the usual symbol-type estimates.
We may again restrict the support of the integral to $k \gg t^{-1/2}$, 
otherwise the bound follows by direct integration, taking advantage of the length of the integration domain. 
We then insert a cutoff $\chi^c(ck\sqrt{t})$ for some small $c>0$.

For the sake of notation let us denote $S_*(k):=\sqrt{k^2+2}$.
Integrating by parts twice in $k$ using 
\[
e^{itS_*(k)} = (it\partial_kS_*)^{-1}\partial_k  e^{itS_*(k)},
\]
we obtain that 
\begin{align}\label{WDecKG11}
|h| & \lesssim t^{-2} | h_1 | + t^{-2} | h_2 | + t^{-2} |h_3| 
+ t^{-2} R, \nonumber
\\
h_1(t,r) & := \int_0^\infty \frac{r}{\jr} 
  \left|\partial_k\Big[\dfrac{1}{\partial_kS_*(k)}\partial_k\Big[\dfrac{k}{\partial_kS_*(k)}\Big]\Big]\Big( (1+\phi(r,k)) \jk \chi(rk) \chi^c(ck\sqrt{t}) 
    \Big)  \right|  |F(k)| \, dk,\nonumber
\\
h_2(t,r) & := \int_0^\infty \frac{r}{\jr} 
  \frac{k}{(\partial_kS_*)^2} \Big| (1+\phi(r,k)) \jk \chi(rk) \chi^c(ck\sqrt{t}) \Big| \,    |\partial_k^2 F(k)| \, dk,
\\
h_3(t,r) & := \int_0^\infty \frac{r}{\jr} 
   \left| \partial_k\Big[\frac{k}{(\partial_kS_*)^2} \Big]\Big((1+\phi(r,k)) \jk \chi(rk) \chi^c(ck\sqrt{t})\Big) \right| 
  \, |\partial_kF(k)|  \, dk,\nonumber
\end{align}
where $R$ denotes all the terms where the symbol $\phi(r,k)$ 
or the cutoffs are differentiated at least one time, which are easier to estimate.

Now observe that $\partial_k\big(\tfrac{1}{\partial_kS_*(k)}\partial_k\big(\tfrac{k}{\partial_kS_*(k)}\big)\big)=0$,
and hence $h_1=0$. 
For the second term, using 
that $r\langle r\rangle^{-1}\jk\lesssim 1$ on the support of the integral, 
a direct application of Cauchy-Schwarz yields
\begin{align*}
|h_2(t,r)| & \lesssim \int_{k \gtrsim t^{-1/2}} k^{-1}  |\jk^2\partial_k^2 F(k)| dk  \lesssim t^{1/4} {\| \jk^2\partial_k^2 F \|}_{L^2_k},
\end{align*}
which is acceptable. 
Similarly, using that $\vert \partial_k [\tfrac{k}{(\partial_kS_*)^2}]\vert\lesssim k^{-2}\jk^2$  
and Cauchy-Schwarz we obtain that
\begin{align*}
|h_3(t,r)| & \lesssim 
\int_{k \gtrsim t^{-1/2}} \dfrac{r}{\langle r \rangle}\dfrac{\jk^3}{k^2}|\partial_k F(k)| dk 
  \lesssim t^{3/4} {\| \jk^2\partial_k F \|}_{L^2_k},
\end{align*}
which along with the extra factor $t^{-2}$ in 
\eqref{WDecKG11} gives a bound that is compatible with \eqref{WDecKG3}, thus finishing the proof for $v^{(2)}$, and therefore concluding the proof of the proposition.
\end{proof}

\begin{rem}[The case $\vert n\vert\geq 2$]\label{section5_nvortex}
Let us briefly comment some of the adaptations needed for the vortex of degree $n$. 
In the case $rk\gtrsim 1$, the definition of $v^{(1)}$ in \eqref{WDecKG7} is replaced by
\[
v^{(1)}(t,r) := \frac{1}{\sqrt{r}} \int_0^\infty  \Phi_{1,n}(r,k^2) \chi^c(rk) e^{it\sqrt{k^2+2}}\, F(k) \, 2k a_{1,n}(k^2)\rho_{1,n}^\prime(k^2) \, dk
\]
and hence the corresponding definition of $w(t,r)$ in \eqref{WDecKG_v1_w}, according to the asymptotics
of $\Phi_{1,n}$ in \eqref{check_nvor_proof8}, is \[
w_n(t,r) := \frac{1}{\sqrt{r}} \int_0^\infty \frac{1}{\sqrt{k}} e^{it S(t,r,k)} (1+\phi(r,k)) \chi^c(rk) \, F(k) \, 2k \, dk,
\]
having used that $\vert a_{1,n}(k^2)\vert^2\rho_{1,n}'(k^2)\approx 1$,
which is exactly the same definition as that of $w(t,r)$ in \eqref{WDecKG_v1_w}.
Therefore, this case does not change at all for the vortex of degree $n$.

In the case of $v^{(2)}$, 
the definition of $h$ in \eqref{WDecKG_v2_h} is now replaced by 
\begin{align}\label{check_nvor_proof12}
h_n(t,r)=\int_0^\infty ke^{it\sqrt{k^2+2}}\dfrac{r^{n}}{\langle r\rangle^n}(1+\phi(r,k))\langle k\rangle^n \chi(rk)F(k)dk.
\end{align}
Hence, using the fact that $r^n\langle r\rangle^{-n}\langle k\rangle^n\lesssim1$
on the support of the integral, we infer that all the needed
calculations follow identically those for $n=1$, thus completing the proof for $\vert n\vert\geq 2$.
\end{rem}


\smallskip
\subsection{Wave-like case} 
We now look at linear solutions associated with $\mathcal{L}_2$. 
The following proposition is the main result of this subsection.

\begin{prop}\label{WDecWa1}
Let $v$ be a solution to $ \partial_t^2 v + \mathcal{L}_2 v = 0$, with initial conditions 
$v(0) = f$, and $\partial_tv(0)=g$. Let $F(k)$ and $G(k)$ be defined by  
\begin{align}\label{WDecWa2}
F(k) &:= \frac{1}{a_2(k)} \wtF_2 (\sqrt{r} f) (k) \qquad \hbox{ and } \qquad  G(k) := \frac{1}{a_2(k)} \wtF_2 (\sqrt{r} g) (k)
\end{align}
where $ \wtF_2$ is the Fourier transform associated to the operator $\mathcal{H}_2 = r^{1/2} \mathcal{L}_2 r^{-1/2}$. 
Then, the following inequality holds:
\begin{align}\label{WDecWa3}
\begin{split}
|v(t,r)| &\lesssim \dfrac{1}{\sqrt{t}}\Vert k\jk^{1/2+}F\Vert_{L^\infty_k}
 + \dfrac{1}{t}\Vert k\jk^{1/2}\partial_kF\Vert_{L^2_k}
 + \dfrac{1}{\sqrt{t}}\Vert \jk^{1/2+}G\Vert_{L^\infty_k}
 + \dfrac{1}{t}\Vert \jk^{1/2}\partial_kG\Vert_{L^2_k} .
\end{split}
\end{align}
\end{prop}


\smallskip
 

%


\begin{proof}
Recall that 
\begin{align}\label{WDecWa12}
v(t,r)
& =\dfrac{1}{\sqrt{r}}\int_0^\infty \! \Phi_2(r,k^2)e^{itk}\widetilde{F}_+(k)2k\rho_2'(k^2)dk 
  + \dfrac{1}{\sqrt{r}}\int_0^\infty \! \Phi_2(r,k^2)e^{-itk}\widetilde{F}_-(k)2k\rho_2'(k^2)dk,
\end{align}
with $\widetilde{F}_+$ and $\widetilde{F}_-$ as in \eqref{ssecKG20}. 
Since proving the claimed estimates for the terms with $\widetilde{F}_+$ and $\widetilde{F}_-$ 
is identical, we will only give the details for $\widetilde{F}_+$; 
furthermore, in view of the factor $k$ in the estimates for $\sqrt{r}f$ (cf. \eqref{WDecWa3}), 
the proof for the term with $k^{-1}\wtF_2 (\sqrt{r}g)$ 
is essentially the same as that of $\wtF_2 (\sqrt{r}f)$, 
and thus we restrict ourselves to working with $\wtF_2(\sqrt{r}f)$. 

With the above reductions, introducing a cutoff function $\chi$ as in the previous proofs, 
and with the notation \eqref{WDecWa2}, we see that it suffices to estimate the following two terms:
\begin{align}\label{WDecWa7}
v^{(1)}(t,r) & := \frac{1}{\sqrt{r}} \int_0^\infty  \Phi_2(r,k^2) 
  \chi^c(rk) e^{itk}\, F(k) \, 2k a_2(k^2)\rho_2^\prime(k^2) \, dk,
\\
\label{WDecWa8}
v^{(2)}(t,r) & :=  \frac{1}{\sqrt{r}} \int_0^\infty \Phi_2(r,k^2) 
  \chi(rk) e^{itk} \,  F(k) \, 2k a_2(k^2)\rho_2^\prime(k^2) \, dk,
\end{align}

\smallskip
\noindent
{\it Estimate of $v^{(1)}$}.
Using the expansion \eqref{prKG2Phi} for $\Phi_2$ when $rk \gtrsim 1$ reduces matters to estimating by the right-hand side of \eqref{WDecWa3}, the following expression
\begin{align*}
w(t,r) & := \frac{1}{\sqrt{r}} \int_0^\infty \frac{1}{\sqrt{k}} e^{i S(t,r,k)} (1+\phi(r,k))
  \chi^c(rk) \, F(k) \, 2k \, dk,
  \ \quad S(t,r,k) := (t - r)k,
\end{align*}
with $\phi(r,k)$ as in \eqref{prKG2phiest}, and having used $|a_2(k^2)|^2\rho_2^\prime(k^2) = 1/(4\pi)$ 
according to \eqref{propSMrho}. 
The contribution of the term associated with $e^{irk}$ is easier to estimate so we disregard it. 
%
%
%

For the sake of simplicity let us denote 
\begin{align*}
L(t,r)&:=\int_0^\infty \dfrac{1}{\sqrt{rk}}e^{itS(t,r,k)}m(r,k) F(k)\, k \, dk,
\qquad m(r,k) := \chi^c(rk)\big(1+\phi(r,k)\big).
\end{align*}

\smallskip
\noindent
{\it Case $r\geq t/10$}. This case follows from direct integration: 
\begin{align*}
\vert L\vert \lesssim \dfrac{1}{\sqrt{t}}\Vert k\jk^{1/2+}F\Vert_{L^\infty_k}\int_0^\infty 
  \dfrac{1}{\sqrt{k}}\jk^{-(1/2+)}\vert m(r,k)\vert\,dk  
  \lesssim \dfrac{1}{\sqrt{t}}\Vert k\jk^{1/2+}F\Vert_{L^\infty_k}.
\end{align*}

\smallskip
\noindent
{\it Case $r<t/10$}.
Integration by parts in $k$ reduces the problem to bounding the following integrals \begin{align*}
L_1 &:= \dfrac{1}{t-r}\int_0^\infty \dfrac{1}{\sqrt{rk}}e^{iS(t,r,k)}m(r,k)\big(\partial_kF(k)\big) \, k \, dk,
\\ 
L_2 &:= \dfrac{1}{t-r}\int_0^\infty e^{iS(t,r,k)}\partial_k\bigg[\dfrac{1}{\sqrt{rk}}m(r,k)k\bigg]F(k)\,dk.
\end{align*}
Regarding $L_1$, by Cauchy-Schwarz inequality, we find that \begin{align*}
\vert L_1\vert \lesssim \dfrac{1}{\vert t-r\vert}\int_0^\infty \dfrac{1}{\sqrt{rk}}
  \big\vert m(r,k) \,k\, \partial_kF(k)\big\vert \,dk
  \lesssim \dfrac{1}{\vert t-r\vert}\Vert k\jk^{1/2}\partial_kF\Vert_{L^2_k}.
\end{align*}
Similarly, for $L_2$ we have
\begin{align*}
\vert L_2\vert \lesssim \dfrac{1}{\vert t-r\vert }\int_0^\infty \Big\vert\dfrac{m(r,k)}{2k\sqrt{rk}}+
  \dfrac{\partial_km(r,k)}{\sqrt{rk}}\Big\vert \,\vert k\, F(k)\vert\,dk
  \lesssim \dfrac{1}{\vert t-r\vert} \Vert kF\Vert_{L^\infty_k}, 
\end{align*}
having used that $\Vert (k\sqrt{rk})^{-1}\chi^c(rk)\Vert_{L^1_k}\lesssim 1$ 
as well as $\vert \partial_k\phi(r,k)\vert \lesssim k^{-1}$. 
Recalling that in this case $\vert t-r\vert >\tfrac{1}{2}t$ 
we see that these bounds are consistent with \eqref{WDecWa3}.
Thus, \[
\vert v^{(1)}(t,r)\vert \lesssim \dfrac{1}{\sqrt{t}}\Vert k\jk^{1/2+}F\Vert_{L^\infty_k}+\dfrac{1}{t}\Vert k\jk^{1/2}\partial_kF\Vert_{L^2_k}
\]
concluding the analysis for $v^{(1)}(t)$.

\medskip
\noindent
{\it Estimate of $v^{(2)}$}.
In this case $rk \lesssim 1$ and the asymptotics in \eqref{prKG2Phi} reduces matters to estimating, 
by to the right-hand side of \eqref{WDecWa3}, the expression
\begin{align*}
h(t,r) & :=
\int_0^\infty 2k e^{itk} \frac{r}{\jr} (1+\phi(r,k)) \jk \chi(rk) \, F(k) \, dk,
\end{align*}
where we have replaced $a_2(k^2)\rho_2^\prime(k^2)$ with $\jk$, for notational convenience
consistently with Proposition \ref{propSM2},
and where $\phi$ satisfies the usual symbol-type estimates.

Note that if $r \geq  t/10$, then 
\begin{align*}
\vert h(t,r)\vert \lesssim \int_{k\lesssim1/r} k\, \vert \chi(rk)\,F(k)\vert \,dk \lesssim \dfrac{1}{r }\Vert k\, F\Vert_{L^\infty_k}\lesssim \dfrac{1}{t}\Vert k \, F\Vert_{L^\infty_k}.
\end{align*}
Therefore, in the sequel we can assume that $r<t/10$. 
In this case we seek to integrate by parts. 
Prior to that note that, just as above, in the region $k\lesssim 1/t$, 
\begin{align*}
\int_{k\lesssim 1/t} \dfrac{r}{\langle r\rangle}\jk \,\vert \chi(rk)\,k F(k)\vert\,dk
  \lesssim \dfrac{1}{t}\Vert k F\Vert_{L^\infty}.
\end{align*} 
Hence, we can add a cutoff $\chi^c(ckt)$ with $c$ small. 
Integrating by parts leads to 
\begin{align*}
|h| \lesssim t^{-1} | h_1 | + t^{-1} | h_2 |, 
\qquad 
h_1(t,r) & :=
\int_0^\infty  \frac{r}{\jr}\partial_k\Big[2k\jk (1+\phi(r,k)) \chi^c(ckt) \chi(rk)\Big] 
  \, F(k) \, dk,
\\
h_2(t,r) & := \int_0^\infty \frac{r}{\jr} 2k \jk (1+\phi(r,k)) \chi(rk) \, \partial_kF(k) \, dk.
\end{align*}

We can estimate the first term by
\begin{align*}
\vert h_1\vert & \lesssim \int_0^\infty \dfrac{r}{\langle r\rangle} \jk \vert \chi^c(ckt) F(k)\vert dk \lesssim \log\langle t\rangle\Vert k\jk^{0+}F\Vert_{L^\infty_k},
\end{align*}
having used the bound $r\langle r\rangle^{-1}\jk\lesssim1$ 
and $\vert\partial_k\phi(r,k)\vert\lesssim k^{-1}$. 
Finally, using Cauchy-Schwarz we see that
\begin{align*}
\vert h_2 \vert & \lesssim \dfrac{r}{\langle r\rangle }\int_0^\infty k\jk \vert \chi(rk)\partial_kF(k)\vert dk\lesssim \Vert k \jk^{1/2}\partial_kF\Vert_{L^2_k}.
\end{align*}
Gathering all the above estimates, we obtain that 
\[
\vert v^{(2)}(t,r)\vert\lesssim \dfrac{\log\langle t\rangle}{t}\Vert k\jk^{0+}F\Vert_{L^\infty_k}+\dfrac{1}{t}\Vert k\jk^{1/2}\partial_kF\Vert_{L^2_k},
\]
which concludes 
the proof of the lemma.
\end{proof}

\medskip
\section{Spectral Analysis}\label{AppSpec}

\subsection{Absence of unstable spectrum}\label{SecSpec}
In this subsection we give a proof of the absence of negative discrete spectrum 
for the operator $-\mathcal{L}$, see the definition in \eqref{alpha_beta_system}, that is for 
$\mathcal{L}_1$ and $\mathcal{L}_2$. This fact is already known from the earlier works of \cite{GusVort0,WX}.
However, in our setting we can provide a particularly simple proof of independent interest, 
based on a super-symmetric factorization (or Darboux transformation); this type of technique 
has proven to be quite helpful in other contexts,
see, for example, \cite{KMMreview,KowMarMun,LSch,KMS,CGNT07,CM22}.


\begin{prop}\label{Propspec0}
The operators $\mathcal{L}_1$ and $\mathcal{L}_2$ have no non-positive eigenvalues.
Moreover, the operator $\mathcal{L}_2$  has absolutely continuous spectrum $[0,\infty)$
and zero energy is a resonance.
The absolutely continuous spectrum of the operator $\mathcal{L}_1$ is $[2,\infty)$, and the bottom is a resonance. 
Finally, the discrete spectrum of $\mathcal{L}_1$ is contained in $(c_0,2]$ for some $c_0 \in (0,2)$.
\end{prop}




\begin{proof}
That the absolutely continuous spectrum $\sigma_{\text{ac}}(\mathcal{L}_2)=[0,\infty)$ 
follows from writing $\mathcal{L}_2$ as \[
\mathcal{L}_2=-\partial_r^2-\dfrac{1}{r}\partial_r+\dfrac{1}{r^2}+V_0(r)=:\mathcal{L}_2^0+V_0(r),
\]
and noticing that $V_0(r)\in L^2(rdr)$, thus it is $\mathcal{L}_2^0$-compact, 
and that the spectrum of $\mathcal{L}_2^0$ can be explicitly calculated by solving the ODE, 
$\sigma_{\text{ac}}(\mathcal{L}_2^0)=[0,\infty)$.

As we have already mentioned, $U$ is a zero-energy resonance, $\mathcal{L}_2U(r)=0$. 
Then, we define the operator 
\begin{align*}
\mathcal{D}:=\partial_r- U'(r) [U(r)]^{-1}
\end{align*}
which satisfies $\mathcal{D}U=0$.
Its adjoint is $\mathcal{D}^*=-\partial_r-r^{-1} - U'(r) [U(r)]^{-1}$,
and an explicit calculation, using \eqref{eq_vortex}, shows that
\[
\mathcal{D}^*\mathcal{D}=-\partial_r^2-\dfrac{1}{r}\partial_r+\dfrac{1}{U(r)}\big(U''(r)+\tfrac{1}{r}U'(r)\big)
  =-\partial_r^2-\dfrac{1}{r}\partial_r+\dfrac{1}{r^2}-(1-U^2) = \mathcal{L}_2.
\]
Next, we consider the `conjugated operator'
\begin{align}\label{L2star}
\begin{split}
\mathcal{D}\mathcal{D}^*&=-\partial_r^2-\dfrac{1}{r}\partial_r
  +\dfrac{1}{r^2}+\dfrac{U'(r)}{rU(r)}+\dfrac{2U'(r)^2}{U(r)^2}-\dfrac{U''(r)}{U(r)}
\\ 
& = -\partial_r^2-\dfrac{1}{r}\partial_r+\dfrac{2U'(r)}{rU(r)}+\dfrac{2U'(r)^2}{U(r)^2}+(1-U(r)^2)
  =: -\Delta_r+V(r)=:\mathcal{L}_2^\star.
\end{split}
\end{align}
We then notice that the potential term $V(r)$ satisfies, for all $r>0$,
\begin{align}\label{pot_conj}
V(r) = \dfrac{2U'(r)}{rU(r)}+\dfrac{2U'(r)^2}{U(r)^2}+(1-U(r)^2)>0.
\end{align}
Moreover $V \approx O(r^{-2})$ for small $r$.


Next, 
assume that there exists an eigenfunction $f\in D(\mathcal{L}_2) 
:= \{f\in L^2(rdr): \mathcal{L}_2f \in L^2(rdr) \}$, 
\[
\mathcal{L}_2f=-\lambda^2 f, \qquad \hbox{ and hence } \qquad \mathcal{D}\mathcal{L}_2f=-\lambda^2 \mathcal{D}f.
\]
It then follows from $f\in D(\mathcal{L}_2)$ that $\mathcal{D}f\in L^2(rdr)$. 
Moreover, we have $f(r)\not\equiv U(r)$ and hence, using that $\mathcal{L}_2=\mathcal{D}^*\mathcal{D}$,
we infer that $g(r):=\mathcal{D}f(r)\not\equiv 0$ solves 
\begin{align}\label{eigen_Lstar}
\mathcal{L}_2^\star g=\mathcal{D}\mathcal{D}^*g=-\lambda^2 g.
\end{align}
Note that $f\in D(\mathcal{L}_2)$ requires in particular that $\tfrac{1}{r}f\in L^2(rdr)$,
and hence $g\in L^2(rdr)$ which, along with \eqref{eigen_Lstar}, imply $g\in D(\mathcal{L}_2 ^\star)$. 
The equation \eqref{eigen_Lstar} is then a contradiction to the identity \eqref{L2star} 
for $\mathcal{L}_2^\star$ with $V>0$. 
%
%

Having shown that 
that $\mathcal{L}_2$ does not have any negative eigenvalues, the same property follows easily for $\mathcal{L}_1$
from writing $\mathcal{L}_1 = \mathcal{L}_2 + 2U^2$, and applying the variational principle.
The same identity, and the fact that $U^2-1 \in L^2(rdr)$, give that $\sigma_{\text{ac}}(\mathcal{L}_1)= [2,\infty)$.

To show that $\sigma(\mathcal{L}_1) \subset (c_0,\infty)$ for some constant $c_0>0$,
one can again use the variational principle after verifying that $\mathcal{L}_1 U' > 0$.
We leave the details to the reader, and we are actually going to prove a stronger result in Theorem \ref{theoLT}. 
Finally, that the bottom of the continuous spectrum is a resonance, follows from our construction
of $\Phi_1^{(0)}$ in Lemma \ref{lem_p1_t1}, which gives that $r^{-1/2}\Phi_1^{(0)}$ is a nontrivial bounded solution 
of $\mathcal{L}_1f = 2f$.
\end{proof}


\smallskip
\subsection{Lieb-Thirring inequalities and the spectrum of $\mathcal{L}_1$}\label{secLT}
We will now provide a lower bound for the first eigenvalue of the operator $\mathcal{L}_1$
as a consequence of a Lieb-Thirring inequality for the second moment of the negative spectrum of 
$\mathcal{L}_1-2$, some estimates on the vortex, and some numerical computations.

To give some context, we recall the Lieb-Thirring inequality from the seminal paper \cite{LiebThi}.
Given a Schr\"odinger operator $H = -\Delta + V(x)$, $V:\mathbb{R}^d \rightarrow \R$,
let $E_k(H)$, $k\in\mathbb{N}$, denote the negative eigenvalues of $H$ on $L^2(\R^d)$ 
repeated according to their multiplicities, and arranged in non-decreasing order,
with the convention that $E_k(H)=0$ if $H$ has less than $k$ negative eigenvalues.
Then, for $V\in L^{1+d/2}(\R^d)$, which in particular guarantees that $H$ is a self-adjoint lower bounded operator 
on $L^2(\R^d)$,
there exist a constant $L_{1,d} < \infty$ (the Lieb-Thirring constant) such that
\begin{align}\label{LT1}
\sum_{k} |E_k(H)| \leq L_{1,d} \int_{\R^d} V_-(x)^{1+d/2} \, dx,
\end{align}
where $V_-=\max(-V,0)$ is the negative part of $V$.
A variant of the above inequality is the following:
for any $\gamma \geq 1/2$ when $d=1$, and any $\gamma>0$ when $d=2$, and $\gamma\geq 0$ for $d\geq 3$,
there exists a constant $L_{\gamma,d}$ such that
\begin{align}\label{LTgamma}
\sum_{k} |E_k(H)|^\gamma \leq L_{\gamma,d} \int_{\R^d} V_-(x)^{\gamma+d/2} \, dx. 
\end{align}
We refer to the book \cite{FrankLTBook} and to the survey \cite{FrankLTsurvey} for an account of the illustrious history
and large literature about these inequalities, as well as for the best known results
concerning the constants $L_{\gamma,d}$.
In our proof we will use a specific result of Ekholm-Frank \cite{EkhFra} combined with the 
result on the optimal constant in \eqref{LTgamma}, for $\gamma \geq 3/2$, from Laptev-Weidl \cite{LapWei}. 
This is our result:

\def\Upzero{0.5832}

\def\bottomx{\lambda_0}
\def\bottom{1.3326
}


\def\turningx{r_0}
\def\turning{0.614489}

\def\upperx{r_1}
\def\upper{7}

\begin{thm}\label{theoLT}
The discrete spectrum of the operator $\mathcal{L}_1$ 
is contained in $(\lambda_0,2)$ with $\lambda_0 = \bottom$.
\end{thm}

\begin{proof}
Recall that $\mathcal{L}_1 = -\Delta_r + r^{-2} - (1-3U^2(r))$
and $\mathcal{H}_1 = - \partial_r^2 + (3/4)r^{-2} - (1-3U^2(r))$.
It suffices to prove the statement for $\mathcal{H}_1$. For this we want to apply a Lieb-Thirring inequality
to $\mathcal{H}_1-2$.
We then write the operator as
\begin{align*}
\mathcal{H}_1-2 = h_0 + V, \qquad h_0 := -\partial_r^2 - (1/4)r^{-2}, \qquad V := r^{-2} - (1-3U^2) 
  \in L^1_{\mathrm{loc}}(\R_+),
\end{align*}
associated with the closure of the quadratic form 
\begin{align*}
\int_0^\infty \Big( |f'|^2 + \frac{3}{4r^2}f^2 - (1-3U^2(r)) f^2 \Big) \, dr,
\end{align*}
on $C_c^\infty((0,\infty))$;
the quadratic form is semi-bounded since, for example, $(\mathcal{H}_1 f,f) \geq 0$ 
in view of Proposition \ref{Propspec0}.
We can then apply Theorem 1.6 in Ekholm-Frank \cite{EkhFra} (see also the statement of that theorem,  
Remark 1.9, and the comments after Lemma 3.1 in the cited paper); in particular,
the quadratic form above is closed and lower semi-bounded on the domain of $h_0$
and the corresponding self-adjoint operator $\mathcal{H}_1-2$  satisfies
\begin{align}\label{LTEkhFra}
\mathrm{Tr} \big( \mathcal{H}_1-2 \big)_-^\gamma \leq 2\pi L_{\gamma,2} 
  \int_{0}^\infty 
  \Big[  \frac{1}{r^2} - 3\big( 1-U^2(r)) \Big]_-^{\gamma + 1} \, r dr
\end{align}
with 
$\gamma>0$ 
(this is the case $\alpha=1$ in Theorem 1.6 of \cite{EkhFra}).
Here $L_{\gamma,2}$ is the best constant in \eqref{LTgamma}.
It turns out, from numerical computations, 
that the most effective bound in our case will be provided by a value of $\gamma\approx 2$,
and we will use exactly $\gamma =2$ for our estimates.

For $\gamma\geq 3/2$, and any $d\geq 1$, one has the optimal Lieb-Thirring constant (which coincides
with the semi-classical constant)
\begin{align*}
L_{\gamma,d} 
  = \frac{1}{(4\pi)^{d/2}} \frac{\Gamma(\gamma+1)}{\Gamma(\gamma+1+\tfrac{d}{2})}
\end{align*}
from the work of Laptev-Weidl \cite{LapWei}. 
In particular,
$L_{\gamma,2} = (4\pi)^{-1} (\gamma+1)^{-1},$
and, using \eqref{LTEkhFra}, we can estimate 
\begin{align}\label{LTpr1}
\mathrm{Tr} \big( \mathcal{H}_1-2 \big)_-^2 
  \leq \frac{1}{6} 
  \int_0^\infty \Big[ \frac{1}{r^2} - 3\big( 1-U^2(r) \big) \Big]_-^{3} \,r dr. 
\end{align}
We aim to estimate the right-hand side of \eqref{LTpr1}
using some numerical computations for parts of the argument.

To accurately resolve the values of the vortex $U$ numerically,
we solve the ODE \eqref{eq_vortex} with Mathematica to find
solutions $U_\pm$ such that $U_-(r) < U(r) < U_+(r)$ for all $r$.
These solutions correspond to values $c_\pm$ with $c_- < a < c_+$, where 
$a= U'(0) \approx \Upzero$ and $c_+-c_- < 10^{-10}$.
$U_+$ is guaranteed to be above the vortex since it diverges as $r$ increases, while 
$U_-$ is guaranteed to be below since it starts oscillating for $r$ large \cite{HerHer}.
We can then use these numerical solution to estimate quantities that involve the exact vortex solution,
up to Mathematica's numerical errors.
See Figure \ref{fig1} for a plot of $U_\pm$.

\begin{figure}
\begin{center}
\includegraphics[scale=0.645]{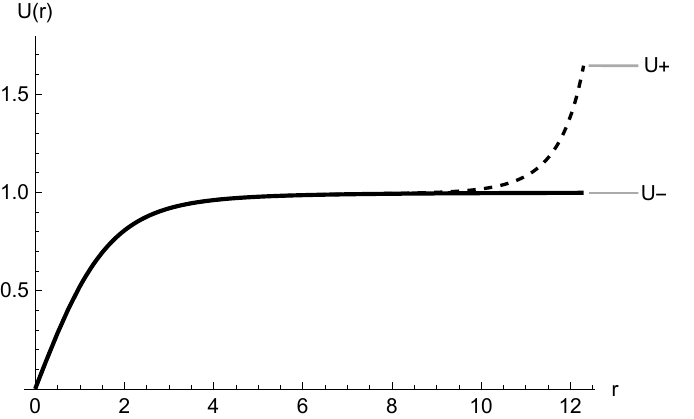}
\quad \includegraphics[scale=0.645]{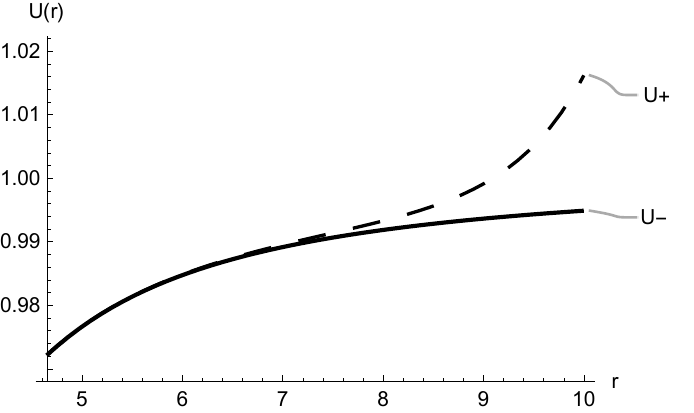}
\caption{Graphs of the numerical solutions $U_+(r)$ and $U_-(r)$ used in the calculation of 
$A_\gamma$ in 
\eqref{LTpr1g}. 
These are solutions of the differential equation \eqref{eq_vortex}
with initial conditions $U_{\pm}(0)=0$ and $U'{\pm}(0)=c_\pm$. 
The solid line represents the solution, that we denote by $U_-(r)$, with $c_-$ 
slightly smaller than the theoretical value of $a=U'(0)$ for the vortex. 
The dashed line corresponds to the solution, that we denote by $U_+(r)$, for $c_+>a$. 
The difference is $c_+-c_-< 10^{-10}$. 
Due to the uniqueness of solutions, the vortex $U$ must be 
in between these two graphs. 
At $r=7$ the difference $U_+-U_-$ is approximately $10^{-4}$. 
We use the numerics only up to $r=7$ and
after that we integrate using Claim \eqref{1-U2}.
}
\label{fig1}
\end{center}
\end{figure}

We begin by observing numerically that the potential $V = r^{-2} -  3\big( 1-U^2)$ is positive for 
$ r \leq \turningx := \turning$, since this is true for $V_- := r^{-2} -  3\big( 1-U_-^2)$; see Figure \ref{fig2}.
Thus, integrating the absolute value of $V_-^3$ on $(r_0,\infty)$ 
guarantees that the entire support of the negative part of the potential in \eqref{LTEkhFra} is 
taken into consideration.
Using this, and defining $\upperx := \upper$, we estimate
\begin{align}\label{LTpr1g}
\begin{split}
& \mathrm{Tr} \big( \mathcal{H}_1-2 \big)_-^\gamma 
  \leq L_{\gamma,2} \cdot  2\pi \int_{r \geq \turningx} 
  \Big( 3\big( 1-U^2(r)) - \frac{1}{r^2} \big) \Big)^{\gamma + 1} \, r dr
  = \frac{1}{2(\gamma+1)} \big( A_\gamma + R_\gamma),
  \\
& A_\gamma := \int_{\turningx \leq r \leq \upperx} \Big( 3\big( 1-U^2(r)) - \frac{1}{r^2} \big) \Big)^{\gamma + 1} \, r dr,
\\
& R_\gamma := \int_{r \geq \upperx} \Big( 3\big( 1-U^2(r)) - \frac{1}{r^2} \big) \Big)^{\gamma + 1} \, r dr.
\end{split}
\end{align}
We left $\gamma$ unspecified here for 
ease of comparison of the computations between different values of $\gamma$,
should the reader wish to check the almost optimality of our choice $\gamma =2$.
%

The integral $A_\gamma$ can be rigorously upperbounded, and then estimated, using
the numerical solution $U_-$ instead of $U$;
from this we obtain
\begin{align}\label{Agamma}
& \left. \frac{1}{2(\gamma+1)} A_\gamma \right|_{\gamma = 2}  \approx 
0.44515. 
\end{align}

Next, we claim that 
\begin{align}\label{1-U2}
1-U^2(r) \leq 1.1 r^{-2},
 \qquad \mbox{for} \quad r \geq r_1. 
\end{align}
See Figure \ref{fig3}.
We postpone the proof of \eqref{1-U2} for the moment, and proceed to use it to estimate the remainder 
\begin{align}\label{R2}
& R_2 \leq 
(2.3)^3 \int_{r \geq \upperx}\frac{1}{r^5} \, dr 
\approx 0.001267.
\end{align}
%
%

%
%
%
%
%
%
%
%
%
Combining \eqref{Agamma} and \eqref{R2} shows that 
$$ \mathrm{Tr} \big( \mathcal{H}_1-2 \big)_-^2  \leq  
0.44536.$$
Therefore, we must have that the lowest eigenvalue of $\mathcal{L}_1 -2$ 
is larger than $-\sqrt{0.44536}$, hence 
the lowest eigenvalue of $\mathcal{L}_1$ is larger than $\bottomx =  2 -  \sqrt{0.44536} \approx \bottom$ as claimed.


\begin{figure}
\begin{center}
\includegraphics[scale=0.645]{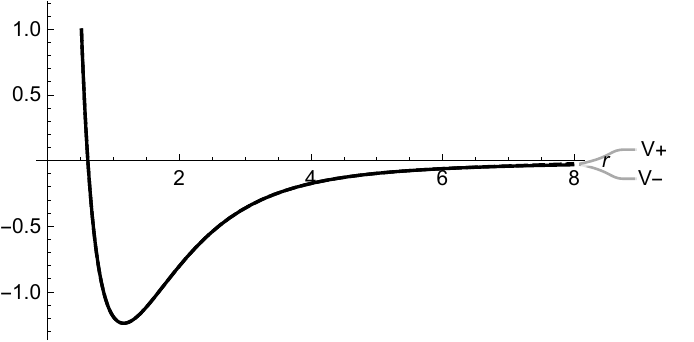}
\quad \includegraphics[scale=0.645]{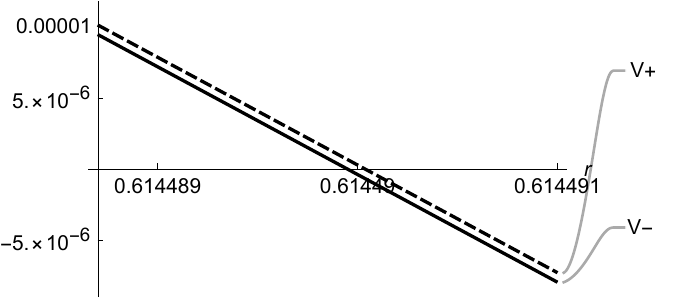}
\caption{On the left: a plot of the potentials $V_\pm = 1/r^2 - 3 (1 - U_\pm^2)$,
which cannot be distinguished at this coarse scale.
\\
On the right: plots of the potentials $V_\pm$ near $r_0=\turning$,
the point at which we claim that $V$ is still positive.}
\label{fig2}
\end{center}
\end{figure}

\smallskip
{\it Proof of \eqref{1-U2}}.
Observe first that it suffices to show that $1-U \leq 0.55r^{-2}$ for all $r\geq r_1 = \upper$. 
One can verify numerically that this desired inequality is true at $r=r_1$ (see Figure \ref{fig3}),
while it holds at $r=\infty$ in view of the expansion \eqref{asympt_U_infty}. 
Then we consider the smooth radial function 
$f := U - (1-0.55r^{-2})$ on the interval $I:= (\upper,\infty)$ and claim that it satisfies 
\begin{align}\label{1-U2feq}
\Delta f \leq c(r) f + f^3, \qquad c(r) \geq 0, \qquad r \in I.
\end{align}
To verify \eqref{1-U2feq} we use \eqref{eq_vortex} and calculate
\begin{align*}
f_{rr} + \dfrac{1}{r} f_r & = \dfrac{1}{r^2}U - (1-U^2)U + \frac{2.2}{r^4} 
\\
& = \dfrac{1}{r^2} f + \frac{1}{r^2} - U  + f^3 + 3 U (U-f)f + (U-f)^3 + \frac{1.65}{r^4}
\\
& = \dfrac{1}{r^2} f + f^3 + \big[3U(U-f) - 1\big] f + 
  \big[ f-U + \frac{1}{r^2} + \frac{1.65}{r^4} + (U-f)^3 \big]
\\  
  & := \dfrac{1}{r^2} f + f^3 + a(r)f + B(r)
\end{align*}
where $ a(r) := 3U(U-f) - 1,$
and
\begin{align*}
B(r) & :=  f - U + r^{-2} + 1.65 r^{-4} + (U-f)^3
 \\
 & = - (1-0.55r^{-2}) + r^{-2} + 1.65r^{-4} + \big( 1 - 0.55r^{-2} \big)^3
\\ 
& = -0.1 r^{-2} + (1.65 + 3(0.55)^2) r^{-4} - (0.55)^3 r^{-6}.
\end{align*}
Using that $U \geq 1/2$ in $I$, we immediately see that $a = 3U(1-0.55r^{-2}) - 1 > 0$ in $I$,
and one can also directly verify that $B<0$ in $I$.
These show the validity of \eqref{1-U2feq} with $c(r) = a(r) + r^{-2}$.
To conclude, we observe that \eqref{1-U2feq} implies that $f$ cannot have a negative minimum in the interval $I$ and,
since it is non-negative at the boundary, it is non-negative in $I$, hence the desired conclusion.
\begin{figure}
\begin{center}
\includegraphics[scale=0.7]{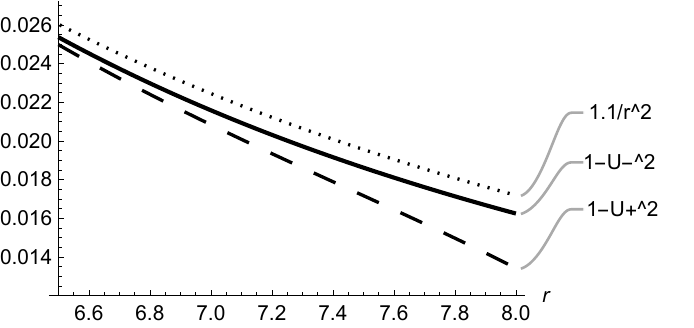}
\caption{A visualization of Claim \eqref{1-U2} near $r=7$.
The solid line represents $1-U_-^2(r)$, the dashed line $1-U_+^2(r)$
and the dotted line is $\tfrac{1.1}{r^2}$. 
The same quantity in the case of the vortex, i.e., $1-U^2(r)$, must be in between the solid and the dashed line,
hence the claimed boundary condition, $f(r_1) > 0$, for the solution of \eqref{1-U2feq} is verified.}
\label{fig3}
\end{center}
\end{figure}
\end{proof}

\appendix

\medskip

\section{Additional and supporting material}


\smallskip

\subsection{The Heat Flow}\label{SecHeat}
For the heat flow \eqref{Heat}, the linearized equations at a perturbation of a vortex are
$\partial_t (\alpha,\beta) = \mathfrak{L}(\alpha,\beta)$,
where $\mathfrak{L}$ is defined in \eqref{full_lin_Re_Im}. 
In particular, for co-rotational perturbations
\begin{align}\label{Heatlin}
\partial_t \left( \begin{matrix}
\alpha \\ \beta 
\end{matrix}\right) 
=
\left(\begin{matrix}
\Delta+1-\tfrac{1}{r^2}-3U^2(r) & 0 \\ 0 & \Delta+1-\tfrac{1}{r^2}-U^2(r)
\end{matrix}\right)
\left( \begin{matrix}
\alpha \\ \beta 
\end{matrix}\right) .
\end{align}
Below we show that $\alpha,\beta$ satisfy decay estimate consistent
with a parabolic flow. This result is not new as it is already contained in the work of Weinstein-Xin \cite{WX}.
However, our proof is different and more direct, based on the use of the distorted Fourier transforms built
in Sections \ref{secL2}-\ref{secL1}.
Once the linear decay estimates from Proposition \ref{Heatprop} 
are obtained, 
a nonlinear stability result can be proved based
on $L^p-L^q$ decay estimates and a bootstrap argument via Duhamel's formula as in \cite[Section 4]{WX}.

\smallskip
\begin{prop}\label{Heatprop}
With the definitions of $\mathcal{L}_i$, $i=1,2$, in \eqref{alpha_beta_system}, we have,
for radial $f=f(r)$, $t>0$,
\begin{align}\label{Heatconc1}
{\big\| e^{-t\mathcal{L}_1} P_c f \big\|}_{L^\infty(\R^2)} \lesssim e^{-2t} t^{-1} {\| f \|}_{L^1(rdr)},
\end{align}
and
\begin{align}\label{Heatconc2}
{\big\| e^{-t\mathcal{L}_2} f \big\|}_{L^\infty(\R^2)} \lesssim 
  t^{-1} {\| f \|}_{L^1(rdr)}.
\end{align}
Moreover, with the definitions of the Littlewood-Paley projections in \eqref{LPF1'L}, 
the following estimates hold for all $\ell \in \Z$:
\begin{align}\label{Heatconc1'}
{\big\| e^{-t\mathcal{L}_1} P^{\mathcal{L}_1}_\ell f \big\|}_{L^p(\R^2)} \lesssim 
  e^{-2t} t^{1/p - 1/q} {\| f \|}_{L^q(rdr)}, \quad 1 \leq q \leq p \leq \infty,
 \end{align}
and
\begin{align}\label{Heatconc2'}
{\big\| e^{-t\mathcal{L}_2} P^{\mathcal{L}_2}_\ell f \big\|}_{L^p(\R^2)} \lesssim 
  t^{1/p - 1/q} {\| f \|}_{L^q(rdr)}, \quad 1 \leq q \leq p \leq \infty.
\end{align}
\end{prop}

\begin{rem}
Combining \eqref{Heatconc1} with our Theorem \ref{theoLT} which quantifies the gap between $0$ and the discrete spectrum,
one obtains, for $\lambda_0 = 1.332$,
\begin{align*}
{\big\| e^{-t\mathcal{L}_1} f \big\|}_{L^\infty(\R^2)} \lesssim e^{-\lambda_0t} t^{-1} {\| f \|}_{L^1(rdr)}.
\end{align*}
\end{rem}






\begin{proof}[Proof of Proposition \ref{Heatprop}]
We only look at the case of $\mathcal{L}_2$ since the case of of $\mathcal{L}_1$ is almost identical.
We set up the proof similarly to that of the relativistic cases in Section \ref{secdecay}
and Propositions \ref{propKG1dec} and  \ref{propKG2dec}.
Using \eqref{defH_2} and the representation \eqref{FT2'}, 
we can write
\begin{align}\label{Heatpr1}
\begin{split}
\big(e^{-t\mathcal{L}_2} f \big)(t,r) & = \int_0^\infty K(t,r,s) f(s) s \, ds, 
\\
K(t,r,s) & :=  \frac{2}{\sqrt{rs}} 
  \int_0^\infty \Phi_2(r,k^2) \Phi_2(s,k^2) e^{-tk^2} \, k \rho_2^\prime(k^2) \, dk.
\end{split}
\end{align}
Note that $K$ is symmetric with respect to exchanging $r$ and $s$.
To show \eqref{Heatconc2} it will then suffice to prove
\begin{align}\label{Heatprmain1}
\sup_{r>s\geq 0} \big| K(t,r,s) \big| \lesssim t^{-1}.
\end{align}

Let us recall here for convenience \eqref{prKG2Phi}, that is
\begin{align}\label{HeatprPhi}
\Phi_2(r,k^2) = \left\{ 
\begin{array}{ll}
r^{3/2} \big( 1 + \phi_{<}(r,k)), &  rk \lesssim 1, \quad r \lesssim 1,
\\
r^{1/2} \big( 1 + \phi_{>}(r,k)), &  rk \lesssim 1, \quad r \gtrsim 1,
\\
\dfrac{1}{\sqrt{k}} \big[ a_2(k^2) e^{irk} \big( 1 + \phi_+(r,k)) 
  + \overline{a_2}(k^2) e^{-irk} \big(1 + \phi_-(r,k)) \big], &  rk \gtrsim 1,
\end{array} 
\right.
\end{align}
where $\phi \in \{ \phi_{<}, \phi_{>}\phi_-,\phi_+ \}$ is a generic function satisfying 
\begin{align}\label{Heatprphiest}
|\partial_k^\alpha \phi(y,k)| \lesssim k^{-\alpha}, \quad \alpha = 0,1,2.
\end{align}
As before, 
we let $\chi = \chi(x)$, $x\geq 0$ be a smooth function equal to $1$ in $[0,1]$, 
decreasing, and vanishing for $x \geq 2$, define $\chi^c := 1-\chi$ and split
\begin{align}
\nonumber
\tfrac{1}{2} K(t,r,s) & = K_{1}(t,r,s) + K_{2}(t,r,s) + K_{3}(t,r,s),
\\
\label{HeatprKer1}
K_1(t,r,s) & :=  \frac{1}{\sqrt{rs}} 
  \int_0^\infty \Phi_2(r,k^2) \Phi_2(s,k^2) e^{-tk^2} \,  k \rho_2^\prime(k^2) \, 
  \chi^c(sk) \, dk,
\\
\label{HeatprKer2}
K_2(t,r,s) & :=  \frac{1}{\sqrt{rs}} 
  \int_0^\infty \Phi_2(r,k^2) \Phi_2(s,k^2) e^{-tk^2} \, k \rho_2^\prime(k^2) \, 
  \chi(rk)\chi(sk) \, dk,
\\
\label{HEatprKer3}
K_3(t,r,s) & :=  \frac{1}{\sqrt{rs}} 
  \int_0^\infty \Phi_2(r,k^2) \Phi_2(s,k^2) e^{-tk^2} \, k \rho_2^\prime(k^2) \, 
  \chi^c(rk)\chi(sk) \, dk.
\end{align}

\medskip
\noindent
{\it Estimate of $K_{1}$.}
On the support of $K_{1}$ we have $rk \geq sk \gtrsim 1$, and
using \eqref{HeatprPhi} we reduce matters to estimating the kernels 
\begin{align*}
\begin{split}
K_{1,\eps_1\eps_2}(t,r,s) := 
  \int_0^\infty \frac{1}{\sqrt{rk}} \frac{1}{\sqrt{sk}} a_2^{\eps_1}(k^2) e^{\eps_1 irk} \big( 1 + \phi_{\eps_1}(r,k))
  a_2^{\eps_2}(k^2) e^{\eps_2 isk} \big( 1 + \phi_{\eps_2}(s,k))  
  \\
  \times e^{itk^2} \, k \rho_2^\prime(k^2) \, \chi^c(sk) \, dk, \qquad \eps_1,\eps_2 \in \{+,-\}.
\end{split}
\end{align*}
It suffices to look at the case $(\eps_1,\eps_2) = (-,+)$; the other cases are similar or easier.
We write 
\begin{align}\label{HeatprK1}
\begin{split}
K_{1,-+}(t,r,s) & := \int_0^\infty \frac{1}{\sqrt{rk}} \frac{1}{\sqrt{sk}} e^{ik(s-r)} e^{-tk^2} m(k) A(r,s,k) \, k \, dk,
\\
m(k) & := |a_2(k^2)|^2 \rho_2^\prime(k^2),
\\
A(r,s,k) & := \big( 1 + \phi_-(r,k)) \big( 1 + \phi_{+}(s,k)) \chi^c(sk).
\end{split}
\end{align}
In view of \eqref{propSMrho} and \eqref{prKG2phiest}, we have the symbol type estimates
\begin{align}\label{prHeat15'}
|\partial^\alpha_k m| \lesssim \jk^{-\alpha},
\qquad |\partial_k^\alpha A(r,s,k)| \lesssim k^{-\alpha}, \qquad \alpha = 0,1,2.
\end{align}

Notice that the contribution to the kernel when $k \lesssim t^{-1/2}$ is pointwise bounded by
\begin{align*}
\begin{split}
C \int_0^{Ct^{-1/2}} e^{-tk^2} k \, dk \lesssim t^{-1}.
\end{split}
\end{align*} 
Therefore, we can restrict our attention to the contribution 
from $k \gtrsim t^{-1/2}$ inserting a cutoff $\chi^c(\sqrt{t}k)$.
In the formula for the kernel we integrate by parts twice using 
$$\partial_k \big[ e^{ik(s-r) - tk^2} \big] = [ i(s-r) - 2tk] e^{ik(s-r) -tk^2},$$
and obtain the term
\begin{align}\label{Heatprbound1}
\begin{split}
I(t,r,s) & := \int_0^{\infty} e^{-tk^2} \partial_k \Big[ 
  \frac{1}{tk + i(r-s)}  
  \\
  & \times \partial_k \Big( \frac{1}{tk + i (r-s)} \, k \frac{\chi^c(sk)}{\sqrt{rk} \sqrt{sk}} 
  A(r,s,k) m(k) \chi^c(\sqrt{t}k) \Big) \Big] \, dk.
\end{split}
\end{align} 
Since $kt \gtrsim 1/k$, using \eqref{prHeat15'}, we can bound 
\begin{align*}
\begin{split}
| I(t,r,s) | \lesssim \int_{k \gtrsim \max(t^{-1/2}, s^{-1})} e^{-tk^2} 
  \frac{(kt)^2}{(tk)^2 + (r-s)^2} \, 
  \, k dk
\lesssim \int_0^\infty e^{-tk^2}  \, k \, dk \lesssim t^{-1},
\end{split}
\end{align*} 
which is the desired bound \eqref{Heatprmain1}. 


\medskip
\noindent
{\it Estimate of $K_{2}$.}
In view of \eqref{HeatprPhi}, we can reduce matters to estimating the expression
\begin{align*}
L(t,r,s) & :=  
  \int_0^\infty \frac{r}{\jr} \frac{s}{\langle s \rangle}  e^{-tk^2} \, 
  \big( 1 + \phi(r,k)) \big( 1 + \phi(s,k)) \chi(rk)\chi(sk) \jk^2 k \, dk,
\end{align*}
for $\phi$ satisfying \eqref{Heatprphiest}.
The case $k \lesssim t^{-1/2}$ is easy to deal with, so we restrict our attention to the complementary region;
there we can integrate by parts the exponential and bound
\begin{align}\label{Heatprbound5}
| L(t,r,s) | \lesssim
  \int_{k \gtrsim t^{-1/2}} \frac{r}{\jr} \frac{s}{\langle s \rangle} e^{-tk^2} \frac{1}{(tk^2)^2} \, 
  \chi(rk)\chi(sk) \jk^2 k \, dk.
\end{align}
This immediately implies the desired pointwise bound as in \eqref{Heatprmain1}:
\begin{align*}
| L(t,r,s) | \lesssim
  \int_{k \gtrsim t^{-1/2}} e^{-tk^2} k \, dk \lesssim t^{-1}.
\end{align*}

\medskip
\noindent
{\it Estimate of $K_{3}$.}
This estimate is a combination of the previous two;
by similar reductions, it suffice to estimate a term of the form 
\begin{align}\label{HeatprKer3'}
  \int_0^\infty e^{irk} \frac{1}{\sqrt{rk}} \big( 1 + \phi(r,k))
  \frac{s}{\langle s \rangle} \big( 1 + \phi(s,k)) e^{-tk^2} \, k \jk \, 
  \chi^c(rk)\chi(sk) \chi^c(\sqrt{t}k) \, dk,
\end{align}
with the same notation above for $\phi$ and $\chi$.
Integrating by parts similarly to \eqref{Heatprbound5} (by letting $s=0$ in the phase) 
produces a main term which is bounded by
\begin{align*}
\begin{split}
M(t,r,s) & := \int_0^{\infty} e^{-tk^2} \partial_k \Big[ 
  \frac{1}{tk + ir}  
   \partial_k \Big( \frac{1}{tk + ir} \, k \frac{\chi(sk) \chi^c(rk)}{\sqrt{rk}} 
  \chi^c(\sqrt{t}k) \Big) \Big] \, dk.
\end{split}
\end{align*} 
Since $kt \gtrsim 1/k$, 
we obtain
\begin{align*}
\begin{split}
M(t,r,s) \lesssim \int_{k \gtrsim t^{-1/2}} e^{-tk^2} 
  \frac{(kt)^2}{(tk)^2 + r^2} \, 
  \, k dk
\lesssim \int_0^\infty e^{-tk^2}  \, k \, dk \lesssim t^{-1}.
\end{split}
\end{align*}

\smallskip
\noindent
{\it Proof of \eqref{Heatconc1'}-\eqref{Heatconc2'}}.
First, we observe that integration by parts arguments similar to those above, and 
to those in the proof of Lemma \ref{LPlem}, where the additional localization in $k\approx 2^\ell$ is included,
show that
\begin{align}\label{Hcorpr0}
{\big\| e^{-t\mathcal{L}_2} P^{\mathcal{L}_2}_\ell f \big\|}_{L^\infty(\R^2)} \lesssim {\| f \|}_{L^\infty}
\end{align}
for all $\ell \in \Z$.
Then, interpolating \eqref{Hcorpr0} with the $L^1 - L^\infty$ estimate \eqref{Heatconc2}
we get
\begin{align}\label{Hcorpr1}
{\big\| e^{-t\mathcal{L}_2} P^{\mathcal{L}_2}_\ell f \big\|}_{L^\infty(\R^2)} \lesssim 
  t^{-1/q} {\| f \|}_{L^q(rdr)}.
\end{align}
Interpolating the basic $L^2 - L^2$ bound with \eqref{Hcorpr0}, 
and using duality, we also get 
\begin{align}\label{Hcorprq}
{\big\| e^{-t\mathcal{L}_2} P^{\mathcal{L}_2}_\ell f \big\|}_{L^q(\R^2)} \lesssim {\| f \|}_{L^q(rdr)}
\end{align}
for $q\in[1,\infty]$.
Interpolating this latter and \eqref{Hcorpr1} we get the claimed $L^q - L^p$ estimate \eqref{Heatconc1'}.
The proof of \eqref{Heatconc2'} is identical.
\end{proof}


\smallskip
\subsection{Estimates for Fourier norms}\label{SsecFouPhys}

Here is a result on bounds for $L^\infty$ and $L^2$-type Fourier norms by (weighted) norms in physical space.

\begin{lem}\label{WDecKG4}
With the definition \eqref{WDecKG2} for $F$, we have the estimates
\begin{align}\label{WDecKG5}
{\| F \|}_{L^\infty_k} \lesssim {\| f \|}_{L^1(rdr)},
\end{align}
and
\begin{align}\label{WDecKG6} 
\begin{split}
{\| \partial_k F \|}_{L^2(dk)} & \lesssim {\| \jr^{3/2+} f \|}_{L^2(rdr)}, 
\qquad {\| \partial_k^2 F \|}_{L^2(dk)}  \lesssim {\| \jr^{5/2+} f \|}_{L^2(rdr)}. 
\end{split}
\end{align}
Moreover, we have the bounds 
\begin{align}\label{WDecKG24}
\begin{split}
\Vert F\Vert_{L^2(dk)}&\lesssim  \Vert \langle r\rangle^{1/2+}f\Vert_{L^2(rdr)},
\\ 
\Vert \langle k\rangle F\Vert_{L^2(dk)}&\lesssim  \Vert r^{-1}\langle r\rangle^{3/2+}f\Vert_{L^2(rdr)} 
  + \Vert r^{1/2-} \langle r\rangle^{0+}\partial_rf\Vert_{L^2 (rdr)},
\\  
\Vert \langle k\rangle^2F\Vert_{L^2(dk)}&\lesssim \Vert r^{-2}\langle r\rangle^{5/2+}f\Vert_{L^2(rdr)} 
  + \Vert r^{-1/2-}\langle r\rangle^{0+}\partial_rf\Vert_{L^2 (rdr)}
  + \Vert r^{1/2-}\langle r\rangle^{0+}\partial_r^2f\Vert_{ L^2(rdr)}.
\end{split}
\end{align}
Finally, the following inequalities hold:
\begin{align}\label{WDecKG20}
\Vert (k\partial_k)F\Vert_{L^2(dk)}&\lesssim \Vert \langle r\rangle^{1/2+}f\Vert_{L^2(rdr)} 
  + \Vert \langle r\rangle^{3/2+}f'\Vert_{L^2(rdr)},
\\ 
\label{WDecKG23}
\Vert (k\partial_k)^2F\Vert_{L^2(dk)}&\lesssim \Vert \langle r\rangle^{1/2+}f\Vert_{L^2(rdr)} 
  + \Vert \langle r\rangle^{3/2+}f'\Vert_{L^2(rdr)}+\Vert \langle r\rangle^{5/2+}f''\Vert_{L^2(rdr)}.
\end{align}
\end{lem}

\smallskip
Estimates of the type \eqref{WDecKG20}-\eqref{WDecKG23} convert the
operator $k\partial_k$ on the Fourier side to $r\partial_r$ on the
physical side; due to our definition of $F$ an extra power of
$\jr^{1/2}$ appears on the right-hand side. These estimates are related to the so-called ``transference identities'' in \cite{KST,KMS}.
Since the proof of Lemma \ref{WDecKG4} is somewhat standard, we will only give the main details
of its proof.

\smallskip
\begin{proof}[Proof of Lemma \ref{WDecKG4}]
As before, we decompose $F$ as
\begin{align*}
F = F^{(1)} + F^{(2)},  \qquad & F^{(1)}(k) := \frac{1}{a_1(k^2)} \int_0^\infty \chi(rk) \Phi_1(r,k^2) \sqrt{r} f(r) \, dr,
\end{align*}
and, according to the asymptotics \eqref{prKG1Phi}, we can reduce the analysis of 
$F^{(1)}$ and $F^{(2)}$ to estimating the integrals (respectively)
\begin{align}
A(k) & := \frac{1}{a_1(k^2)} \int_0^\infty \chi(rk) \frac{r^{3/2}}{\jr} (1+\phi(r,k)) \sqrt{r} f(r) \, dr, \label{WDecKG21}
\\
B(k) & := 
  \int_0^\infty \chi^c(rk) \dfrac{1}{\sqrt{k}} 
  e^{irk} \big( 1 + \phi(r,k)) \sqrt{r} f(r) \, dr,\label{WDecKG22}
\end{align}
where $\phi$ is a generic function satisfying the symbol-type estimates
\begin{align}\label{WDecKG12}
| (k\partial_k)^\alpha (r\partial_r)^\beta \phi(r,k) | \lesssim 1, \qquad \alpha,\beta = 0,1,2.
\end{align}
For \eqref{WDecKG5}-\eqref{WDecKG6} we then want to show that 
\begin{align}\label{WDecKG13}
{\| A \|}_{L^\infty_k} + {\| B \|}_{L^\infty_k}  \lesssim {\| f \|}_{L^1(rdr)},
\end{align}
and
\begin{align}\label{WDecKG14}
{\| \partial_k A \|}_{L^2(dk)} +  {\| \partial_k B \|}_{L^2(dk)} \lesssim {\| \jr^{3/2+} f \|}_{L^2(rdr)},
\\
\label{WDecKG15}
{\| \partial_k^2 A \|}_{L^2(dk)} +  {\| \partial_k^2 B \|}_{L^2(dk)} \lesssim {\| \jr^{5/2+} f \|}_{L^2(rdr)}.
\end{align}

\smallskip
\noindent
{\it Proof of \eqref{WDecKG13}.}
Recall that $a(k^2) \approx \jk^{-1}$, so that, 
\begin{align*}
| A(k) | + | B(k) |  
  & \lesssim \int_{r \lesssim 1/k} \jk \frac{r^{3/2}}{\jr} \sqrt{r} |f(r)| \, dr
  +  \int_{r \gtrsim 1/k} \dfrac{1}{\sqrt{k}} \sqrt{r} |f(r)| \, dr \lesssim \int_0^\infty r\vert f(r)\vert \, dr,
\end{align*}
having used that $r\langle r\rangle^{-1}\jk\lesssim 1$ on the support of the integral. 

\smallskip
\noindent
{\it Proof of \eqref{WDecKG14}-\eqref{WDecKG15}.}
For the first term we have $|\partial_k A| \lesssim | A_1 | + | A_2 | + | A_3 |$ with
\begin{align}\label{WDecKG16}
\begin{split}
\\
A_1(k) &:= \frac{\partial_k a(k^2)}{a^2(k^2)} \int_0^\infty \chi(rk) \frac{r^{3/2}}{\jr} \sqrt{r} f(r) \, dr,
\\
A_2(k) &:= \frac{1}{a(k^2)} \int_0^\infty \chi'(rk) r \, \frac{r^{3/2}}{\jr} \sqrt{r} f(r) \, dr,
\\
A_3(k) &:= \frac{1}{a(k^2)} \int_0^\infty \chi(rk) \, \frac{r^{3/2}}{\jr} \partial_k \phi(r,k) \, \sqrt{r} f(r) \, dr,
\end{split}
\end{align}
Using \eqref{propSMa} which gives $|a|^{-2} | \partial_k a | \lesssim 
1$, we integrate in $k$ whence it follows that
\begin{align*}
{\| A_1 \|}_{L^2_k} & \lesssim 
\int_0^\infty  {\| \chi(rk) \|}_{L^2_k} \frac{r^2}{\jr} |f(r)| \, dr
	\lesssim \Big( \int_0^\infty  \jr^{1+} |f(r)|^2 \, r dr \Big)^{1/2}.
\end{align*}
We can similarly bound
\begin{align*}
{\| A_2 \|}_{L^2_k}  & \lesssim \int_0^\infty  {\| \jk \chi'(rk) \|}_{L^2_k} \frac{r^3}{\jr} |f(r)|  dr
	\lesssim \Big( \int_0^\infty  \jr^{3+} |f(r)|^2 \,r dr \Big)^{1/2},
\end{align*}
having used $\Vert \jk \chi'(rk)\Vert_{L^2_k}\lesssim \langle r\rangle r^{-3/2}$. 
Finally, for the third term, we first bound
\begin{align*}
{\| A_3 \|}_{L^2_k}  & \lesssim  \left\Vert \int_0^\infty  \jk \chi(rk)  \frac{r^2}{\jr} \vert \partial_k\phi(r,k)\vert \, |f(r)| \,  dr\right\Vert_{L^2_k} .
\end{align*}
Then, using $rk\lesssim1$, an inspection of Lemma \ref{lemFou1} which shows that $\vert \partial_k\phi(r,k)\vert \lesssim r$,
and applying Cauchy-Schwarz it follows that 
\[
 \Vert A_3\Vert_{L^2_k}\lesssim \int_0^\infty r^{3/2}\vert f(r)\vert dr
  \lesssim \Big( \int_0^\infty  \jr^{3+} |f(r)|^2 \,r dr \Big)^{1/2}.
\]

\smallskip
To prove  \eqref{WDecKG14} for the integral $B$, we first bound 
$|\partial_k B| \lesssim | B_1 | + | B_2 | + | B_3 |,$ 
with
\begin{align}\label{WDecKG17}
\begin{split}
& B_1(k) := \int_0^\infty \partial_k \big[ \chi^c(rk) \dfrac{1}{\sqrt{k}} \big] e^{irk} \big( 1 + \phi(r,k)) \sqrt{r} f(r) \, dr,
\\
& B_2(k) := \int_0^\infty \chi^c(rk) \dfrac{1}{\sqrt{k}} \, ir e^{irk} \big( 1 + \phi(r,k)) \sqrt{r} f(r) \, dr,
\\
& B_3(k) :=\int_0^\infty \chi^c(rk) \dfrac{1}{\sqrt{k}} e^{irk} \, \partial_k \phi(r,k) \sqrt{r} f(r) \, dr.
\end{split}
\end{align}
To bound these terms in $L^2$ we use the following standard boundedness property of pseudo-differential operators:
let $m(r,k)$ be an operator supported on $r \approx 2^\ell$ and such that 
\begin{align}\label{PDOcrit}
\sup_{r} \sup_{l=0,1,2}|(r \partial_r)^l m(r,k)| \leq M,
\end{align}
then the PDO with symbol $m$ is bounded by $C M$ as an operator from $L^2$ to $L^2$.
We can then write 
\begin{align*}
B_1(k) & = \int_0^\infty m_1(r,k) e^{irk} r^2 f(r) \, dr, 
\qquad m_1(r,k) := \partial_k \big[ \chi^c(rk) (rk)^{-1/2} r^{-1} \big] \big( 1 + \phi(r,k))
\end{align*}
so that, using the symbol properties from Lemma \ref{lemWeyl}, 
and applying the above criterion we immediately get
\begin{align*}
{\| B_1(k) \|}_{L^2} \lesssim \sum_{\ell \in \Z} {\| r^2 \varphi_\ell \, f \|}_{L^2(dr)} = {\| \jr^{3/2+} f \|}_{L^2(rdr)}.
\end{align*}
It is easy to see that a similar argument applies almost verbatim to $B_2$ and $B_3$ by writing
\begin{align*}
B_2(k) & = \int_0^\infty m_2(r,k) e^{irk} r^2 f(r) \, dr, 
\qquad m_2(r,k) := \chi^c(rk) (rk)^{-1/2} \big( 1 + \phi(r,k)),
\\
B_3(k) & = \int_0^\infty m_3(r,k) e^{irk} r^2 f(r) \, dr, 
\qquad m_3(r,k) := \chi^c(rk) (rk)^{-1/2} r^{-1} \partial_k \phi(r,k).
\end{align*}
%

It is not difficult to see that we can use similar arguments to bound the second derivatives.
For the term $\partial_k^2 A$, we notice that, with the notation in \eqref{WDecKG16},
we can bound $|\partial_k^2 A| \lesssim | \partial_kA_1 | + | \partial_kA_2 | + | \partial_kA_3 |$. 
Then, we observe that an application of $\partial_k$ 
to the three terms in \eqref{WDecKG16} will cost exactly a factor of $r$, 
using also that $|\partial_k^\alpha \phi | \lesssim r^\alpha$, $\alpha =0,1,2$.
Similarly, with the same notation in \eqref{WDecKG17},
we have $|\partial_k^2 B| \lesssim | \partial_k B_1 | + | \partial_k B_2 | + | \partial_k B_3 |$,
and each application of $\partial_k$ to the terms in \eqref{WDecKG17} costs either a factor of $r$,
which is consistent with the desired bound,
or a factor of $1/k$, but this is essentially equivalent to a factor of $r$ in view of the cutoff $\chi^c(rk)$.

\medskip
\noindent
{\it Proof of \eqref{WDecKG24}.} 
Let $m\in\{0,1,2\}$. We use again the decomposition in \eqref{WDecKG21} and \eqref{WDecKG22}. 
In the case of $A(k)$, using that $\Vert \langle k\rangle^m\chi(rk)\Vert_{L^2_k}\lesssim r^{-m-1/2}\langle r\rangle^{m}$,
we have 
\begin{align*}
\Vert \jk^m A(k)\Vert_{L^2_k} \lesssim \int_0^\infty {\| \langle k\rangle^m \chi(rk) \|}_{L^2_k} 
  |f(r)| rdr \lesssim \Vert r^{-m}\langle r\rangle^{m+1/2+}f \Vert_{L^2(rdr)}.
\end{align*}

To bound the term $B(k)$ it suffices to look at $k \gg 1$, 
so that we can replace $\langle k\rangle$ for $k$. 
We also suppose that $m=2$, the case $m=1$ being similar.
Integrating by parts we obtain
\begin{align*}
& \Big| \int_0^\infty \chi^c(rk) \dfrac{k^2}{\sqrt{k}} e^{irk} \big( 1 + \phi(r,k)) \sqrt{r} f(r) \, dr \Big|
\lesssim |\widetilde{B}_1|+ |\widetilde{B}_2| + |\widetilde{B}_3|,
\\
& \widetilde{B}_1 := \int_0^\infty  \dfrac{1}{\sqrt{k}} e^{irk} \partial_r^2\Big(\chi^c(rk) \chi^c(ck)\big( 1 + \phi(r,k)) \sqrt{r} \Big)\,f(r) \, dr,
\\ 
& \widetilde{B}_2 :=\int_0^\infty  \dfrac{1}{\sqrt{k}} e^{irk} \partial_r\Big(\chi^c(rk) \chi^c(ck)\big( 1 + \phi(r,k)) \sqrt{r} \Big)\,\partial_rf(r) \, dr,
\\ 
& \widetilde{B}_3 := \int_0^\infty  \dfrac{1}{\sqrt{k}} e^{irk} \chi^c(rk) \chi^c(ck)\big( 1 + \phi(r,k)) \sqrt{r} \partial_r^2f(r) \, dr.
\end{align*}
The above integrals can conveniently be written as
\[
\widetilde{B}_j := \int_0^\infty \widetilde{m}_i(r,k)e^{irk} r^{j-2} \partial_r^{j-1} f(r) \, dr
\]
with the natural definition of the symbols $\widetilde{m}_j$.
%
%
To conclude, it suffices to notice that the symbols $\wt{m}_j$, $j=1,\dots,3$, 
satisfy \eqref{PDOcrit}, in view of
$|(r \partial_r)^\alpha \phi | \lesssim 1$, see \eqref{WDecKG12}; 
then applying the same $L^2$-boundedness criterion for PDOs used before, we obtain
\begin{align*}
\Vert \widetilde{B}_j \Vert_{L^2_k} & \lesssim \Vert r^{j-5/2-}\langle r \rangle^{0+} \partial_r^{j-1} f \Vert_{L^2(rdr)}.
%
\end{align*}
Combining these with the bound for the term $A$ we obtain the claimed \eqref{WDecKG24}.

\medskip
\noindent
{\it Proof of \eqref{WDecKG20}-\eqref{WDecKG23}.} 
The proof is almost identical to that of the previous case. Using that $rk\lesssim 1$ 
in the support of $\partial_kA$ in \eqref{WDecKG16}, 
bounding as we did for $\Vert \partial_kA\Vert_{L^2_k}$ we see that 
\[
\Vert k\partial_kA\Vert_{L^2_k} \lesssim \Big(\int_0^\infty \langle r\rangle^{0+} \vert f(r)\vert^2 \,rdr \Big)^{1/2}.
\] 
For the other term we first bound $|k\partial_kB| \lesssim |kB_1| + |kB_2| + |kB_3|$, see \eqref{WDecKG17}.
We then write
\begin{align*}
kB_j(k) & = \int_0^\infty q_j(r,k) e^{irk} r f(r) \, dr, \qquad j\in\{1,3\},
\end{align*}
with the natural definition of the symbols $q_1$ and $q_3$;
it is easy to verify that these satisfy the symbol-type bounds \eqref{PDOcrit},
and hence, 
$\Vert kB_1\Vert_{L^2_k}+\Vert kB_2\Vert_{L^2_k}\lesssim \Vert \langle r\rangle^{1/2+}f\Vert_{L^2(rdr)}.$
Regarding $kB_2$, integration by parts in $r$ leads to $|kB_2(k)|\lesssim |\widetilde{B}_{2,1}| + |\widetilde{B}_{2,2}|$,
where
\begin{align*}
\widetilde{B}_{2,1} & := \int_0^\infty \widetilde{m}_{2,1}(r,k) e^{irk} r f(r) \, dr, 
 & & \widetilde{m}_{2,1}(r,k) :=\dfrac{1}{r\sqrt{k}} \partial_r\Big[\chi^c(rk) ir  \big( 1 + \phi(r,k)) \sqrt{r}\Big],
\\
\widetilde{B}_{2,2} & := \int_0^\infty \widetilde{m}_{2,2}(r,k) e^{irk} r^2 f'(r) \, dr, 
 & & \widetilde{m}_{2,2}(r,k) := \chi^c(rk) (rk)^{-1/2} \big( 1 + \phi(r,k)).
\end{align*}
Since the symbols $\widetilde{m}_{2,1}$ and $\widetilde{m}_{2,2}$ also satisfy the bound \eqref{PDOcrit}
we get
$\Vert kB_2\Vert_{L^2_k}\lesssim \Vert \langle r\rangle^{1/2+}f\Vert_{L^2(rdr)}+\Vert \langle r\rangle^{3/2+}f'\Vert_{L^2(rdr)},$
hence concluding the proof of \eqref{WDecKG20}.

\smallskip
Finally, it is not difficult to see that the proof of \eqref{WDecKG23} follows the same lines up to obvious modifications.
If both derivatives $\partial_k$ hit the exponential $e^{irk}$, then we integrate by parts in $r$ twice. 
If only one derivative $\partial_k$ hits the exponential, then we integrate by part once.
Note that whenever the derivative acts on any term other than the exponential, 
we lose exactly a factor $k^{-1}$, and hence the operator $(k\partial_k)$ yields the same estimates as before in those cases. 
The proof of the lemma is complete. 
\end{proof}



\smallskip
\subsection{Littlewood-Paley projections}\label{ssecLP}
In this subsection we prove that the Littlewood-Paley projections associated to
the operators $\mathcal{L}_j$, $j=1,2$, as defined in \eqref{LPF1L} and \eqref{LPF1'}
(see also the notation from Subsection \ref{secnot}) are bounded on Lebesgue spaces.
Since this is somewhat standard, we will only provide a sketch of the proof below.

\begin{lem}\label{LPlem}
Consider, for $j=1,2$, the operators
\begin{align}\label{LPF1Lap}
\begin{split}
& P_\ell^{\mathcal{L}_j} f (r) 
:= \int_0^\infty \frac{1}{\sqrt{rs}} K_{j,\ell}(r,s) f(s) \, s  ds,
\\
& K_{j,\ell}(r,s) := \int_0^\infty \Phi_j(r,k^2) \Phi_j(s,k^2)  \, \varphi_\ell(k) \, 2k \rho_j'(k^2) \, dk,
\end{split}
\end{align}
where $\Phi_j$ are the generalized eigenfunctions associated to 
the conjugated operator $\mathcal{H}_j = r^{1/2} \mathcal{L}_j r^{-1/2}$.
Then
\begin{align}\label{LPlemconc}
{\big\| P_\ell^{\mathcal{L}_j} f (r) \big\|}_{L^p(rdr)} \lesssim {\| f \|}_{L^p(rdr)}, \qquad 1\leq p \leq \infty.
\end{align}

\end{lem}

\begin{proof}
We restrict our attention to the case $j=2$ and will drop the index $j$ from the notation.
The case $j=1$ is identical up to obvious modifications; in particular, the extra
logarithmic terms in $\Phi_1$ (cfr. Lemmas \ref{lemFou1} and \ref{Phi1Op1_1}) do not play any role.
As usual, we will employ the expansions for $\Phi_2$ from Lemmas \ref{lemFou1} and \ref{lemWeyl}
to estimate the kernel by distinguishing three different cases, depending on the sizes of $rk$ and $sk$.

\smallskip
\noindent
{\it Case 1: $rk,sk \gtrsim 1$}.
According to \eqref{LPF1Lap}, plugging in the asymptotics in Lemma \ref{lemWeyl} for the 
Weyl solutions at infinity, 
we can reduce matters to considering a kernel of the form
\begin{align*}
L_\ell(r,s) := \frac{1}{\sqrt{rs}} \int_0^\infty e^{i(r\pm s) k} 
  \chi^c(rk) \chi^c(sk) (1+\phi(r,k))(1+\phi(s,k)) \, \varphi_\ell(k) 
  \, dk
\end{align*}
where $\phi$ denotes a generic function with the symbol property
\begin{align}\label{LPlem2}
|\partial_k^\alpha \phi(r,k)| \lesssim k^{-\alpha}, \quad \alpha = 0,1,2,
\end{align}
(see \eqref{prKG1phi}), and obtaining the bound 
\begin{align}\label{LPlem1}
|L_\ell(r,s) | \lesssim \frac{2^{2\ell}}{(1+ 2^{2\ell}|r\pm s|^2)}.
\end{align}
The desired estimate by the right-hand side of \eqref{LPlemconc}
would then follow from Young's inequality.
It is not hard to see that \eqref{LPlem1} holds true integrating by parts in $k$
using $(1+2^{2\ell}|r\pm s|^2) e^{i(r\pm s) k} = (1-2^{2\ell}\partial_k^2) e^{i(r\pm s) k}$,
and the fact that $(rs)^{-1/2} \lesssim k \approx 2^\ell$.

\smallskip
\noindent
{\it Case 2: $rk \lesssim 1, sk \lesssim 1$}.
In this case we use Lemma \ref{lemFou1} to expand $\Phi$,
and Proposition \ref{propSM0} for the spectral measure.
Then, similarly to the proofs in Section \ref{secdecay}, we can reduce matters to bounding in $L^p(rdr)$
the quantity
\begin{align*}
\int_0^\infty M_\ell(r,s) \, f(s) s ds
\end{align*}
where the kernel is of the form
\begin{align*}
M_\ell(r,s) :=  \int_0^\infty
  \chi(rk) \chi(sk) \frac{r}{\jr} \frac{s}{\langle s \rangle} (1+\phi(r,k)) (1+\phi(s,k)) \, \varphi_\ell(k) 
  k \jk^2 \, dk,
\end{align*}
with $\phi$ a generic function with the same symbol property \eqref{LPlem2} above.
To obtain an estimate consistent with \eqref{LPlemconc} via Schur's test it suffices to show
\begin{align*}
\sup_{s\geq 0} \int_0^\infty |M_\ell(r,s)| \, r dr + \sup_{r \geq 0} \int_0^\infty |M_\ell(r,s)| \, s ds \lesssim 1.
\end{align*}
By symmetry, it suffices to bound the first of the two quantities above.
Since $k\approx 2^\ell$ and $rk \lesssim 1$ on the support of the integral, 
using also that $x \jk/\langle x \rangle \lesssim 1$ for $xk \lesssim 1$, we have
\begin{align*}
\int_0^\infty |M_\ell(r,s)| \, r dr \lesssim \int_{r \lesssim 2^{-\ell}}  
  \int_{k \approx 2^\ell} k dk \, r dr \lesssim 1.
\end{align*}
This gives a sufficient bound in this case.

\smallskip
\noindent
{\it Case 3: $sk \lesssim 1 \lesssim rk$}.
This case is a combination of the previous two.
We use Lemma \ref{lemWeyl} for $\Phi(r,k)$ and Lemma \ref{lemFou1} for $\Phi(s,k)$,
as well as Proposition \ref{propSM2} for the spectral measure and the coefficient $a(k)$.
Similarly to the proofs in Section \ref{secdecay} and the one above, 
we can then reduce matters to bounding 
\begin{align*}
\int_0^\infty N_\ell(r,s) \, f(s) s ds
\end{align*}
where the kernel has of the form
\begin{align*}
N_\ell(r,s) := \frac{1}{\sqrt{rk}} \int_0^\infty e^{i r k} 
  \chi^c(rk) \chi(sk) (1+\phi(r,k)) \frac{s}{\langle s \rangle} (1+\phi(s,k)) \, \varphi_\ell(k) 
  \, k \jk \, dk,
  \end{align*}
and $\phi$ denotes a generic function with the property \eqref{LPlem2}.
To obtain an estimate consistent with \eqref{LPlemconc}, it then suffices to 
apply Schur's test with the two inequalities
\begin{align}\label{LPlem3}
\sup_{s\geq 0} |N_\ell(r,s)| \lesssim \frac{2^{2\ell}}{(1+ 2^{2\ell}r^2)},
\qquad \sup_{r \geq 0} \int_0^\infty |N_\ell(r,s)| \, s ds \lesssim 1.
\end{align}
The first inequality can be verified using integration by parts in $k$
via the identity $(1+2^{2\ell}r^2) e^{ir k} = (1-2^{2\ell}\partial_k^2) e^{ir k}$, similarly to Case 1 above,
and that $rk\gtrsim 1$ and $s \jk/\langle s \rangle \lesssim 1$ on the support of the integral;
we skip the details.
For the second inequality in \eqref{LPlem3} we can bound, similarly to Case 2 above,
\begin{align*}
\int_0^\infty |N_\ell(r,s)| \, s ds \lesssim \int_{s \lesssim 2^{-\ell}} \int_{k\approx 2^\ell}   \,k dk\, s ds
\lesssim 1.
\end{align*}
This concludes the proof of the lemma.
\end{proof}



\end{document}